\PassOptionsToPackage{pdftex}{graphicx}
\documentclass[10pt,leqno]{amsart}
\usepackage{graphicx} 

\DeclareUnicodeCharacter{0302}{\^{}}

\usepackage{eucal}
\usepackage[utf8]{inputenc}
\usepackage[english]{babel}
\usepackage{microtype}
\usepackage{mathrsfs}
\usepackage{amsthm, thmtools}
\usepackage{amsmath, centernot}
\usepackage{amssymb}
\usepackage{indentfirst,csquotes}
\usepackage{amscd}
\usepackage{mathtools}
\usepackage{wasysym}
\usepackage{braket}
\usepackage{hyperref}
\usepackage{appendix}
\usepackage{dsfont}
\usepackage{enumitem}
\usepackage{framed}
\usepackage{caption}
\usepackage{subfig}
\usepackage[all]{xy}
\usepackage{xfrac}
\usepackage{hyperref}
\usepackage{color}
\usepackage{bbm}
\usepackage{stmaryrd}
\usepackage{fancyhdr}
\usepackage{tikz-cd}
\usepackage{graphicx}
\usepackage{nomencl}
\usepackage{ esint }
\usepackage{tikz}
\usepackage{multicol}

\DeclareMathAlphabet{\pazocal}{OMS}{zplm}{m}{n}

\numberwithin{equation}{section}

\allowdisplaybreaks

\newtheorem{teorema}{Theorem}[section]
\newtheorem{prop}[teorema]{Proposition}
\newtheorem{co}[teorema]{Corollary}
\newtheorem{lemma}[teorema]{Lemma}

\newtheorem{es}[teorema]{Example}

\newtheorem{df}[teorema]{Definition}

\newtheorem{oss}[teorema]{Remark}

\renewcommand{\L}{\mathcal{L}}

\newcommand{\R}{\mathbb{R}}

\newcommand{\N}{\mathbb{N}}

\newcommand{\PP}{\mathcal{P}}

\newcommand{\m}{\mathbf{m}}

\newcommand{\CE}{\operatorname{CE}}

\renewcommand{\d}{\mathrm{d}}

\usepackage[backend=biber,style=alphabetic,doi=true,isbn=true, maxalphanames=4, 
  minalphanames=4,maxbibnames=99, url = false]{biblatex}
\title[Superposition for random measures and geometry of the WoW space]{Nested superposition principle for  random measures and the geometry of the Wasserstein on Wasserstein space}
\author{Alessandro Pinzi and Giuseppe Savar\'e}
\date{\today}

\begin{document}

\begin{abstract}
    We study the geometric structure of the space of random measures \(\mathcal{P}_p(\mathcal{P}_p(X))\), endowed with the Wasserstein 
    on Wasserstein metric, where \((X, d)\) is a complete separable metric space. In this setting, we prove a metric superposition principle, in the spirit of \cite{lisini2007characterization}, that will allow us to recover important geometric features of the space.
    \\
    When $X$ is $\R^d$, we study the 
    differential structure of $\PP_p(\PP_p(\R^d))$ in analogy with the simpler Wasserstein space $\PP_p(\R^d)$. We 
    show that 
    continuity equations for random measures 
    involving the abstract concept of derivation acting on cylinder functions 
    can be more conveniently described by 
    suitable non-local vector fields $b:[0,T]\times \R^d \times \PP(\R^d) \to \R^d$. In this way, we can 
    \begin{itemize}
        \item characterize the absolutely continuous curves on the Wasserstein on Wasserstein space;
        \item define and characterize its tangent bundle;
        \item prove a superposition principle for the solutions to the standard non-local continuity equation in terms of solutions of interacting particle systems.
    \end{itemize}
\end{abstract}

\maketitle

{\small
		.
		\par
	}

\tableofcontents
\newcommand{\GGG}{\color{blue}}
\newcommand{\nc}{\normalcolor}
\section{Introduction}

The study of measure-valued solutions to the continuity equation in Euclidean spaces, 
\begin{equation}\label{eq: 1}
    \partial_t \mu_t + \operatorname{div}(v_t\mu_t) = 0, \quad v:[0,T]\times \R^d \to \R^d, \ \mu_t\in \PP(\R^d), 
\end{equation}
has become of central interest in the last decades, in particular for its connections to optimal transport \cite{benamou2000computational} and to Wasserstein gradient flows \cite{jordan1998variational}, \cite{ambrosio2005gradient}. A crucial role is played by the so-called superposition principle \cite{ambrosio2004transport} (see also \cite{ambrosio2005gradient}, \cite{ambrosio2014continuity}), that represents solutions of the continuity equation as time marginals of a probability measure concentrated on the associated characteristics system of ODEs, that is 
\begin{equation}\label{eq: 2}
    \lambda \in \PP(C([0,T],\R^d)) \ \text{ s.t. } \ \lambda\text{-a.e. } \gamma \text{ is in } AC([0,T],\R^d) \text{ and solves }\dot{\gamma}(t) = v_t(\gamma(t)).
\end{equation}
It was first proved by L. Ambrosio for studying well-posedness of a Lagrangian system for ODE under non-smooth assumptions of the vector field, started by the seminal work \cite{diperna1989cauchy}, and then extended from Sobolev to BV in \cite{ambrosio2004transport}; we also refer the reader to the subsequent works \cite{crippa2008estimates, brue2021positive, brue2024sharp}. 

The superposition principle actually holds in much greater generality, substituting the ground space $\R^d$ with a general (complete, separable) metric space $X$. A first result in this setting was proved by \cite{lisini2007characterization} for absolutely continuous curves of measures in $\PP_p(X)$, the space of probability measures on $X$ endowed with the $p$-Wasserstein metric: such evolution can be represented in terms of measures over $p$-absolutely continuous curves in $X$. This result was useful to prove the equivalence of definitions for Sobolev spaces on metric measure space, in particular for the one using test-plans, that were introduced in \cite{ambrosio2013density, ambrosio2014calculus}. For a further understanding of the topic, we refer the reader to \cite{AmIkLuPa24}. 
\\
The superposition result was then refined 
by taking into account the
non-smooth `differential structure' of the space: thanks to the concept of derivations, introduced in \cite{weaver2000lipschitz} (see also \cite{di2014recent}), we may give meaning to \eqref{eq: 1} and \eqref{eq: 2}, and the superposition result still holds. These kinds of results can be found in \cite{ambrosio2014well}, in which they work on a metric measure space $(X,d,\m)$, and in \cite{stepanov2017three}, working on general metric spaces, also comparing it with a Smirnov-type decomposition for normal metric currents, \cite{smirnov1994decomposition,paolini2012decomposition, paolini2013structure}. It is worth citing also these kinds of results obtained on Wiener spaces \cite{ambrosio2009flows, trevisan2015lagrangian}, or in a stochastic setting \cite{figalli2008existence, trevisan2016well}, that aimed to prove existence and uniqueness for Lagrangian flows associated to Sobolev and BV coefficients. 

One of the main application of these results is to study the geometry of the Wasserstein space. In \cite[Chapters 6-8]{ambrosio2005gradient}, the authors extensively studied the geometry of the Wasserstein space $(\PP_p(\R^d),W_p)$: the (static) optimal transport problem, with generic convex costs; the characterization of absolutely continuous curves and geodesics on the Wasserstein space; the characterization of the OT problem in terms of the continuity equation (the so-called Benamou-Brenier formulation); the characterization of its tangent bundle; etc... 

In this paper, one of the main goals is to reproduce this theory for the continuity equation for measures on $\PP_p(\R^d)$ 
and for corresponding evolutions in the \textit{Wasserstein on Wasserstein space} $(\PP_p(\PP_p(\R^d)),\mathcal{W}_p)$. For a generic metric space $(X,d)$, the set $\PP_p(\PP_p(X))$ is defined as the collection of random measures $M\in \PP(\PP(X))$ that satisfy
\[\int_{\PP(X)} W_p^p(\mu,\delta_{x_0}) dM(\mu) = \int_{\PP(X)} \int_X d^p(x,x_0) d\mu(x)dM(\mu)<+\infty,\]
for some (and then all) $x_0\in X$. Then, the Wasserstein on Wasserstein distance is defined as 
\begin{equation}\label{eq: static p-Wass prob}
    \mathcal{W}_p^p(M,N) := \inf \left\{ \int_{\PP(X)\times \PP(X)} W_p^p(\mu,\nu)d\Pi(\mu,\nu) \, : \, \Pi \in \Gamma(M,N) \right\},
\end{equation}
for all $M,N\in \PP_p(\PP_p(X))$, where $\Gamma(M,N)\subset \PP(\PP(X)\times \PP(X))$ is the set of couplings between $M$ and $N$. In \cite{PS25convex}
(see also \cite{emami2025optimal, beiglbock2025brenier}), 
we study the 
structure of solutions to the static problem \eqref{eq: static p-Wass prob}
for $p=2$ and $X$ Hilbert space
and we show its link with the theory of \textit{totally convex functionals}, that allows (surprisingly) to recover many features of convex analysis in this infinite dimensional and non-linear setting. 

In the present paper, we focus on the \textit{dynamic aspects} of the Wasserstein on Wasserstein space:
\begin{enumerate}
    \item its absolutely continuous curves and geodesics (Section \ref{metric structure}), under the 
    solely general assumption that $(X,d)$ is a complete and separable metric space;
    \item its `differential structure', i.e. the description of absolutely continuous curves in terms of an abstract continuity equation for random measures, Sections \ref{CE: derivations vs VF} and \ref{superposition}, when $X=\R^d$;
    \item the description of its tangent bundle $\operatorname{Tan}\PP_p(\PP_p(\R^d))$, Section \ref{subsec: tangent}, again when $X=\R^d$.
\end{enumerate}
One of the main tools we 
develop to pursue these objectives is 
a \textit{nested superposition principle}: it comes both in the metric form (in the spirit of \cite{lisini2007characterization}) and its differential form (in the spirit of \cite[Theorem 8.2.1]{ambrosio2005gradient}). The term \textit{nested}, comes from the particular structure of the space considered. Indeed, we will prove that starting from an absolutely continuous curve of random measures $(M_t)_{t\in[0,T]}\in C([0,T],\PP(\PP(\R^d)))$ 
(possibly arising as a solution of a continuity equation), we can lift it to two ``dynamic''
measures: the first one, $\Lambda \in \PP(C([0,T],\PP(\R^d)))$,
is a measure on curves of probability measures
on $\PP(\R^d)$,
the second one, $\mathfrak{L}\in \PP(\PP(C([0,T],\R^d)))$, 
is a law on random curves in $\R^d$. 


We describe such a nested superposition principle in both the metric and the differential setting, in the symplified case of 
of the evolution of 
a particle system.

\newcommand{\myX}{X}
\newcommand{\pmyX}{X}
\bigskip
\noindent
\textbf{Interacting $N$-particle systems.} Consider an interacting system of $N\in \N$ particles in 
$\myX:=\R^d$, described by a vector $\boldsymbol x(t)=(x_1(t),\cdots,x_N(t))\in \pmyX^N$;
the velocity of each particle $x_i(t)$ is expressed by $v(t,x_i(t),\boldsymbol x(t))$, where
$v:[0,T]\times \myX \times \pmyX^N \to \myX$ is a smooth (for the easy of presentation) vector field symmetric in the last $N$ components, i.e.~denoting by $S_N$ the usual symmetric group of permutations of $\{1, \dots,N \}$, it holds
\[v(t,x,\boldsymbol{x}) = v(t,x,\sigma(\boldsymbol{x})) \quad \forall \sigma\in S_N, \quad \sigma(x_1,\dots,x_N) := (x_{\sigma(1)},\dots,x_{\sigma(N)}).\]
Given any initial position $\boldsymbol x_{0} = (x_{1,0},\dots,x_{N,0})\in \pmyX^N$
of the $N$ particles, there exists a unique solution $\boldsymbol x:[0,T]\to \pmyX^N$ 
whose components solve
\begin{equation}\label{int part system intro}
\begin{cases}
    \dot{x}_i(t) = v(t,x_i(t),\boldsymbol x(t))
    \\
    x_i(0) = x_{i,0}
\end{cases}
\end{equation}
The evolution can also be described either by time-dependent flow map 
 $\boldsymbol X:[0,T]\times \pmyX^N \to \pmyX^N$ such that $X(t,\boldsymbol x_0) = \boldsymbol x(t)=(x_1(t),\dots,x_N(t))$, or by the point-to-curve
evolution map $\Gamma:\pmyX^N \to 
C([0,T],\pmyX^N)$ such that $\Gamma(\boldsymbol{x}_0) = \boldsymbol{x}(\cdot)
=\boldsymbol X(\cdot,\boldsymbol x_0)$. Both these maps are invariant with respect permutations of the initial distribution of particles. 

Suppose that we describe a distribution
on the initial configurations of particles by assigning a symmetric
probability measure $m \in \PP(\pmyX^N)$,
thus satisfying $\sigma_\#m = m$ for all permutations $\sigma\in S_N$.
The above maps $\boldsymbol X,\Gamma$, 
can be used to describe the evolution of 
$m$ driven by the system \eqref{int part system intro}: 
\begin{enumerate}
    \item for all $t\in [0,T]$, $m_t := (X(t,\cdot))_\#m \in \PP(\pmyX^N)$ is a curve of probability measures that solves the continuity equation 
    \begin{equation}\label{eq: ce on R^d^N}
    \partial_t m_t + \operatorname{div}_{\boldsymbol{x}}(\boldsymbol{v}_tm_t) = 0,
    \end{equation}
    where $\boldsymbol v:[0,T]\times \pmyX^N\to \pmyX^N$ 
    whose components are $\boldsymbol v_i(t,\boldsymbol x)=v(t,x_i,\boldsymbol x)$;
    \item $\boldsymbol{\eta}:= \Gamma_\#m \in \PP(C([0,T],\pmyX^N))$ is a probability measure over curves of $\pmyX^N$, concentrated over curves $\boldsymbol{x}(\cdot)$ that solve the system \eqref{int part system intro}; 
    \item the canonical isomorphism 
    between $C([0,T],\pmyX^N)$ and $C([0,T],\myX)^N$  allows us to define a probability measure $\boldsymbol{\theta} \in \PP(C([0,T],\myX)^N)$.
\end{enumerate}
Thanks to the invariance with respect to permuations, it is natural
to consider the projected evolution
in the quotient space $\myX^N/S_N$,
that we can identify with the 
space of uniform discrete measures 
of the form $\frac{1}{N}\sum_{i=1}^N\delta_{x_i}.$
To this aim, for every space $\mathcal X$, we 
introduce the function
\begin{equation}\label{eq: 3}
    \mathcal{J}:
    \mathcal X^N \to \PP(\mathcal X)\quad 
    \mathcal{J}(x_1,\dots, x_N)
    \GGG :=
    \nc \frac{1}{N}\sum_{i=1}^N\delta_{x_i},
\end{equation}
and we 
transform the measures
$(m_t),\boldsymbol \eta,\boldsymbol \theta$ 
under the action of 
the corresponding versions of $\mathcal J$:
\begin{enumerate}
    \item 
    starting from the curve $(m_t)$
    of measures in $\myX^N$
    we 
    obtain the curve of probability measures over probability measures (that we call (law of) \textit{random measures})
    \begin{equation}
    \label{eq: random measures associated to N-particle}
        M_t:=\mathcal{J}_\#m_t \in \PP(\PP(\myX)),\quad t\in [0,T];
    \end{equation}
    \item 
    starting from $\boldsymbol\eta$ 
    and using the map $\mathcal J':
    C([0,T];\myX^N)\to 
    C([0,T];\PP(\myX))$ 
    defined as 
    \[ \mathcal J'[\boldsymbol x](t):=\mathcal J(\boldsymbol x(t))\quad\text{for every }
    \boldsymbol x\in C([0,T];\myX^N),
    \]
    we get 
    $\Lambda := \mathcal{J}'_\#\boldsymbol{\eta} \in \PP(C([0,T],\PP(\R^d)))$;
    \item 
    starting from $\boldsymbol{\theta}$
    and using the map $\mathcal J$
    in the space
    $\mathcal X:=
    C([0,T],\myX)$
    we obtain
    $\mathfrak{L}:= \mathcal{J}_\#\boldsymbol{\theta}\in \PP(\PP(C([0,T],\R^d)))$.
\end{enumerate}
Thanks to the symmetry assumption, the original structure that we had between the measures $(m_t)_{t\in [0,T]}$, $\boldsymbol{\eta}$, and $\boldsymbol{\theta}$ is maintained, and no information have been lost about the evolution.
Moreover, interpreting $v(t,x,\boldsymbol x)$
as the restriction to discrete probability measures of a nonlocal vector field 
$b:[0,T]\times \R^d \times \PP(\R^d) \to \R^d$ 
via the formula 
\begin{equation}\label{non-local vf intro finite part system}
    v(t,x,\boldsymbol x)= b(t,x,\mathcal{J}(\boldsymbol x)),
\end{equation}
we can provide intrinsic differential characterizations of 
$M,\Lambda,\mathfrak{L}$ as follows:
\begin{enumerate}
    \item 
     the curve of random measures $(M_t)_{t\in [0,T]}$ solves an abstract continuity equation 
     (see Example \ref{example: N particles CERM})
\begin{equation}\label{CERM intro}
    \partial_tM_t + \operatorname{div}_{\PP}(b_tM_t) = 0,
\end{equation}
in duality with smooth cylinder functions (see \eqref{eq:continuity-test-intro} and Section \ref{CE: derivations vs VF}).
The latter
are functions
$F:\PP(\myX) \to \R$
of the form 
$F(\mu) = \Psi(\int_{\myX}\phi_1 d\mu,\ldots, \int_{\myX}\phi_kd\mu)$
and admit a 
Wasserstein gradient 
defined as 
\[\nabla_W F(x,\mu) := \sum_{j=1}^k \partial_j\Psi\left(\int_{\myX}\phi_1 d\mu,\ldots, \int_{\myX}\phi_kd\mu\right) \nabla\phi_j(x) \quad \forall (x,\mu)\in \myX\times \PP(\myX).\]
    \item the probability measure $\Lambda \in \PP(C([0,T],\PP(\myX)))$ is concentrated over curves $(\mu_t)_{t\in[0,T]}$ solutions of the non-local continuity equation on $\myX$ given by 
    \begin{equation}
        \partial_t\mu_t + \operatorname{div}_{\myX}(b_t(\cdot,\mu_t) \mu_t) = 0.
    \end{equation}
    We can think that each $\mu_t$ is an empirical measure associated to a group of $N$ particles following the flow given by $b$. 
    \item the probability measure $\mathfrak{L}\in \PP(\PP(C([0,T],\myX)))$ is concentrated over 
    dynamic measures
    $\lambda \in \PP(C([0,T],\myX))$ that, in turn, are concentrated over solutions of 
    \[\dot{\gamma}(t) = b(t,\gamma(t), (e_t)_\#\lambda),\]
    where $(e_t)_\#\lambda$ is the marginal at time $t$ of $\lambda$. This means that $\lambda$ 
    is the collective 
    distribution of the 
    trajectories of single particles
    evolving according to the system \eqref{int part system intro}.
\end{enumerate}

\bigskip
\noindent 
\textbf{Nested superposition principles.} 
Our goal is to recover the same structure
for general evolution of 
laws of random measures, 
also including the
general setting of a complete and separable metric space $(X,d)$ for the 
metric-variational aspects 
(we use the shorthand $C_T(\mathcal X)$ for 
$C([0,T],\mathcal X)$ for any space $\mathcal X$):
\begin{itemize}
    \item a (absolutely) continuous curve of random measures $(M_t)_{t\in [0,T]}\in C_T(\PP(\PP(X)))$ (see §\ref{subsec: topology over random measures});
    \item a probability measure over (absolutely) continuous curves of measure $\Lambda \in \PP(C_T(\PP(X)))$;
    \item a 
    law of random measures over 
    (absolutely) continuous curves $\mathfrak{L}\in \PP(\PP(C_T(X)))$.
\end{itemize}
%
There is a natural hierarchy 
``$\mathfrak{L} \implies \Lambda \implies M_t$'' between these objects,
in the sense that any measure $\mathfrak{L}\in \PP(\PP(C_T(X)))$ naturally induces a $\Lambda\in \PP(C_T(\PP(X)))$, which in turn induces a curve $(M_t)_{t\in [0,T]} \in C_T(\PP(\PP(\R^d)))$.
This relation can be described from the standard viewpoint of Bayesian sampling schemes. Recall that, for a general random probability measure \(M \in \PP(\PP(X))\), the associated sampling procedure is: first draw \(\mu \in \PP(X)\) with law \(M\); then, conditional on \(\mu\), draw \(x \in X\) with law \(\mu\).

In particular, we will apply this scheme to the three objects introduced above. 
Given $t\in [0,T]$ 
we have:
\begin{align*}
        & \mu \sim M_t & & \boldsymbol{\mu}  \sim \Lambda &  &\lambda  \sim \mathfrak{L} 
        \intertext{and, conditionally,}
        & x \ | \ \mu   \sim \mu & & y \ | \ \boldsymbol{\mu}  \sim \mu_t & & \gamma \ | \ \lambda  \sim \lambda
\end{align*}
where $\mu\in \PP(X)$, $\boldsymbol{\mu} = (\mu_t)_{t\in [0,T]} \in C_T(\PP(X))$, $\lambda \in \PP(C_T(X))$, $\gamma \in C_T(X)$ and $x,y \in X$.
Therefore, 
\begin{itemize}
    \item[(i)] given a sample $(\lambda,\gamma)$ from the third model, a sample from the second model is 
    \begin{equation}\label{model from L to Lambda}
    \boldsymbol{\mu} := E\big(\lambda\big), \quad y:= e_t(\gamma) =[\gamma](t),
    \end{equation}
    where $E:\PP(C_T(X))\to C_T(\PP(X))$ is defined as $E(\lambda) := \big((e_t)_\#\lambda\big)_{t\in[0,T]}$. In other words, $\Lambda := E_\# \mathfrak{L}$;
    \item[(ii)] given a sample $(\boldsymbol{\mu},y)$ from the second model, we recover a sample from the first model as 
    \[\mu := \mu_t, \quad x := y.\]
    In other words, $M_t := (\mathfrak{e}_t)_\#\Lambda$, where $\mathfrak{e}_t(\boldsymbol{\mu}) := \mu_t$;
    \item[(iii)] Clearly, we could also 
    pass directly from $\mathfrak{L}$ to $M_t$, by defining
    \[\mu := \mathfrak{e}_t\circ E (\lambda), \quad x:= \gamma(t).\]
    Notice also that $E_t := \mathfrak{e}_t \circ E: \PP(C_T(X)) \to \PP(X)$ coincides with the push forward of the evaluation map at time $t$, $(e_t)_\#$. Then, $M_t := (E_t)_\# \mathfrak{L}=((e_t)_\sharp)_\sharp
    \mathfrak L$, 
    is obtained by a nested push-forward construction.
\end{itemize}

In Section \ref{metric structure}, we show that such a hierarchy can be reversed, preserving
relevant regularity and structural properties. 
More precisely, 
given 
a 
curve of random measure $(M_t)_{t\in [0,T]} \in AC_T^p(\PP_p(\PP_p(X)))$, $p>1$,
\nc 
we show that there exist $\Lambda \in \PP(C_T(\PP(X)))$ and $\mathfrak{L}\in \PP(\PP(C_T(X)))$ 
consistent with the hierarchical structure above and 
satisfying the additional properties that are summarized in the following theorem, 
whose complete proof is a byproduct of Sections \ref{from ac(pp) to P(ac(p))}---%
\ref{comp of maps}. 

\begin{teorema}[Nested metric superposition and minimal energy liftings]\label{metric superposition}
    Let $(X,d)$ be a complete and separable metric space and $p>1$. Let $\boldsymbol{M}=(M_t)_{t\in[0,T]} \in AC_T^p(\PP_p(\PP_p(X)))$.
    Then, there exist $\Lambda \in \PP(C_T(\PP(X)))$ and $\mathfrak{L}\in \PP(\PP(C_T(X)))$ satisfying:
    \begin{enumerate}
        \item[(1)] $(\mathfrak{e}_t)_\#\Lambda = M_t$ for all $t\in [0,T]$, 
        $\Lambda$-a.e. $\boldsymbol{\mu}$ is in $AC_T^p(\PP_p(X))$, and 
        \begin{equation}
            \int_0^T |\dot{\boldsymbol{M}}|_{\mathcal{W}_p}^p(t) dt = \int \int_0^T |\dot{\boldsymbol{\mu}}|_{W_p}^p(t) dt\, d\Lambda(\boldsymbol{\mu});
        \end{equation}
        \item[(2)] $(E_t)_\#\mathfrak{L} = M_t$ for all $t\in [0,T]$,  $\mathfrak{L}$-a.e. $\lambda \in \PP(C_T(X))$ is concentrated over $AC_T^p(X)$, and 
        \begin{equation}
            \int_0^T |\dot{\boldsymbol{M}}|_{\mathcal{W}_p}^p(t) dt = \int \int \int_0^T |\dot{\boldsymbol{\gamma}}|^p \,dt
        \,d\lambda(\boldsymbol{\gamma}) \,d\mathfrak{L}(\lambda);
        \end{equation}
        \item[(3)] $\Lambda = E_\#\mathfrak{L}$ and there exists a $\Lambda$-measurable map $G:C_T(\PP(\R^d)) \to \PP(C_T(\R^d))$ such that $\mathfrak{L}= G_\# \Lambda$ and $E(G(\boldsymbol{\mu})) = \boldsymbol{\mu}$ for $\Lambda$-a.e. $\boldsymbol{\mu}$.
    \end{enumerate}
\end{teorema}
 
In the particular case of the Euclidean setting,
we can further leverage the fine-grained information encoded in the continuity equation
\eqref{CERM intro}.
The corresponding main result
when $X=\R^d$ 
is a byproduct of Sections \ref{subsec: first superposition result} and \ref{superposition}
and can be summarized as follows
(we keep the notation for the 
evaluations maps $e_t,E,E_t,\mathfrak e_t$ introduced above).

\begin{teorema}[Nested superposition principle for the random continuity equation in $\R^d$]\label{main theorem}
    Let $\myX=\R^d$ and let $\boldsymbol{M}=(M_t)_{t\in [0,T]}\in C_T(\PP(\PP(\myX)))$ and $b:[0,T]\times \myX \times \PP(\myX) \to \myX $ be a 
    Borel non-local vector field satisfying the integrability condition
    $$\int_0^T \int \int |b(t,x,\mu)| d\mu(x) dM_t(\mu) dt <+\infty.$$
    \nc%
    Assume that 
    $(\boldsymbol{M},b)$
    satisfy the continuity equation \eqref{CERM intro}
    \begin{equation}
    \label{eq:CERM2}
        \partial_tM_t + \operatorname{div}_{\PP}(b_tM_t)=0\quad t\in (0,T),
    \end{equation}
    is satisfied 
    in duality with cylinder functions (see Definition \ref{cyl functions}), in the sense that for any $F\in \operatorname{Cyl}_c^1(\PP(\myX))$ it holds
    \begin{equation}
        \label{eq:continuity-test-intro}
    \frac{d}{dt}\int_{\PP(\myX)}F(\mu) dM_t(\mu) = \int_{\PP(\myX)}\int_{\myX} b_t(x,\mu)\cdot \nabla_W F(x,\mu) d\mu(x) dM_t(\mu)\quad\text{in }\mathscr D'(0,T).
    \end{equation}
    Then, there exists $\Lambda\in \PP(C_T(\PP(\myX)))$ and $\mathfrak{L}\in \PP(\PP(C_T(\myX)))$ such that:
    \begin{enumerate}
        \item\label{prop of Lambda R^d} $(\mathfrak{e}_t)_{\#}\Lambda = M_t$ for any $t\in [0,T]$, 
        and $\Lambda$-a.e.~$\boldsymbol{\mu}$ belongs to $
        AC_T(\PP(\myX))$ 
        and solves 
        $$\partial_t\mu_t + \operatorname{div}(b_t(\cdot,\mu_t)\mu_t)=0
        \quad\text{in }\mathscr D'((0,T)\times \myX)$$
        \item\label{properties of L R^d} $(E_t)_\#\mathfrak{L}=M_t$
        and $\mathfrak{L}$-a.e.~$\lambda\in \PP(C_T(\myX))$
        is concentrated over 
        absolutely continuous curves 
        $\gamma$ that are solutions of
        $$\dot{\gamma}(t) = b(t,\gamma_t,(e_t)_\#\lambda)\quad\text{in }(0,T).$$
        \item\label{third property} $\Lambda = E_\#\mathfrak{L}$.
    \end{enumerate}
\end{teorema}



\bigskip

\textbf{Links with the Wasserstein on Wasserstein geometry.}
As a consequence of Theorems \ref{metric superposition} and \ref{main theorem}, we obtain important geometric
properties of $(\PP_p(\PP_p(\myX)),\mathcal{W}_p)$. In particular:
\begin{enumerate}
    \item 
When $(X,d)$ is 
a complete, separable, and \emph{geodesic}
 we characterize 
 geodesics in
 $(\PP_p(\PP_p(\myX)),\mathcal{W}_p)$ as superposition of 
 laws of random 
 geodesics of $X$, see Section \ref{subsec: geodesics}. 
 The geodesics are also related to optimal couplings and optimal random couplings presented in \cite{PS25convex} (in the Hilbertian case with $p=2$) 
 and in Section \ref{couplings};
    \item Absolutely continuous curves of random measures $(M_t)_{t\in[0,T]}$ in $AC_T^p(\PP_p(\PP_p(\R^d)))$ 
    can be represented 
     as solutions to 
    the continuity equation 
    \eqref{eq:CERM2} 
    driven by 
    a unique non-local vector field $b:[0,T]\times \R^d \times \PP(\R^d) \to \R^d$ 
    satisfying suitable variational and integrability conditions, 
    see Section \ref{subsec: corr AC with CERM}. 
    \item
    We can fully justify the definition 
    of the cotangent space 
    of $\PP_p(\PP_p(\R^d))$ at $M$ as the closure of the Wasserstein gradient of cylinder functions (see \eqref{eq: tangent} and Section \ref{subsec: tangent}):
    \begin{equation}\label{eq: cotangent-intro}
        \operatorname{CoTan}_M \PP_p(\PP_p(\R^d)) := 
        \overline{
        \Big\{\nabla_W F \ : \ F\in\operatorname{Cyl}_c(\PP(\R^d)) \Big\} }^{L^{p'}(\widetilde{M};\R^d)},
    \end{equation}
    where the measure $\widetilde{M} \in \PP(\R^d \times \PP(\R^d))$ is defined as 
    \[\widetilde{M} := \int_{\PP(\R^d)} \mu \otimes \delta_{\mu}\, dM(\mu),
\qquad
\widetilde{M}(A \times B) = \int_{B} \mu(A)\, dM(\mu)
\]
for all Borel sets \(A \subseteq \R^d\) and \(B \subseteq \PP(\R^d)\).
    In fact, its corresponding tangent space 
    in $L^{p}(\widetilde{M};\R^d)$
    obtained by the duality map from $L^{p'}$ to $L^p$ is generated by 
    all the non-local vector fields of minimal velocity, thus representing
    the infinitesimal behaviour of
    all the absolutely continuous curves according to the previous point (2). In this way,
    we reproduce the results in \cite[Chapter 8.4]{ambrosio2005gradient} 
    at the level of random measures.
\end{enumerate}
A remarkable corollary of the above results is the \textit{Benamou-Brenier like formula} for the $p$-Wasserstein distance on random measures
(see also \cite{HUESMANN2025104633} for a similar setting):
\begin{teorema}[Benamou-Brenier formula]\label{thm: BB}
    Let $p>1$. For all $M_0,M_1\in \PP_p(\PP_p(\R^d))$ it holds
    \begin{equation}\label{eq: BB intro}
    \begin{aligned}
        \mathcal{W}_p^p(M_0,M_1) = \operatorname{min}\bigg\{
        &\int_0^1 \int_{\PP}\int_{\R^d} |b_t(x,\mu)|^p d\mu(x)dM_t(\mu)dt \, :
        M\in AC^p(0,1;\PP_p(\PP_p(\R^d)),
        \\&\quad 
        M(0)=M_0,\ M(1)=M_1,\ 
        \partial_t M_t + \operatorname{div}_\PP(b_tM_t) = 0 \bigg\}.
    \end{aligned}
    \end{equation}
\end{teorema}

\noindent
\textbf{Other results and literature.} 
The Wasserstein on Wasserstein metric has been studied in recent years, mostly to quantify convergence properties for non-parametric statistical problems \cite{nguyen2016borrowing, catalano2024hierarchical}. In the recent paper \cite{bonet2025flowing}, the authors use the gradient flow theory on the $L^2$-Wasserstein on Wasserstein space to solve learning tasks (e.g. the multi-classification problem) through the minimization of suitable functions defined over random measures. For this aim, it was crucial to define a notion of tangent space, a continuity equation, and their link with absolutely continuous curves. In this paper, we completely characterize these objects, extending some of their results, even in the case $\PP_p(\PP_p(\R^d))$ with $p>1$.

In \cite{acciaio2025absolutely} the authors prove a metric superposition principle in the spirit of \cite{lisini2007characterization} for absolutely continuous curves of stochastic processes with respect to the adapted Wasserstein metric \cite{bartl2024wasserstein}. The strong relation between the iterated Wasserstein space and the one of filtered processes endowed with the adapted Wasserstein metric has been highlighted in \cite{beiglbock2025brenier}, where they study a Monge-Brenier theorem for the static iterated optimal transport in the $N$-iterated $2$-Wasserstein space. We plan to further develop the techniques presented in this paper for the study of the geometry of the $N$-iterated $p$-Wasserstein space, with possible applications to the space of filtered processes.  

In \cite{pinziatomic2025}, the first author studies the nested superposition principle adding the requirements that all the random measures are absolutely continuous with respect to suitable reference random measures $Q\in \PP_p(\PP_p(\R^d))$ (see also \cite{dello2022dirichlet,delloschiavo2024massive}), with applications to the study of the metric measure space $(\PP_p(\R^d),W_p,Q)$. The same technique is then applied to the Wasserstein space over a compact Riemannian manifold, and a version of Theorem \ref{main theorem} is proved in this setting as well.

In Theorem \ref{main theorem}, claim (1) can be seen as a particular case of the main result of \cite{stepanov2017three}, 
under the hypotheses
\begin{equation}\label{eq: 4}
\int_0^T \int_{\PP(\R^d)}\int_{\R^d} |b_t(x,\mu)|^pd\mu(x) dM_t(\mu)dt<+\infty, \quad M_0 \in \PP_p(\R^d),
\end{equation}
for some $p>1$. Indeed, as shown in \cite{sodini2023general}, the local Lipschitz constant in the space $(\PP_p(\R^d),\mathcal{W}_p)$ of a cylinder function $F:\PP(\R^d) \to \R$ is given by 
\[ \left(\int_{\R^d} |\nabla_WF(x,\mu)|^pd\mu(x)\right)^{1/p}\]
Exploiting it and our definition of derivation (see Definition \ref{def derivations}), one can rewrite our setting only using the metric properties needed to apply the results in \cite{stepanov2017three}. Anyway, the differences are in the fact that we do not need the additional integrability assumption (actually, in \cite{pinzifirst2025} it is relaxed to the integrability of $\frac{b_t(x,\mu)}{1+|x|}$), and exploiting the particular structure of the space of probability measures, we can perform the other lifting as well, as in Claims (2) and (3).

A similar result to Theorem \ref{main theorem} was already obtained in \cite{lacker2022superposition}. In addition to a non-local vector field, they have also two operators $\sigma:[0,T]\times \R^d \times \PP(\R^d) \to \R^{d\times d}$ and $\gamma:[0,T]\times \R^d \times \PP(\R^d) \to \R^{d\times d}$, that describe diffusive terms associated with two independent Brownian motions. The novelties in the present paper are various:
\begin{itemize}
    \item to directly apply the result of \cite{lacker2022superposition}, one should ask the $p$-integrability assumption for the vector field $b$, as in \eqref{eq: 4};
    \item the proof of Claim (1) is proved identifying $\PP(\R^d)$ with $\R^\infty$ in both cases. On the other hand, to prove Claim (2), in \cite{lacker2022superposition} the authors reproduce the approximation procedure that is commonly used to prove superposition results. Here, we propose a different approach, based on using as a black-box the finite dimensional superposition principle to perform a measurable selection, that will give us, as a byproduct, Claim (3). In doing this, we need to prove that the set of curves of measures that are solutions of the continuity equation $\partial_t\mu_t + \operatorname{div}(b_t(\cdot,\mu_t)\mu_t) = 0$ is Borel, and similarly for the set of $\lambda\in \PP(C_T(\R^d))$ that are concentrated over solutions of $\dot{\gamma}(t) = b_t(\gamma(t),(e_t)_\#\gamma)$. These measurability results are the main results in Section \ref{subsec: from mathfrak L to Lambda} and may have their own independent interest;
    \item we use this specific setting as a tool to study the geometry of the Wasserstein on Wasserstein space. In particular, Theorem \ref{main theorem} put in relation the purely-metric setting of the Wasserstein on Wasserstein space with the non-local continuity equations of the form $\partial_t\mu_t + \operatorname{div}(b_t(\cdot,\mu_t)\mu_t) = 0$, that have been intensively studied in recent years, e.g. \cite{bonnet2021differential,bonnet2024caratheodory, cavagnari2023dissipative, cavagnari2023lagrangian}.
\end{itemize}


Regarding uniqueness, in Section \ref{uniqueness section} we assume $p\geq 1$, $M_0\in \PP_p(\R^d)$ and $b$ such that
\begin{equation}\label{eq: intro lip assum}
    \int_{\R^d\times \R^d}|b(t,x_0,\mu_0) - b(t,x_1,\mu_1)|^p d\pi(x_0,x_1) \leq L(t)  W_p^p(\mu_0,\mu_1),
\end{equation}
for all $\mu_0,\mu_1 \in \PP_p(\R^d)$ and some optimal coupling $\pi$ between $\mu_0$ and $\mu_1$, with $L\in L^1(0,T)$. Under this Lipschitz assumption, we show uniqueness of $\boldsymbol{M}$, $\Lambda$ and $\mathfrak{L}$ that start from $M_0$. As already pointed out in \cite{carmona2018probabilistic,lacker2022superposition}, if the vector field is Lipschitz in $x$, uniformly with respect to $(t,\mu)$, then well-posedness for the interacting particle system easily follows by standard techniques, from which uniqueness of the previous objects follows. On the other hand, this is less trivial assuming only \eqref{eq: intro lip assum}: our proof actually shows that such a Lipschitz property is rigid enough to imply that the map $\operatorname{supp}(\mu)\ni x\mapsto b_t(x,\mu)$ is $L(t)$-Lipschitz for any $\mu\in \PP_p(\R^d)$, and now we can proceed by proving uniqueness of the interacting particle system by the previously cited techniques.

\textbf{Plan of the paper.}
In \textbf{Section \ref{preliminary section}}, we recall the main known ingredients that we need to develop our results. In particular, we fix natural (Polish) topologies over the space of probability measures over a Polish space, the space of continuous curves and all their possible combinations, that will be fixed for the rest of the paper.

In \textbf{Section \ref{metric structure}}, we prove Theorem \ref{metric superposition}, introducing the method used for the proof in Section \ref{couplings}, that shows how the optimal transport problem between random measures $M,N \in \PP_p(\PP_p(\R^d))$ can be seen as a minimum problem either over couplings $\Pi \in \PP(\PP(\R^d) \times \PP(\R^d))$ or over random couplings $\mathfrak{P}\in\PP(\PP(\R^d\times \R^d))$. In Section \ref{subsec: geodesics}, we then apply Theorem \ref{metric superposition} to study the geodesics of the Wasserstein on Wasserstein space $\PP_p(\PP_p(\R^d))$.

In \textbf{Section \ref{CE: derivations vs VF}}, we introduce the continuity equation on random measures, that follows either the dynamics of a family of derivations defined over cylinder functions or the dynamics led by a non-local vector field. In this setting, we prove Claim (1) in Theorem \ref{main theorem}, following the strategy developed in \cite{ambrosio2014well}. This result, together with Theorem \ref{metric superposition}, will allow us to characterize absolutely continuous curves of random measures as solutions of a continuity equation, in Section \ref{subsec: corr AC with CERM}.
In Section \ref{subsec: tangent}, we define the tangent and cotangent spaces of $\PP_p(\PP_p(\R^d))$ at a fixed random measure $M$, and we give a characterization of it in terms of non-local vector fields of minimal energy solving the continuity equation. Finally, we prove a representation result for derivations as non-local vector field, see Theorem \ref{derivation induced by a vf theorem}.

In \textbf{Section \ref{superposition}}, we complete the proof of Theorem \ref{main theorem}, using a similar strategy used for Theorem \ref{metric superposition}. In particular, we first show the measurability of the sets $\operatorname{CE}(b)$ and $\operatorname{SCE}(b)$ (see Definition \ref{def: CE and SCE}) that we exploit to apply a measurable selection argument to define the map $G_b$. 

In \textbf{Section \ref{uniqueness section}}, we study the case of a Lipschitz non-local vector field, in the measure variable. 

In the appendices, we collect some technical results. In particular: in \textbf{Appendix \ref{appendix souslin}} we recall the definitions and properties of Lusin and Souslin sets, together with a measurable selection theorem; in \textbf{Appendix \ref{appendix R^infty}} a natural topological-metric structure of $\R^\infty$ is highlighted, and its relation to the space of probability measures as well; in \textbf{Appendix \ref{appendix meas for curves}} some results about curves in $\PP(\R^\infty)$ and $\PP(\PP(\R^d))$ are collected;
in \textbf{Appendix \ref{app: measurability}} we show some results concerning measurability on the space of probability measures. 

\bigskip

\noindent\textbf{Acknowledgements.} We wish to thank Anna Korba, Eugenio Regazzini and Luciano Tubaro
for stimulating and insightful discussions. We also thank Beno\^it Bonnet-Weill and Martin Huesmann for their helpful comments on a draft of the present paper. 
\\
GS has been supported by the MIUR-PRIN 202244A7YL project Gradient Flows and Non-Smooth
Geometric Structures with Applications to Optimization and Machine Learning, by the INDAM project
E53C23001740001, and by the Institute for Advanced Study of the Technical University of Munich, funded
by the German Excellence Initiative.

\section{Preliminaries}\label{preliminary section}

\noindent Here's a list of the main notations used throughout the paper.

\smallskip
\halign{$#$\hspace{0.7cm}\ & #\hfil
\cr
C_b(X), \, (C_b(X;\R^n) )&continuous and bounded functions from $X$ to $\R$ (resp. $\R^n$)
\cr
C_c(X), \, (C_c(X;\R^n) )&cont. functions with compact support from $X$ to $\R$ (resp. $\R^n$)
\cr
C_c^k(\R^d), \, (C_c^k(\R^d;\R^n) )&$k$-times differentiable functions in $C_c(\R^d)$ (resp. $C_c(\R^d;\R^n)$) 
\cr
L^p(\sigma;\R^d), & functions $\R^d$-valued that are $p$-integrable in a measure space $(X,\sigma)$
\cr
\mathcal{M}_+(Y)&finite positive Borel measures on Polish space
$Y$
\cr
\mathcal{P}(Y)&Borel probability measures on a Polish space
$Y$
\cr
\mathcal{M}(Y)&signed Borel measures with finite total variation on $Y$
\cr
\mathcal{M}(Y;\R^d)&Borel measures on Y with values in $\R^d$ and finite total variation
\cr
\mathcal{P}_F(Y) & see Def. \ref{finite energies prob}
\cr
\mathcal{L}^1, \, (\mathcal{L}^1_T) &Lebesgue measure on $\R$ (resp. $[0,T]$)
\cr
C_T(X)&continuous curves from $[0,T]$ in a topological space $X$
\cr
AC_T(X)&absolutely continuous curves from $[0,T]$ to a metric space $X$ 
\cr
AC_T^p(X) & absolutely continuous curves with finite $p$-energy
\cr
D_d&$\sup$ distance in $C_T(X)$ w.r.t. the distance $d$ 
\cr
a_p,\Bar{a}_p& see Def. \ref{ac energies} and \eqref{eq: ac energy with initial moment}
\cr
\mathcal{A}_p,\Bar{\mathcal{A}}_p& see Def. \ref{energy curve of meas} and \eqref{eq: ac energy on measures with initial moment}
\cr
\Bar{\mathfrak{A}}_p&see \eqref{eq: def of mathfrak A_p}
\cr
\PP_p(X)&prob. measures on a metric space $(X,d)$ with finite $p$-moment
\cr
W_{p,d}&$p$-Wasserstein distance on $\PP_p(X)$ built on the distance $d$
\cr
\hat{W}_{1,d}&$W_{1,d\wedge 1}$, i.e. $1$-Wasserstein distance built on truncated distance
\cr
\mathcal{W}_p&$W_{p,W_p}$, distance on random measures (see §\ref{subsec: topology over random measures})
\cr
\hat{\mathcal{W}}_1&$W_{1,\hat{W}_{1,d}}$, the Wasserstein distance built over $\hat{W}_{1,d}$
\cr
e_t&evaluation at time $t$ of a curve
\cr
\mathfrak{e}_t&the evaluation at time $t$ of a curve of measures $(\mu_t)_{t\in[0,T]} \in C_T(\PP(X))$
\cr
E_t &$(e_t)_\sharp$, i.e. push-forward with respect to the map $e_t$
\cr
\operatorname{Cyl}_c^1(\PP(\R^d))&see Def. \ref{def cyl func and wass grad}
\cr
\operatorname{Cyl}_b^1(\PP(\R^d))&see Def. \ref{def cyl func and wass grad}
\cr
\nabla_W F&see Def. \ref{def cyl func and wass grad}
\cr
\widetilde{M}&see Remark \ref{rem: tilde{M}}
\cr
M_t\otimes \,\d t,\,\widetilde{M}_t\otimes \,\d t \hspace{0.5cm}&see §\ref{subsub cont curve in prob space} 
\cr
\CE(b),\,\operatorname{SPS}(b)&see Def. \ref{CEb}
\cr
}
\vspace{1cm}
\noindent
In this section, we introduce the notation about spaces of measures and spaces of curves that we will use in the following. 
We will reserve the notation
$(X,d)$ for a reference complete and separable metric space and we will use the letter $Y$ 
to denote a generic space which typically arise from suitable, possibly iterated, topological, metric or
measure-theoretic constructions
starting from $(X,d)$. 
We take some care to distinguish 
notions which solely depend on the 
(Polish) topology $\tau_Y$ of $Y$ from concepts
that also depend on the choice of a metric
$d_Y$ on $Y$.

\subsection{Spaces of curves}

\subsubsection{Space of continuous curves.}\label{subsub space of cont curves}
Let $(Y,\tau_Y)$ be a Polish topological space. We will denote with $C_T(Y):= C([0,T],Y)$ the space of continuous curves $\boldsymbol{y}:[0,T]\to Y$, naturally endowed with the \textit{compact-open topology}.

Such a topology can be metrized as well, resulting as Polish (see \cite[Theorem 2.4.3]{srivastava2008course} for separability): 
it is sufficient to 
choose the 
usual sup-metric 
\begin{equation}
    D(\boldsymbol{y}_1, \boldsymbol{y}_2):= \sup_{t\in [0,T]} \delta_Y(\boldsymbol{y}_1(t),\boldsymbol{y}_2(t)) \quad \forall \boldsymbol{y}_1,\boldsymbol{y}_2 \in  C_T(Y).
\end{equation}
associated with
any metric $\delta_Y$ over $Y$ that induces its topology $\tau_Y$. Clearly, by 
using a bounded metric $\delta_Y$,
we may assume that $D$ is bounded as well.

We will denote by $e_t:C_T(Y)\to Y$
the (continuous) evaluation map
$e_t(\boldsymbol{y}):=\boldsymbol{y}(t).$


\subsubsection{Space of absolutely continuous curves in metric spaces}\label{subsub ac curve on metric spaces}
Let us collect some definitions and results about absolutely continuous curves taking values in 
the (complete, separable) metric space
$(X,d_X)$:
note that, in this case, the metric matters, not only its induced topology.

\begin{df}[Absolutely continuous curves]
   A curve $\boldsymbol{x}:[0,T] \to X$ is said to be absolutely continuous, and we write $\boldsymbol{x}\in 
   AC_T(X)$, if there exists a function $g\in L^1(0,T)$ such that 
   \begin{equation}\label{abs cont def}
        d_X(\boldsymbol{x}(t),\boldsymbol{x}(s)) \leq \int_s^t g(r) dr \quad \text{whenever}\quad 0\leq s \leq t \leq T.
   \end{equation}
   If $g\in L^p(0,T)$, for some $p\in (1,+\infty]$, we say that $\boldsymbol{y}\in 
   AC^p_T(X)$. 
\end{df}

\noindent The space $AC^p_T(X)$, for $p\in[1,+\infty]$ is a 
\nc 
Borel subsets of $C_T(X)$ (see Appendix \ref{appendix meas for curves}).

\begin{prop}
    Let $\boldsymbol{x} 
    \in AC_T(X)$. Then the limit 
    \begin{equation}
        \lim_{s\to t} \frac{d_X(
        \boldsymbol{x}(s),
        \boldsymbol{x}(t))}{|t-s|} =:|\dot{\boldsymbol{x}}|_{d_X}
        \kern-2pt(t)
    \end{equation}
    exists for $\mathcal{L}^1$-a.e. $t\in [0,T]$ 
    and it 
    provides the smallest $g$ such that \eqref{abs cont def} is satisfied.
    We will omit 
    the subscript $d_X$ 
    when it will be clear from the context.
\end{prop}

\begin{df}[$p$-action of a curve]\label{ac energies}
    Let $\boldsymbol{x}\in C_T(X)$ and $p\in [1,+\infty)$. The $p$-action of $\boldsymbol{x}$ is defined as  
    \begin{equation}
        a_p(\boldsymbol{x}):= 
        \begin{cases}
            \displaystyle
            \int_0^T |\dot{\boldsymbol{x}}|^p(t) \,\d t \quad & \text{if } \boldsymbol{x}\in AC_T(X),
            \\
            +\infty & \text{otherwise.}
        \end{cases}
    \end{equation}
\end{df}
\subsubsection{Geodesics}\label{subsub: geod} A (minimal, constant speed)
geodesic in $(X,d_X)$ 
is a curve  $\boldsymbol{x}\in C([0,1],X)$ that satisfies 
\begin{equation}\label{eq: geodesic def}
    d(\boldsymbol{x}(t),\boldsymbol{x}(s)) = |t-s|d(\boldsymbol{x}(0),\boldsymbol{x}(1)) \quad 
    \text{for every } s,t\in[0,1].
\end{equation}
Observe that in \eqref{eq: geodesic def} it is enough to require 
$\le$, since any strict inequality somewhere would contradict the triangle inequality.

We denote by $\operatorname{Geo}(X)\subset C([0,1],X)$ the 
closed (thus Borel) subset of constant speed geodesics. 

We say that $X$ is a \textit{geodesic space} if for all $x,y \in X$ there exists $\boldsymbol{x}\in \operatorname{Geo}(X)$ such that $\boldsymbol{x}(0) = x$ and $\boldsymbol{x}(1) = y$. In a geodesic space, using a measurable selection argument (see Theorem \ref{measurable selection}), one can always find a Souslin-Borel measurable map 
\begin{equation}\label{eq: geo map}
    \operatorname{geo}:X \times X \to \operatorname{Geo}(X)
\end{equation}
that selects a constant speed geodesic given the starting and ending points. Moreover, if the geodesic property is enforced with uniqueness, i.e. if for all $x,y \in X$ there exists a unique $\boldsymbol{x}\in \operatorname{Geo}(X)$ such that $\boldsymbol{x}(0) = x$ and $\boldsymbol{x}(1) = y$, then the map $\operatorname{geo}$ is uniquely defined and Borel measurable (see \cite[Lemma 6.7.1]{Bogachev07}). 

Finally, for all $t\in[0,1]$, we denote by $\operatorname{geo}_t:X \times X \to X$ the map $e_t\circ\operatorname{geo}$, that is the evaluation at time $t$ of $\operatorname{geo}$.

\subsection{Spaces of measures}
\subsubsection{Narrow topology over the spaces of measures}\label{subsub space of measures}
Let $(Y,\tau_Y)$ be a Polish space. 
We denote by $\mathcal B(Y)$
the Borel $\sigma$-algebra generated by $\tau_Y$.
We denote with $\PP(Y)$ the set of Borel probability measures on $Y$.
More generally, we introduce the sets $\mathcal{M}_+(Y)$, $\mathcal{M}(Y)
=\mathcal M(Y,\R)$, and $\mathcal{M}(Y,\R^d)$, that are, respectively, the set of all positive finite measures, the set of all signed measures with finite total variation, and the set of all measures taking values in $\R^d$ with finite total variation. 
Recall that the total variation $|\nu|\in \mathcal{M}_+(Y)$ of a measure $\nu \in \mathcal{M}(Y;\R^d)$ is defined as 
\[|\nu|(A) := \sup\left\{ \sum_{n=1}^{+\infty} |\nu(E_n)| \ : \ \bigcup E_n = A, \ E_i \cap E_j = \emptyset \text{ as } i\neq j \right\}.\]
Note that $\PP(Y) \subset \mathcal{M}_+(Y) \subset \mathcal{M}(Y)$.
The space $\mathcal{M}(Y;\R^d)$ is endowed with the \textit{narrow topology} $\tau_N$, i.e. the coarsest topology under which
the functions let $ \mathcal{M}(Y;\R^d) \ni \nu \mapsto \int \phi\cdot \,\d\nu$ are continuous for all $\phi\in C_b(Y;\R^d)$. $\PP(Y)$ and $\mathcal{M}_+(Y)$ are closed subsets of $\mathcal{M}(Y)$. 

Recall that, given a measurable function $f:Z_1 \to Z_2$, where $(Z_i,\mathcal{S}_i)$ are general measurable spaces, and a measure $\mu\in \mathcal{M}_+(Z_1)$, we denote with $f_\sharp\mu \in \mathcal{M}_+(Z_2)$ the push-forward measure, defined as 
\[f_\sharp\mu(S) := \mu(f^{-1}(S)) \quad \forall S\in \mathcal{S}_2.\]

Under the Polish assumption on the ambient space $Y$, the narrow topology is completely characterized by the narrow convergence \cite[ Theorem 8.9.4(ii)]{Bogachev07}: given $\nu_n,\nu \in \mathcal{M}(Y;\R^d)$, we say that $\nu_n$ narrowly converges to $\nu$ (we write $\nu_n \to \nu$) if 
\[\int_Y \phi \cdot \,\d\nu_n \to \int_Y \phi \cdot \,\d\nu \quad \forall \phi \in  C_b(Y;\R^d).\]

A nice characterization of compactness in the narrow topology has been given by Prohorov (see e.g. \cite[Theorem 8.6.2]{Bogachev07}). The theorem is stated for measures in $\mathcal{M}(Y;\R^d)$, and it is also true for $\PP(Y)$ and $\mathcal{M}_+(Y)$, since they are closed subsets of $\mathcal{M}(Y)$.

\begin{teorema}
    Let $\mathcal{F}\subset \mathcal{M}(Y;\R^d)$. Then the following are equivalent:
    \begin{enumerate}
        \item $\mathcal{F}$ is relatively compact in the narrow topology;
        \item $\mathcal{F}$ is equi-bounded in total variation and equi-tight, i.e. 
        \[ \sup_{\nu\in \mathcal{F}} |\nu|(Y) <+\infty\quad \text{and}\quad\forall \varepsilon>0 \ \exists K_{\varepsilon}\subset Y \text{ compact s.t. } \sup_{\nu\in \mathcal{F}}|\nu|(Y\setminus K_\varepsilon) <\varepsilon.\]
    \end{enumerate}
\end{teorema}
Hereafter, unless otherwise stated, these spaces will always be equipped with the narrow topology; 
this, in turn, generates the 
corresponding Borel $\sigma$-algebra.

\subsubsection{Wasserstein metric}\label{OT and Wass}
Assume now the reference space is endowed with a (complete and separable) metric, and we refer to it with $(X,d_X)$. Given 
$p\geq 1$,
we introduce the space 
\[\PP_p(X):= \left\{ \sigma \in \PP(X) \ : \ \int_X d_X^p(x,\bar{x}) \,\d\sigma(x) \text{ for some } \bar{x}\in Y\right\}.\]
The set $\PP_p(X)$ is 
endowed with the $p$-Wasserstein metric defined as
    \[W_{p,d_X}^p(\sigma_1,\sigma_2):=  {\min}  \left\{ 
    \int_{X\times X} d_X^p(x_1,x_2) \,\d\pi(x_1,x_2) \ : \ \pi \in \Gamma(\sigma_1,\sigma_2) \right\}, \]
where $\Gamma(\sigma_1,\sigma_2)$ is the collection of all the transport plans 
(or couplings) $\pi$ 
with marginals $\sigma_1$ and $\sigma_2$, i.e. all the probability measures $\pi \in \PP(X\times X)$ satisfying $\pi(A\times X) = \sigma_1(A)$ and $\pi(X\times B) = \sigma_2(B)$ for all $A,B \in \mathcal{B}(X)$. When the distance $d_X$ is clear from the context, we will simply write $W_p$. $(\PP_p(X),W_p)$ is a complete and separable metric space.

By Kantorovich duality we 
have:
    \begin{equation}
    \begin{aligned}
        W_p^p(\sigma_1,\sigma_2) = \sup\bigg\{ 
        \int_X \phi \,\d\sigma_1 + \int_X \psi \,\d\sigma_2 \ : \ & \phi,\psi\in C_b(X),
        & \phi(x_1) + \psi(x_2) \leq d^{p}(x_1,x_2) \bigg\}.
    \end{aligned}
    \end{equation}
    In the particular case $p=1$, the duality formula can be rewritten as 
    \begin{equation}
        W_1(\sigma_1,\sigma_2) = \sup\bigg\{ 
        \int_X \phi \,\d\sigma_1 - \int_X \phi \,\d\sigma_2 \, : \, \phi\in\operatorname{Lip}_1(X) \bigg\},
    \end{equation}
where $\operatorname{Lip}_1(X)$ is the family of real Lipschitz functions 
with Lipschitz constant less or equal than $1$.

When $d_X$ is bounded,
then $\PP_p(X)=\PP(X)$ 
and every metric 
$W_{p,d_X}$ metrizes the narrow topology
in $\PP(X)$.
\nc 
\begin{oss}\label{equiv of Borel for polish}
    So far, given $(X,d_X) $ 
    a complete and separable metric space, we introduced the Polish space $\PP(X)$ 
    (for which we only care about the narrow topology $\tau_N$) and the Wasserstein spaces $\PP_p(X)$ for any $p\geq 1$, inducing the Wasserstein
    topology $\tau_p = \tau_{W_p}$. 
    When $d_X$ is unbounded,
    it is well-known that $\tau_p$ is strictly 
    finer than the restriction of 
     $\tau_N$ to $\PP_p(X)$. 
     However, since 
     $\PP_p(X)$ is a Borel subset of $\PP(X)$
     and thus a Lusin space
     with respect to the narrow topology,
     and the $p$-Wasserstein topology is finer, 
     then thanks to \cite[Corollary 2, pp. 101]{schwariz1973radon} 
     the induced Borel $\sigma$-algebras $\mathcal{B}(\PP_p(X))$ coincide. 
\end{oss}
\subsubsection{Laws of random probability measures}\label{subsec: topology over random measures}

Given a Polish space $(Y,\tau_Y)$, the main objects of our study will be the so-called \textit{laws of random probability measures}, or just \textit{random measures}, 
$M \in \PP(\PP(Y))$. 
Since $\PP(Y)$ is a Polish space,
we observe that \nc 
\begin{itemize}
    \item over $\PP(\PP(Y))$, we consider the narrow topology, 
    induced by 
    the underlying (Polish) narrow topology over $\PP(Y)$. 
    If $\delta_Y$ is 
    any \emph{bounded} metric
    inducing $\tau_Y$,
    the narrow topology on $\PP(\PP(Y))$
    is induced by the 
    bounded Wasserstein metric
    $\widehat{\mathcal{W}}_1:= 
    W_{1,W_{1,\delta}}$. 
    \item 
    When 
    a (complete, separable) metric $d_Y$ 
    is assigned on the underlying space $Y$, 
    then we can endow
    $\PP_p(\PP_q(Y))$, with $p,q\geq 1$,
    with the 
    \emph{Wasserstein on Wasserstein} 
    metric 
    $\mathcal{W}_{p,q}:= W_{p,W_{q,d_Y}}$.
    We will only 
    deal with 
    the case $p=q$, and we use the notation $\mathcal{W}_p = \mathcal{W}_{p,p}
    =W_{p,W_{p,d_Y}}$.
\end{itemize}
Thanks to Remark \ref{equiv of Borel for polish}, it is equivalent to consider a random measure $M\in \PP(\PP(Y))$ concentrated over $\PP_p(Y)$ or a random measure $M\in\PP(\PP_p(Y))$, since the Borel $\sigma$-algebras induced on $\PP_p(Y)$ by the narrow topology 
concides with 
the Borel $\sigma$-algebra 
induced by the $L^p$-Wasserstein 
metric. In particular, we can always work with random measures $M\in \PP(\PP(Y))$, possibly specifying later that it is concentrated over measures with finite $p$-moments.

It is worth highlighting a structure that will appear often in the following: given two Polish spaces $Y,Z$ and a Borel map $f:Y\to Z$, we can define
\begin{equation}\label{push forward of push forward}
\begin{aligned}
    F:= f_\sharp  : \PP(Y) &\to \PP(Z)
    \\
    F_\sharp=f_{\sharp\sharp} : \PP(\PP(Y)) &\to \PP(\PP(Z)),
\end{aligned}
\end{equation}
that are Borel maps with the topologies we considered (see Proposition \ref{meas of push forward}). A property of this nested push-forward is the following: for all $M\in \PP(\PP(Y))$ and $g:Z\to [0,+\infty]$
\begin{equation}\label{property of push for of push for}
    \int_{\PP(Z)} \int_Z g(z) \,\d\nu(z) \,\d F_\sharp M(\nu) = \int_{\PP(Y)} \int_{Y} g\circ f(y) \,\d\mu(y) \,\d M(\mu).
\end{equation}

\begin{oss}\label{rem: tilde{M}}
    Any random measure $M \in \PP(\PP(Y))$ can also be identified through the measure $\widetilde{M}:= \int \mu\otimes \delta_\mu  \,\d M(\mu)\in \PP(Y \times \PP(Y))$, i.e.~the only measure for which, for all 
    Borel $g:Y \times \PP(Y) \to [0,+\infty)$ bounded it holds
    \begin{equation}
        \int_{Y\times \PP(Y)} g(y,\mu) \,\d\widetilde{M}(y,\mu) = \int_{\PP(Y)} \int_Y g(y,\mu) \,\d\mu(y) \,\d M(\mu).
    \end{equation}
\end{oss}

\begin{oss}\label{measures are Polish/Lusin}
The narrow topology is Polish 
when restricted either to $\PP(Y)$ or $\mathcal{M}_+(Y)$ (see Section \ref{OT and Wass} and Lemma \ref{weak conv is metrizable}). On the other hand, the narrow topology over $\mathcal{M}(Y)$ and $\mathcal{M}(Y;\R^d)$ cannot be metrized. However, the narrow topology is metrizable when restricted to sets of measures with bounded total variation, in analogy with the weak (and $\text{weak}^*$) topology on Banach spaces (see \cite[Chapter 3]{brezis2011functional}). Thanks to this fact, we obtain that the space $\mathcal{M}(Y;\R^d)$ is a Lusin space (see Appendix \ref{appendix souslin}), which will allow us to recover some useful properties about Borel measurability in these spaces. 
\end{oss}

\subsection{Curves of measures}
\subsubsection{Continuous curves}\label{subsub cont curve in prob space}
Here we specialize the notation of \ref{subsub space of cont curves} to the Polish spaces $\PP(Y)$ or $\PP(\PP(Y))$, where $(Y,\tau)$ is a Polish space. 

In both the spaces $C_T(\PP(Y))$ and $C_T(\PP(\PP(Y)))$, the compact-open topology is Polish. Indeed, using the distances described above, the compact-open topology over $C_T(\PP(Y))$ is induced by the sup distance $D_{\hat{W}_1}$, where $\hat{W}_1:= W_{1,\hat{d}}$ and $\hat{d}$ is a bounded distance inducing $\tau$ (notice that $\hat{W}_1 \neq W_1 \wedge 1$). Similarly, for $C_T(\PP(\PP(Y)))$, its compact-open topology is induced by the sup distance $D_{\hat{\mathcal{W}}_1}$, where $\hat{\mathcal{W}}_1 := W_{1,\hat{W}_1}$.

These distances are just a possible choice for inducing such topologies, but these specific choices will be useful for our purposes.

Following Remark \ref{rem: tilde{M}}, it is important to notice the identification between any curve of random measures $(M_t)_{t\in[0,T]}\in C_T(\PP(\PP(X)))$ with the measure $\widetilde{M}_t\otimes \,\d t \in \mathcal{M}_+([0,T]\times X \times \PP(X))$, defined through the integration formula 
\begin{equation}\label{eq: tilde M otimes dt}
    \int f(t,x,\mu) \,\d\big(\widetilde{M}_t\otimes \,\d t\big)(t,x,\mu) = \int_0^T \int_{\PP(X)}\int_X f(t,x,\mu) \,\d\mu(x) \,\d M_t(\mu)\,\d t,\end{equation}
for all $f:[0,T]\times X \times \PP(X) \to [0,1]$ Borel measurable. Similarly, we introduce the measure $M_t\otimes \,\d t \in \mathcal{M}_+([0,T]\times \PP(X))$ as 
\begin{equation}\label{eq: M otimes dt}
    \int g(t,\mu) \,\d\big(M_t\otimes \,\d t\big)(t,\mu) = \int_0^T \int_{\PP(X)}g(t,\mu)  \,\d M_t(\mu)\,\d t,
\end{equation}
for all $g:[0,T] \times \PP(X) \to [0,1]$ Borel measurable.

\subsubsection{Topology over $\PP(C_T(Y))$}\label{subsub top over P(C_T)}
Given $(Y,\tau)$ a Polish space, the natural topology over $\PP(C_T(Y))$ is the narrow topology with the compact-open as ground topology. With this topology, the space $\PP(C_T(Y))$ is Polish as well, and a convenient metric that induces it is given by 
$W_{1,D_{\hat{d}}}$,
where $\hat{d}$ is any bounded distance inducing $\tau$ (in particular, again we can take $\hat{d}= d\wedge 1$ where $d$ is any distance on $Y$ inducing $\tau$).

\subsubsection{Absolutely continuous curves}\label{subsub ac curve in prob space}
Let $(X,d)$ be a complete and separable metric space. Consider the metric spaces $(\PP_q(X),W_q)$ and $(\PP_p(\PP_q(X)),\mathcal{W}_{p,q})$, where $p,q\in [1,+\infty)$. To avoid confusion in the following, we restate the definition of the action of a curve in these cases.

\begin{df}\label{energy curve of meas}
    Let $p,q\in [1,+\infty)$ and $\boldsymbol{\mu}=(\mu_t)_{t\in[0,T]}\in C_T(\PP_q(X))$. The $p$-action of $\boldsymbol{\mu}$ is defined as 
    \begin{equation}
        \mathcal{A}_{p,q}(\boldsymbol{\mu}):= \begin{cases}
            \int_0^T |\dot{\boldsymbol{\mu}}|^p_{W_q}(t) \,\d t \quad & \text{ if }\boldsymbol{\mu} \in AC_T(\PP_q(X))
            \\
            +\infty \ & \text{ otherwise}
        \end{cases}
    \end{equation}
    In the case $p=q$, we simply denote $\mathcal{A}_p=\mathcal{A}_{p,p}$.
\end{df}

\begin{df}[Absolutely continuous curves of random measures]
    Let $p,q,r \in [1,+\infty)$ and $\boldsymbol{M} = (M_t)_{t\in [0,T]} \in C_T(\PP_p(\PP_q(X)))$. Its $r$-action is then defined as 
    \begin{equation}
        \boldsymbol{A}_{r,p,q}(\boldsymbol{M}) :=
        \begin{cases}
            \int_0^T |\dot{\boldsymbol{M}}|^r_{\mathcal{W}_{p,q}}(t) \,\d t \quad & \text{ if }\boldsymbol{M} \in AC_T(\PP_p(\PP_q(X)))
            \\
            +\infty \ & \text{ otherwise}
        \end{cases}
    \end{equation}
    In the case $p=q=r$, we simply denote $\boldsymbol{A}_p = \boldsymbol{A}_{r,p,q}$.
\end{df}
In this paper, we will always deal with the case $p=q=r$.
A fundamental theorem for our analysis is a lifting result due to \cite{lisini2007characterization}.

\begin{teorema}\label{lisini lifting}
    Let $(X,d)$ be a complete and separable metric space and $p\in (1,+\infty)$. Let $\boldsymbol{\mu}=(\mu_t)_{t\in [0,T]} \in AC^p_T(\PP_p(X))$. Then, there exists a lifting $\lambda\in \PP(C_T(X))$ such that 
    \begin{enumerate}
        \item $(e_t)_{\sharp}\lambda = \mu_t$ for any $t\in [0,T]$, where $e_t (\boldsymbol{x}) = \boldsymbol{x}(t)$ for any $\boldsymbol{x}\in C_T(X)$;
        \item $\lambda$ is concentrated over $AC_T^p(X)$ and
        \begin{equation}\label{minimal lifting}
            \int a_p(\boldsymbol{x}) \,\d\lambda (\boldsymbol{x}) = \mathcal{A}_p(\boldsymbol{\mu}) < +\infty.
        \end{equation}
    \end{enumerate}
    On the other hand, for any $p\in [1,+\infty)$, given $\lambda\in \PP(C_T(X))$ concentrated over absolutely continuous curves, with $(e_0)_\sharp\lambda\in \PP_p(X)$, the curve $\boldsymbol{\mu}:= ((e_t)_\sharp\lambda)_{t\in [0,T]}$ belongs to $AC^p_T(\PP_p(X))$ and it satisfies
    \begin{equation}\label{ineq energy}
        |\dot{\boldsymbol{\mu}}|_{W_p}^p (t) \leq \int |\dot{\boldsymbol{x}}|^p(t) \,\d\lambda(\boldsymbol{x})
    \end{equation}
    for $\mathcal{L}^1$-a.e. $t\in (0,T)$.
\end{teorema}

\begin{oss}\label{remark minimal}  
    Putting together the formulas \eqref{minimal lifting} and \eqref{ineq energy}, we can see that a lifting $\lambda$ as in the first part of the Theorem, is of minimal energy among all the possible lifting satisfying only the first condition. 
\end{oss}

\subsubsection{Continuity equation over \texorpdfstring{$\R^d$}{}}\label{subsub CE prob over R^d}

Given a Borel time-dependent vector field $v:[0,T]\times \R^d \to \R^d$, we say that a curve of measure $\boldsymbol{\mu}=(\mu_t)_{t\in [0,T]} \in C_T(\PP(\R^d))$ solves $\partial_t\mu_t +\operatorname{div}(v_t\mu_t)=0$ if $\int_0^T\int |v(t,x)|\,\d\mu_t(x) \,\d t <+\infty$ and for all $\psi\in C_c^{1}((0,T)\times \R^d)$ it holds
\begin{equation}\label{fin dim ce}
    \int_0^T \int_{\R^d} \partial_t\psi(t,x) + \nabla\psi(t,x)\cdot v(t,x) \,\d\mu_t(x) \,\d t = 0.
\end{equation}
Thanks to \cite[Lemma 8.1.2]{ambrosio2005gradient}, it is not restrictive to assume that $t\mapsto \mu_t$ is continuous.

If the vector field is smooth enough to have that the ordinary differential equation given by 
\begin{equation}\label{eq: ODE system prel}
    \begin{cases}
        \dot{\boldsymbol{x}}(t) = v(t,\boldsymbol{x}(t))
        \\
        \boldsymbol{x}(0) = \bar{x}
    \end{cases}
\end{equation}
admits a unique solution $[0,T] \ni t\mapsto \boldsymbol{X}_t(\bar{x})$ for any $\bar{x}\in \R^d$, it holds that the unique solution of the continuity equation starting from $\mu_0\in \PP(\R^d)$ is 
\begin{equation}
    \mu_t := (\boldsymbol{X}_t)_\sharp\mu_0.
\end{equation}

A similar scheme is still valid in a non-smooth setting. The so-called \textit{finite dimensional superposition principle}, highlights it in a completely non-smooth setting (see e.g. \cite{ambcri06}).

\begin{teorema}[Finite dimensional superposition principle]\label{finite dim superposition principle}
    Let $v:[0,T]\times \R^d \to \R^d$ be a Borel vector field and $\boldsymbol{\mu}\in C_T(\PP(\R^d))$ be satisfying \eqref{fin dim ce} and $\int_0^T \int |v| \,\d\mu_t \,\d t<+\infty$. Then there exists a \textit{superposition solution} $\lambda\in \PP(C_T(\R^d))$ satisfying
    \begin{enumerate}
        \item $(e_t)_\sharp\lambda= \mu_t$ for all $t\in [0,T]$;
        \item $\lambda(AC_T(\R^d)) = 1$ and $\lambda$-a.e. $\boldsymbol{x}\in AC_T(\R^d)$ solves the integral formulation of \eqref{eq: ODE system prel}
        \[\boldsymbol{x}(t) = \boldsymbol{x}(0) + \int_0^t v(s,\boldsymbol{x}(s)) ds \quad \forall t\in [0,T].\]
    \end{enumerate}
    Conversely, given $\lambda\in \PP(C_T(\R^d))$ satisfying (2) and 
    \[ \int \int_0^T |v(t,\boldsymbol{x}(t))| \,\d t \,\d\lambda(\boldsymbol{x})<+\infty,\] 
    then the curve of measures $\mu_t^{\lambda}:= (e_t)_{\sharp}\lambda$ solves the continuity equation.
\end{teorema}

The structure of this result is very similar to the one of the lifting described in Theorem \ref{lisini lifting}. The main difference is that the first result cares in selecting a minimal energy lifting, while the second takes in account the leading velocity vector field of the evolution. This finite dimensional superposition principle will play a fundamental role in the proof of Theorem \ref{main theorem}.

\subsection{Recap}\label{recap} 
Given a Polish space $(Y,\tau)$, we fixed a Polish topology on the following spaces:
\begin{itemize}
    \item[(i)] $C_T(Y)$, see \ref{subsub space of cont curves};
    \item[(ii)] $\PP(Y)$, $\mathcal{M}_+(Y)$ and $\mathcal{M}(Y,\R^d)$, see \ref{subsub space of measures};
    \item[(iii)] $C_T(\PP(Y))$ and $C_T(\PP(\PP(Y)))$, see \ref{subsub cont curve in prob space};
    \item[(iv)] $\PP(C_T(Y))$, see \ref{subsub top over P(C_T)}.
\end{itemize}

Dealing with a specific metric structure $(X,d)$, so that we care about the distance itself and not just its induced topology, we defined some subspaces of the spaces mentioned above. In some cases, they can be seen using a more general notation. 

\begin{df}\label{finite energies prob}
    Let $Y$ be a Polish space and $F:Y\to[0,+\infty] $ a Borel function. Define the set 
    \begin{equation}
        \mathcal{P}_F(Y):= \left\{\mu\in \PP(Y) \ : \ \int_Y F(y) \,\d\mu(y) <+\infty\right\}.
    \end{equation}
\end{df}

Notice that, thanks to Lemma \ref{measurability int g d mu}, $\PP_F(Y)$ is a Borel subset of $\PP(Y)$. When $F$ is lower semicontinuous, we can also say that $\PP_F(Y)$ is an $F_{\sigma}$ set, i.e. it is union of closed sets.

Then, given a metric space $(X,d)$, we have:
\begin{itemize}
    \item[(v)] for any $p\geq 1$, $\PP_p(X) = \PP_F(X)$, with $F(x) := d^{p}(x,\bar{x})$ for some $\bar{x}\in X$. Its natural metric is $W_p$, see \ref{OT and Wass};
    \item[(vi)] for any $p,q\geq 1$, $\PP_p(\PP_q(X)) = \PP_F(\PP(X))\subset \PP(\PP(X))$, with $F(\mu) = W_q^p(\mu,\delta_{\bar{x}})$, for some $\bar{x}\in X$. Its natural metric is $\mathcal{W}_{p,q}$, see \ref{OT and Wass};
    \item[(vii)] as in Remark \ref{remark minimal}, an object $\lambda \in \PP(C_T(X))$ satisfying 
    \[\int \int_0^T |\dot{\boldsymbol{x}}|^p(t) \,\d t \,\d\lambda(\boldsymbol{x})<+\infty\]
    can be simply identified by writing $\lambda \in \PP_{a_p}(C_T(X))$, where $a_p$ is the finite energy of a continuous curve, see Definition \ref{ac energies}.
\end{itemize}

Regarding the sets of absolutely continuous curves, notice that thanks to Lemma \ref{meas of AC^p} and Remark \ref{equiv of Borel for polish}, we have: 
\begin{itemize}
    \item[(viii)] $AC_T^p(X) \subset C_T(X)$ is a Borel subset, for $p>1$;
    \item[(ix)] $AC_T^p(\PP_q(X))\subset C_T(\PP(X))$ is a Borel subset, for $q\geq 1$ and $p>1$;
    \item[(x)] $AC_T^r(\PP_p(\PP_q(X))) \subset C_T(\PP(\PP(X)))$ is a Borel subset, for $p,q\geq 1$ and $r>1$.
\end{itemize}
\section{Nested lifting for an absoltuely continuous curve of random measures}\label{metric structure}

Let $(X,d)$ be a complete and separable metric space, $\overline{x} \in X$ and $p\geq 1$. The goal of this section is to study the structure of absolutely continuous curves of random measures valued in $\PP_p(\PP_p(X))$, i.e. $\boldsymbol{M} = (M_t)_{t\in[0,T]} \in AC_T^p(\PP_p(\PP_p(X)))$. Referring to the previous section for all the topological and metric notions, let us fix the notation that will be used for the rest of this section:
\begin{itemize}
    \item a generic element of $\PP_p(X)$ will be indicated as $\mu$;
    \item a generic element of $\PP_p(\PP_p(X))$ will be indicated as $M$;
    \item a generic element of $C_T(X)$ will be indicated as $\boldsymbol{x}=(x_t)_{t\in [0,T]}$;
    \item a generic element of $AC_T^p(\PP_p(X))$ will be indicated as $\boldsymbol{\mu}:=(\mu_t)_{t\in[0,T]}$; 
    \item a generic element of $AC_T^p(\PP_p(\PP_p(X)))$ will be indicated as $\boldsymbol{M}:=(M_t)_{t\in[0,T]}$; 
    \item following the notation of Definition \ref{finite energies prob}, a generic element of $\PP_{\overline{a}_p}\big(C_T(X)\big)$ will be indicated as $\lambda$, where
    \begin{equation}\label{eq: ac energy with initial moment}
        \overline{a}_p(\boldsymbol{x}):= d^p(\overline{x},x_0) +  a_p(\boldsymbol{x})
    \end{equation}
    for a fixed $\overline{x}\in X$.
    Note that $\PP_{\overline{a}_p}\big(C_T(X)\big) \subset \PP_{a_p}(C_T(X))$, because in addition we are asking that the marginal at time $t=0$ (and then every marginal) is in $\PP_p(X)$. We also introduce the notation 
    \begin{equation}\label{minimal optimal lifting}
        \lambda \in \PP_{\overline{a}_p}^{\operatorname{min}}(C_T(X)) \subset \PP_{\overline{a}_p}(C_T(X))
    \end{equation}
    for the liftings that, in addition, satisfy the minimality condition \eqref{minimal lifting} too, i.e. 
    \[\PP_{\overline{a}_p}^{\operatorname{min}}(C_T(X)):= \left\{\lambda 
    \in \PP_{\overline{a}_p}(C_T(X)) \ : \ \int a_p(\boldsymbol{x}) d\lambda(\boldsymbol{x}) = \int_0^T|\dot{\boldsymbol{\lambda}}|^p(t) dt\right\},\]
    where $\boldsymbol{\lambda} = \big( (e_t)_\sharp \lambda\big)_{t\in[0,T]}$.
\end{itemize}

Using again Definition \ref{finite energies prob}, we also introduce two sets of probability measures, that are, respectively, Borel subsets of $\PP(C_T(\PP(X)))$ and $\PP(\PP(C_T(X)))$:
\begin{itemize}
    \item $\Lambda \in \PP_{\bar{\mathcal{A}}_p}\big(
    C_T(\PP(X))\big)$, where 
    \begin{equation}\label{eq: ac energy on measures with initial moment}
        \bar{\mathcal{A}}_p(\boldsymbol{\mu}) := W_p^p(\mu_0, \delta_{\overline{x}}) + \mathcal{A}_p(\boldsymbol{\mu}).
    \end{equation}
    In particular, each $\Lambda \in \PP_{\bar{\mathcal{A}}_p}\big(C_T(\PP(X))\big)$ is concentrated on $AC^p_T(\PP_p(X))$ and $(\mathfrak{e}_t)_\sharp \Lambda \in \PP_p(\PP_p(X))$ for all $t\in[0,T]$;
    \item $\mathfrak{L}\in \PP_{\overline{\mathfrak{A}}_p}\big(\PP(C_T(X))\big)$ where 
    \begin{equation}\label{eq: def of mathfrak A_p}
        \overline{\mathfrak{A}}_p (\lambda):= \int \overline{a}_p d\lambda =  \int d^p(\overline{x},x_0) d\lambda(\boldsymbol{x}) + \int a_p(\boldsymbol{x})d\lambda(\boldsymbol{x}).
    \end{equation}
    Notice that each $\mathfrak{L}\in \PP_{\overline{\mathfrak{A}}_p}\big(\PP(C_T(X))\big)$ is concentrated on the set $\PP_{\overline{a}_p}\big(C_T(X)\big)$.
\end{itemize}

In this section, we prove Theorem \ref{metric superposition}, that links $\boldsymbol{M} \in AC_T^p(\PP_p(\PP_p(X)))$ with $\Lambda \in \PP_{\overline{\mathcal{A}}_p}(C_T(\PP(X)))$ and $\mathfrak{L}\in \PP_{\overline{\mathfrak{A}}_p}\big(\PP(C_T(X))\big)$. To better understand the strategy of the proof, in the next subsection we expose it in an easier scenario.

\subsection{Couplings}\label{couplings}

We discuss briefly how, given $M_0,M_1 \in \PP_p(\PP_p(X))$, we can associate to them:
\begin{itemize}
    \item a coupling $\Pi \in \PP_p\big(\PP_p(X)\times\PP_p(X)\big) = \PP_F(\PP(X)\times \PP(X))$, with 
    \[F(\mu,\nu) = W_{p}^p(\mu,\delta_{\overline{x}}) + W_{p}^p(\nu,\delta_{\overline{x}}),\]
    i.e. it is the $p$-Wasserstein space built over the product metric space $\PP_p(X)\times\PP_p(X)$. We say that $\Pi\in \Gamma(M_0,M_1)$ if its marginals are $M_0$ and $M_1$;
    \item a random coupling, i.e. a probability measure $\mathfrak{P}\in\PP_p\big(\PP_p(X\times X)\big) = \PP_F\big(\PP(X\times X)\big)$, with \[F(\pi) := \int \Big(d^p(x,\overline{x}) + d^p(y,\overline{x}) \Big) d\pi(x,y).\]
    We say that $\mathfrak{P}\in \mathrm{R}\Gamma(M_0,M_1)$ if $P^1_\sharp \mathfrak{P} = M_0$ and $P^2_\sharp \mathfrak{P} = M_1$, where 
    \begin{equation}
        P^i:\PP(X\times X) \to \PP(X), \quad P^i(\pi) = p^i_\sharp \pi.
    \end{equation}
\end{itemize}

The plan $\Pi \in \PP_p\big(\PP_p(X)\times\PP_p(X)\big)$ is simply selected as a $\mathcal{W}_p$-optimal plan between $M_0$ and $M_1$ , i.e. such that its marginals are $M_0$ and $M_1$, and it realizes the distance $\mathcal{W}_p$. 

Given such $\Pi \in \PP_p\big(\PP_p(X)\times\PP_p(X)\big)$, we build $\mathfrak{P}\in \PP_p\big(\PP_p(X \times X)\big)$ by defining a map $Q:\PP_p(X)\times \PP_p(X) \to \PP_p(X\times X)$ that for any pair $\mu,\nu\in \PP_p(X)$ gives (in a measurable way) an optimal plan $Q(\mu,\nu)\in \Gamma_0(\mu,\nu)$. To define such $Q$, let 
\begin{equation}\label{eq: 3.6}
    \begin{aligned}
        P:\PP(X\times X) & \to \PP(X)\times \PP(X)
        \\
        \pi \quad & \mapsto (p^1,p^2)_\sharp \pi.
    \end{aligned}
\end{equation}
This map is continuous, so it is Borel.
Consider the space of optimal couplings 
\begin{equation}
    \mathcal{P}^{\operatorname{opt}}_p (X\times X) := \left\{\pi \in \PP_p(X\times X) \ : \ \int d^p(x,y) d\pi(x,y) = W_p^p(p^1_\sharp \pi,p^2_\sharp \pi)\right\}.
\end{equation}
Notice that the map $P: \mathcal{P}^{\operatorname{opt}}_p(X\times X) \to \PP_p(X) \times \PP_p(X)$ is Borel and surjective, since for any couple $(\mu,\nu) \in \PP_p(X) \times \PP_p(X)$ there exists an optimal coupling. Then thanks to the measurable selection theorem \ref{measurable selection}, there exists a (Souslin-Borel measurable, see Appendix \ref{appendix souslin}) right inverse $P^{-1}$. Then, let $Q:= P^{-1}$ and define 
\begin{equation}
\begin{aligned}
    Q_\sharp : \PP_p\big(\PP_p(X) \times \PP_p(X)\big) & \to \PP_p\big(\PP_p(X\times X)\big)
    \\
    \Pi \quad & \mapsto \quad \mathfrak{P}:= Q_\sharp  \Pi
\end{aligned}
\end{equation}

By construction, each $\mathfrak{P}$ obtained in this way is concentrated on the set of optimal couplings $\PP^{\operatorname{opt}}_p(X\times X)$. Finally, we have that $\mathfrak{P}\in \mathrm{R}\Gamma(M_0,M_1)$. Summing up, we have this result.

\begin{prop}\label{prop: coupling and rand couplings}
    Let $M_0,M_1\in\PP_p(\PP_p(X))$. Then, there exist $\Pi \in \Gamma(M_0,M_1)$ and $\mathfrak{P}\in \mathrm{R}\Gamma(M_0,M_1)$ such that 
    \begin{equation}\label{eq: opt}
        \mathcal{W}_p^p(M_0,M_1) = \int_{\PP(X)\times \PP(X)} \hspace{-0.05cm}W_p^p(\mu,\nu) d\Pi(\mu,\nu) = \int_{\PP(X\times X)}\int_{X\times X} \hspace{-0.05cm}d^p(x,y)\pi(x,y)d\mathfrak{P}(\pi).
    \end{equation}
    Whenever \eqref{eq: opt} is satisfied, we say that $\Pi\in \Gamma_0(M_0,M_1)$ and $\mathfrak{P}\in \mathrm{R}\Gamma_0(M_0,M_1)$.
\end{prop}

This result shows the strategy we will follow to prove our nested superposition principle (both the metric and the differential one) and will play an important role for characterizing the geodesics of $(\PP_p(\PP_p(X)),\mathcal{W}_p)$ for $p>1$ (see §\ref{subsec: geodesics}).

\subsection{From \texorpdfstring{$\operatorname{AC}(\PP(\PP))$}{} to \texorpdfstring{$\PP(\operatorname{AC}(\PP))$}{}}\label{from ac(pp) to P(ac(p))}

A similar strategy can be applied to the case of an absolutely continuous curve $\boldsymbol{M}\in AC_T^p(\PP_p(\PP_p(X)))$.

Notice that we can always associate to $\boldsymbol{M}$ a measure $\Lambda \in \PP_{\bar{\mathcal{A}}_p}(C_T(\PP(X)))$, using Theorem \ref{lisini lifting} with $(Y,d) = (\PP_p(X),W_p)$, here is the specific statement.

\begin{prop}\label{prop: acpp to pacp}
    For any curve $\boldsymbol{M}=(M_t)_{t\in[0,T]} \in AC^p\big(\PP_p(\PP_p(X))\big)$ there exists a lifting $\Lambda \in \mathcal{P}_{\bar{\mathcal{A}}_p}\big( C_T(\mathcal{P}_p(X))\big)$ such that
    \begin{equation}\label{mimimality property of Lambda}
        (\mathfrak{e}_t)_\sharp \Lambda = M_t \
    \text{ and }
    \ \int\mathcal{A}_p(\boldsymbol{\mu}) d\Lambda(\boldsymbol{\mu}) = \int_0^T |\dot{\boldsymbol{M}}|_{\mathcal{W}_p}^p(t) dt<+\infty,
    \end{equation}
    where $\mathfrak{e}_t(\boldsymbol{\mu})=\mu_t$.
\end{prop}

As in the general case of Theorem \ref{lisini lifting}, the measure $\Lambda$ is possibly non-unique. Any possible selection $\Lambda$ will be indicated as $\operatorname{Lift}(\boldsymbol{M})$, i.e. $\operatorname{Lift}(\boldsymbol{M}) := \{\Lambda \in \PP_{\bar{\mathcal{A}}_p}\big( C_T(\mathcal{P}_p(X))\big) \ : \ \eqref{mimimality property of Lambda} \text{ holds}\}$. The Proposition \ref{prop: acpp to pacp} can be restated as: if $\boldsymbol{M}\in AC_T^p(\PP_p(\PP_p(X)))$, then $\operatorname{Lift}(\boldsymbol{M}) \neq \emptyset$.

\subsection{From \texorpdfstring{$\PP(\operatorname{AC}(\PP))$}{} to \texorpdfstring{$\PP(\PP(\operatorname{AC}))$}{}}\label{metric from P(AC(P)) to P(P(AC))}
In this subsection, we want to define a map that associates an element $\mathfrak{L}\in \PP_{\bar{\mathfrak{A}}_p}(\PP(C_T(X)))$ to $\Lambda\in \PP_{\bar{\mathcal{A}}_p}(C_T(\PP(X)))$. First, we need to define the map
\begin{equation}\label{def of E, section 3}
    \begin{aligned}   
        E:\PP\big(C_T(X)\big) & \to C_T(\PP(X))
        \\
        \lambda \quad & \mapsto\big((e_t)_\sharp  \lambda\big)_{t\in[0,T]}
    \end{aligned}
    \end{equation}
\begin{lemma}\label{well posedeness of E}
    For all $\lambda \in \PP(C_T(X))$, it holds $E[\lambda] \in C_T(\PP(X))$ and the map $E$ is continuous. 
    Moreover, $E$ is surjective from $\PP^{\operatorname{min}}_{\overline{a}_p}(C_T(X))$ (see \eqref{minimal optimal lifting}) to $AC_T^p(\PP_p(X))$, i.e.
    \begin{equation}
        E\big(\PP_{\overline{a}_p}(C_T(X))\big) = E\big(\PP_{\overline{a}_p}^{\operatorname{min}}(C_T(X))\big) = AC^p_T(\PP_p(X)).
    \end{equation}
\end{lemma}

\begin{proof}
    Consider the distances $D_{\hat{W}_1}$ and $W_{1,D_{\hat{d}}}$, respectively, on $C_T(\PP(X))$ and $\PP(C_T(X))$ to induce their topologies (see §\ref{recap}), where $\hat{d} := d\wedge 1$ and $\hat{W}_1 := W_{1,\hat{d}}$. We prove that the map $E$ is $1$-Lipschitz with these choices.
    \\
    For all $\lambda \in \PP\big(C_T(X)\big)$ and for all sequence $t_n\in [0,T]$ converging to $t$, it holds
    \[W_{1,\hat{d}}\big((e_t)_\sharp \lambda,(e_{t_n})_\sharp \lambda\big) \leq \int \hat{d}(\gamma_t,\gamma_{t_n}) d\lambda(\gamma) \to 0\]
    as $n\to +\infty$ by the dominated convergence theorem. Considering $\lambda,\rho \in \PP\big(C_T(X)\big)$, we have
    \[D_{W_{1,\hat{d}}}\big(E[\lambda],E[\rho]\big) = \sup_{t\in [0,T]} W_{1,\hat{d}}\big((e_t)_\sharp \lambda,(e_{t})_\sharp \rho\big) \leq W_{1,D_{\hat{d}}}(\lambda,\rho),\]
    because $(e_t)_\sharp $ is a contraction, indeed saying that $\Pi$ is a $W_{1,D_{\hat{d}}}$-optimal coupling between $\lambda$ and $\rho$ we have that 
    \[
    \begin{aligned} W_{1,\hat{d}}\big((e_t)_\sharp \lambda,(e_{t})_\sharp \rho\big) \leq & \int \hat{d}(x^1_t, x^2_t ) d\Pi(\boldsymbol{x}^1,\boldsymbol{x}^2) 
    \leq \int D_{\hat{d}}(\boldsymbol{x}^1,\boldsymbol{x}^2) d\Pi(\boldsymbol{x}^2,\boldsymbol{x}^2) = W_{1,D_{\hat{d}}}(\lambda, \rho).
    \end{aligned}
    \]
    Regarding the second part of the statement, first of all we need to prove that $E$ is well defined: consider $\lambda \in \PP_{\overline{a}_p}(C_T(X))$, then
    \begin{align*}
        W_{p,d}^p((e_t)_{\sharp }\lambda,(e_s)_{\sharp }\lambda ) 
        &\leq 
        \int d^p(x_t,x_s) d\lambda(\boldsymbol{x})
        \leq 
        \int \left(\int_s^t |\dot{\boldsymbol{x}}|(r) dr \right)^p\hspace{-0.2cm} d\lambda(\boldsymbol{x})
        \\
        &\leq 
        |t-s|^{p-1}\int_s^t\int |\dot{\boldsymbol{x}}|^p(r) d\lambda(\boldsymbol{x}) dr.
    \end{align*}
    Moreover, $(e_0)_\sharp \lambda \in \PP_p(X)$. Putting everything together, the curve $\big((e_t)_\sharp \lambda\big)_{t\in[0,T]}$ is in $AC^p_T(\PP_p(X))$, and by the Lebesgue theorem, for a.e. $t\in[0,T]$ it holds
    \[\lim_{h\to 0}\frac{W_{p,d}^p((e_{t+h})_{\sharp }\lambda,(e_t)_{\sharp }\lambda )}{|h|^p} \leq \int |\dot{\boldsymbol{x}}|^p(t) d\lambda(\boldsymbol{x})<+\infty, \]
    and in particular
    \(\mathcal{A}_p(E[\lambda]) \leq \int_0^T a_p(\boldsymbol{x}) d\lambda(\boldsymbol{x}).\)  
    The surjectivity from $\PP_{\overline{a}_p}^{\operatorname{min}}(C_T(X))$ to $AC_T^p(\PP_p(X))$ is implied by Theorem \ref{lisini lifting}.
\end{proof}

The goal is to find a measurable right inverse of the map $E$, in particular we would like to apply Theorem \ref{measurable selection}. Recall that we know that 
$E$ defined in \eqref{def of E, section 3}
is not surjective: for example, consider two distinct points $x_0,x_1\in X$ and define $\mu_t:= (T-t)\delta_{x_0} + t \delta_{x_1}$ for $t\in [0,T]$. It is clear that $\boldsymbol{\mu}= (\mu_t)_{t\in [0,T]} \in C_T(\PP(X))$, but it cannot be lifted to any measure $\PP(C_T(X))$.

The situation is nicer if we restrict $E$ as a map $E:\PP_{\overline{a}_p}^{\operatorname{min}}\big(C_T(X)\big)  \to AC^p_T \big
(\PP_p(X)\big)$, which is surjective thanks to Lemma \ref{well posedeness of E}.

\begin{teorema}\label{metric theorem from Lambda to mathfrak{L}}
    There exists a Souslin-Borel measurable (see Appendix \ref{appendix souslin}) map $G:AC_T^p(\PP_p(X))\to \PP_{\bar{a}_p}^{\operatorname{min}}(C_T(X))$ such that $E\circ G[\boldsymbol{\mu}] = \boldsymbol{\mu}$ for any $\boldsymbol{\mu}\in AC_T^p(\PP_p(X))$, i.e. $G(\boldsymbol{\mu})$ is a lifting of $\boldsymbol{\mu}$. Moreover, for any $\boldsymbol{\mu}\in AC^p_T(\PP_p(X))$ it holds
    \begin{equation}\label{minimal meas selection}
        \int a_p(\boldsymbol{x}) d\big(G[\boldsymbol{\mu}]\big)(\boldsymbol{x}) = \mathcal{A}_p(\boldsymbol{\mu}).
    \end{equation}
    In particular, the following map is well defined: 
    \begin{equation}\label{min energy lifting equation}
    \begin{aligned}
        G_\sharp  : \mathcal{P}_{\bar{\mathcal{A}}_p}\big(C_T(\mathcal{P}(X))\big)& \to \mathcal{P}_{\bar{\mathfrak{A}}_p}\big(\mathcal{P}(C_T(X))\big)
        \\
        \Lambda\quad & \mapsto \quad  \mathfrak{L}:= G_{\sharp }\Lambda.
    \end{aligned}
    \end{equation}
\end{teorema}

\begin{proof}
The subset $\PP_{\overline{a}_p}^{\operatorname{min}}(C_T(X))$ is a Borel subset of $\PP(C_T(X))$, thanks to Lemma \ref{measurability int g d mu}. The same holds for the subset $AC_T^p(\PP_p(X))\subset C_T(\PP(X))$. Then, because $E$ is a surjection, we can apply Theorem \ref{measurable selection} to obtain a Souslin-Borel measurable map $G:AC_T^p(\PP_p(X)) \to \PP_{\overline{a}_p}^{\operatorname{min}}(C_T(X))$ such that $E\circ G[\boldsymbol{\mu}] = \boldsymbol{\mu}$ for all $\boldsymbol{\mu} \in AC_T^p(\PP_p(X))$.
\\
Given $\Lambda\in \PP_{\bar{\mathcal{A}}_p}(C_T(\PP(X)))$, thus concentrated on $AC_T^p(\PP_p(X))$, thanks to Proposition \ref{univ meas of souslin} and Corollary \ref{push-forward by a souslin-borel map}, since $G$ is Souslin-Borel measurable, we have that 
    \(\mathfrak{L}:= G_{\sharp }\Lambda\)
    is a probability measure over Borel sets of $\mathcal{P}(C_T(X))$.
    It remains to show that $\mathfrak{L} \in \PP_{\mathfrak{A}_p}(\mathcal{P}(C_T(X)))$. First of all, given $\overline{x}\in X$, it holds
    \[
    \begin{aligned}\int\int d^p(\overline{x},x_0) d\lambda(\boldsymbol{x})d\mathfrak{L}(\lambda) = & \int\int_X d^p(x,\overline{x})d\mu_0(x)d\Lambda(\boldsymbol{\mu}) 
    =  \int W_p^p(\mu_0,\delta_{\overline{x}}) d\Lambda(\boldsymbol{\mu})<+\infty.
    \end{aligned}\]
    
    \noindent Then 
    \begin{equation}\label{equality energy}
        \begin{aligned}
        \int \int &a_p(\boldsymbol{x}) d\lambda(\boldsymbol{x}) d\mathfrak{L}(\lambda) = 
        \int_0^T \int\int|\dot{\boldsymbol{x}}|^p d\lambda(\boldsymbol{x})d\mathfrak{L}(\lambda)dt 
        \\
        = & \int_0^T \int\int|\dot{\boldsymbol{x}}|_d^p d\big(G[\boldsymbol{\mu}]\big)(\boldsymbol{x})d\Lambda(\boldsymbol{\mu})dt
        =  \int_0^T \int |\dot{\boldsymbol{\mu}}|^p_{W_{p}}d\Lambda(\boldsymbol{\mu}) dt <+\infty,
    \end{aligned}
    \end{equation}
    where the last equality follows from $G[\boldsymbol{\mu}] \in \PP_{\overline{a}_p}^{\operatorname{min}}(C_T(X))$ for all $\boldsymbol{\mu}\in AC_T^p(\PP_p(X))$.
\end{proof}

\begin{oss}
    Looking closely at the previous proof, the last step highlights why it is important to invert the map $E$ from the domain $\PP^{\operatorname{opt}}_{\bar{a}_p}(C_T(X))$, instead of $\PP_{\bar{a}_p}(C_T(X))$, otherwise, we could not conclude that the measure $\mathfrak{L}$ belongs to $\PP_{\bar{\mathfrak{A}}_p}(\PP(C_T(X)))$. 
\end{oss}

\subsection{From \texorpdfstring{$\PP(\PP(\operatorname{AC}))$}{} to \texorpdfstring{$\operatorname{AC}(\PP(\PP))$}{}}\label{section from P(P(AC) to AC(P(P)) metric}
We conclude our construction by discussing the natural projection from $\mathfrak{L}\in \PP_{\overline{\mathfrak{A}}_p}(\PP(C_T(X)))$ to $AC_T^p(\PP_p(\PP_p(X)))$. We use the nested push-forward described in \eqref{push forward of push forward} using the evaluation map. For any $t \in [0,T]$, we define the maps 
\begin{equation}\label{eq: ppac to acpp}
    \begin{aligned}
        E_t: \PP\big(C_T(X)\big) & \to \PP(X) \qquad \mathfrak{E}: \PP(\PP(C_T(X))) \to C_T(\PP(\PP(X)))
        \\
        \lambda \quad & \mapsto (e_t)_\sharp \lambda, \hspace{2.85cm} \mathfrak{L}  \mapsto \big((E_t)_\sharp \mathfrak{L}\big)_{t\in[0,T]}. 
    \end{aligned}
\end{equation}

\begin{prop}\label{well def from L to M}
    Let $\mathfrak{L}\in \PP_{\mathcal{A}_p}\big(\PP(C_T(X))\big)$ and define $M_t:= (E_t)_\sharp \mathfrak{L}$ for any $t\in [0,T]$. Then, $\boldsymbol{M}= (M_t)_{t\in [0,T]} \in AC^p_T\big(\PP_p(\PP_p(X))\big)$ and it holds
    \begin{equation}
         |\dot{M}|_{\mathcal{W}_p}^p(t) dt \leq \int\int |\dot{\boldsymbol{x}}|^p(t) d\lambda(\boldsymbol{x}) d\mathfrak{L}(\lambda) <+\infty
    \end{equation}
    for a.e. $t\in (0,T)$. In particular
    \begin{equation}
        \int_0^T |\dot{M}|_{\mathcal{W}_p}^p(t) dt \leq \int\int a_p(\boldsymbol{x}) d\lambda(\boldsymbol{x}) d\mathfrak{L}(\lambda) <+\infty.
    \end{equation}
\end{prop}

\begin{proof}
    First, notice that $\boldsymbol{M} \in C_T(\PP(\PP(X)))$, since $\mathfrak{L} \in \PP(\PP(C_T(X)))$. Now, we have to prove that $M_0 \in \PP_p(\PP_p(X))$ and $\int_0^T|\dot{M}|_{\mathcal{W}_p}^p(t) dt <+\infty$.
    First of all,
    \[
    \int W_{p}^p(\mu,\delta_{\overline{x}}) dM_0(\mu) =\int W_{p}^p((e_0)_\sharp \lambda, \delta_{\overline{x}}) d\mathfrak{L}(\lambda)  \leq 
    \int \int d^p(x_0,\overline{x}) d\lambda(\boldsymbol{x}) d\mathfrak{L}(\lambda) <+\infty, 
    \]
    which implies that $M_0 \in \PP_p(\PP_p(X))$. Then
    \begin{align*}
        \mathcal{W}_p^p(M_t,M_s) \leq & \int_{\PP\times\PP} W_{p}^p(\mu,\nu) d\big((E_t,E_s)_\sharp \mathfrak{L}\big)(\mu,\nu) 
        =  \int W_{p}^p\big((e_t)_\sharp \lambda, (e_s)_\sharp \lambda\big) d\mathfrak{L}(\lambda)
        \\
        \leq &
        \int \int d(x_t,x_s)^pd\lambda(\boldsymbol{x})d\mathfrak{L}(\lambda) 
        \leq 
        |t-s|^{p-1} \int_s^t \int \int |\dot{\boldsymbol{x}}|^p_d(r) d\lambda(\boldsymbol{x}) d\mathfrak{L}(\lambda) dr,
    \end{align*}
    where we used $d(x_t, x_s) \leq \int_s^t|\dot{\boldsymbol{x}}|_d(r) dr$, Holder inequality and Fubini's theorem.
    This implies that $(M_t)_{t\in[0,T]}$ is absolutely continuous, and by Lebesgue theorem it holds
    \[|\dot{\boldsymbol{M}}|_{\mathcal{W}_p}^p(t)\leq \int\int |\dot{\boldsymbol{x}}|_{d}^p(t) d\lambda(\boldsymbol{x}) d\mathfrak{L}(\lambda)\quad \text{ for a.e. }t\in(0,T).\qedhere\]
\end{proof}

\subsection{Composition of maps}\label{comp of maps}
In this section, we analyze the relations between the (possibly multivalued) map $\operatorname{Lift}$, and the maps $G_\sharp $ and $\mathfrak{E}$.

\begin{prop}\label{3.8}
    The composition $\mathfrak{E}\circ (G_\sharp ) \circ \operatorname{Lift}$, represented by the following diagram
    \begin{center}
    \color{black}
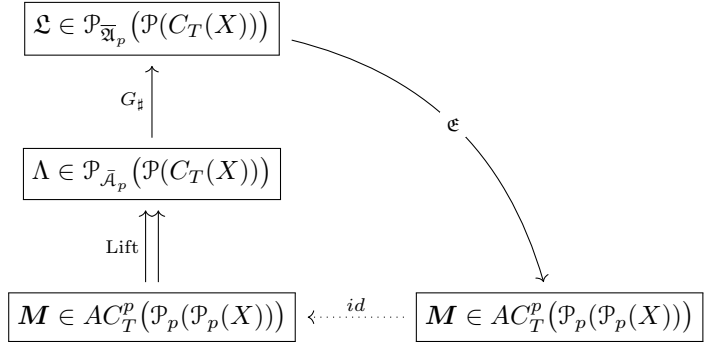
    \begin{tikzcd}[row sep = large ,column sep=large]
        & \boxed{\mathfrak{L} \in \PP_{\overline{\mathfrak{A}}_p} \big(\PP(C_T(X))\big)} \arrow[ddr, bend left, "\mathfrak{E}" description]
        \\
        & \boxed{ \Lambda \in\PP_{\bar{\mathcal{A}}_p}\big(\PP(C_T(X))\big) }\arrow[u, "G_\sharp "]      
        \\
        & \boxed{ \boldsymbol{M} \in AC^p_T\big(\PP_p(\PP_p(X))\big) }   \arrow[u,"\operatorname{Lift}",shift left=1]\arrow[u, shift right=1]
             & \boxed{\boldsymbol{M} \in AC^p_T\big(\PP_p(\PP_p(X))\big)} \arrow[l, dotted, "id"']
    \end{tikzcd}
    \normalcolor
    \end{center}
    is the identity, i.e. $\mathfrak{E}\circ G_\sharp  \circ \operatorname{Lift}   (\boldsymbol{M}) = \boldsymbol{M}$ for any $\boldsymbol{M}\in AC_T^p(\PP_p(\PP_p(X)))$. In other words, given $\boldsymbol{M}\in AC_T^p(\PP_p(\PP_p(X)))$, for any $\Lambda\in\operatorname{Lift}(\boldsymbol{M})$ and defining $\mathfrak{L}:= G_\sharp \Lambda$, it holds $\mathfrak{E}(\mathfrak{L}) = \boldsymbol{M}$.
\end{prop}

\begin{proof}
    Let $\boldsymbol{M} \in AC^p_T\big(\PP_p(\PP_p(X))\big)$ and apply to it the maps following the diagram above. Then for any $F:\PP(X) \to [0,+\infty]$ Borel and for any $t\in[0,T]$, it holds
    \begin{align*}
         \int_{\PP(X)} F(\mu) d\big((E_t)_\sharp \mathfrak{L}\big)(\mu) = & \int F((e_t)_{\sharp }\lambda) d\mathfrak{L}(\lambda) 
         =  \int F \big( (e_t)_\sharp (G(\boldsymbol{\mu}) \big) d\Lambda(\boldsymbol{\mu})
         \\
         = & \int F(\mu_t) d\Lambda(\boldsymbol{\mu}) 
         =  \int_{\PP(X)} F(\mu) dM_t(\mu),
    \end{align*}
    where the second equality in the second line comes from $\mathfrak{L} = G_\sharp \Lambda$, the third one follows from $(e_t)_\sharp \circ G = E_t\circ G = \mathfrak{e}_t$, since $E\circ G = \operatorname{id}$, and the last one is a consequence of $(\mathfrak{e}_t)_\sharp \Lambda = M_t$.
\end{proof}

\begin{prop}\label{starting from Lambda}
    Let $\Lambda \in \PP_{\mathcal{A}_p}\big(C_T(\PP_p(X))$ and define $\mathfrak{L} = G_\sharp \Lambda$, $\boldsymbol{M} = \mathfrak{E}(\mathfrak{L})$ and take any $\Lambda_{\boldsymbol{M}} \in \operatorname{Lift}(\boldsymbol{M})$, according to the following diagram
    \begin{center}
    \color{black}
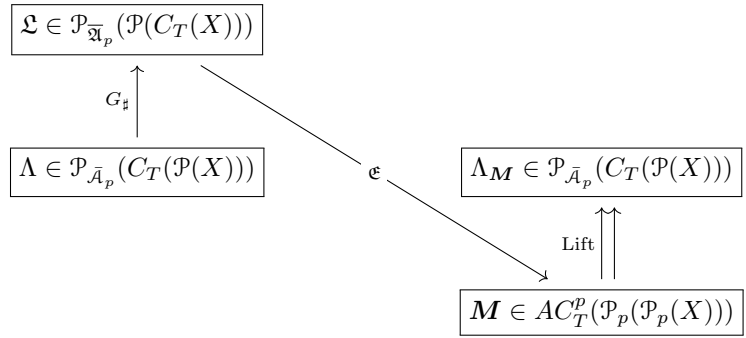
    \begin{tikzcd}[column sep=large,row sep = large]
        & \boxed{ \mathfrak{L} \in \PP_{\overline{\mathfrak{A}}_p}(\PP(C_T(X)))} \arrow[ddr,"\mathfrak{E}" description]
        \\
        & \boxed{\Lambda \in \PP_{\bar{\mathcal{A}}_p}(C_T(\PP(X))) }\arrow[u,"G_\sharp "]
            &[1cm] \boxed{\Lambda_{\boldsymbol{M}} \in \PP_{\bar{\mathcal{A}}_p}(C_T(\PP(X)))}
        \\
        & 
        & \boxed{ \boldsymbol{M} \in AC_T^p(\PP_p(\PP_p(X))) }\arrow[u,"\operatorname{Lift}"]\arrow[u, shift right=2]
    \end{tikzcd}
    \normalcolor
    \end{center}
  Then, both $\Lambda $ and $\Lambda_{\boldsymbol{M}}$ are lifting of $\boldsymbol{M}$, i.e. $(\mathfrak{e}_t)_\sharp \Lambda = (\mathfrak{e}_t)_\sharp \Lambda_{\boldsymbol{M}} = M_t$ for all $t\in [0,T]$, and for a.e. $t\in [0,T]$ it holds
    \begin{equation}
        |\dot{\boldsymbol{M}}|_{\mathcal{W}_p}^p(t) = \int |\dot{\boldsymbol{\mu}}|_{W_{p,d}}^p(t) d\Lambda_{\boldsymbol{M}}(\boldsymbol{\mu}) \leq \int |\dot{\boldsymbol{\mu}}|_{W_{p,d}}^p(t) d\Lambda(\boldsymbol{\mu}).
    \end{equation}
\end{prop}

\begin{proof}
    For all $t\in [0,T]$ and for all $F:\PP_p(X) \to [0,+\infty)$ bounded Borel, we have 
    \begin{align*}
        \int F(\mu) dM_t(\mu) 
        = & 
        \int F(\mu) d\big((E_t)_\sharp  \mathfrak{L}\big)(\mu)  
        =  
        \int F\big((e_t)_\sharp  \lambda \big) d\mathfrak{L}(\lambda) 
        =  
        \int F\big((e_t)_\sharp  \lambda \big) d\big(G_\sharp  \Lambda\big)(\lambda)
        \\
        = &
        \int F\big((e_t)_\sharp  G[\boldsymbol{\mu}] \big) d \Lambda(\boldsymbol{\mu}) 
        = 
        \int F(\mu_t) d \Lambda(\boldsymbol{\mu}) 
        = 
        \int F(\mu) d\big((\mathfrak{e}_t)_\sharp \Lambda\big)(\mu).
    \end{align*}
    For the second part, notice that 
    \[W_p^p(M_t,M_s) \leq \int \mathcal{W}_p^p(\mu_t,\mu_s) d\Lambda'(\boldsymbol{\mu}) \leq |t-s|^{p-1} \int \int_s^t |\dot{\boldsymbol{\mu}}|^p(r) dr d\Lambda'(\boldsymbol{\mu})\]
    for any $\Lambda' \in  \PP_{\mathcal{A}_p}\big(C_T(\PP(X)) $ with marginal $M_t$ at any time $t\in [0,T]$. This implies that 
    \[|\dot{\boldsymbol{M}}|^p(t) \leq \int |\dot{\boldsymbol{\mu}}|^p(t) d\Lambda(\boldsymbol{\mu}) \quad  \text{ and } \quad |\dot{\boldsymbol{M}}|^p(t) \leq \int |\dot{\boldsymbol{\mu}}|^p(t) d\Lambda_{\boldsymbol{M}}(\boldsymbol{\mu})\]
    for a.e. $t\in [0,T]$. Moreover, thanks to Proposition \ref{prop: acpp to pacp}, we have that 
    \[\int \int_0^T |\dot{\boldsymbol{\mu}}|^p(t) dt d\Lambda_{\boldsymbol{M}}(\boldsymbol{\mu)} = \int_0^T |\dot{\boldsymbol{M}}|^p(t) dt.\]
\end{proof}

\begin{prop}\label{3.10}
    Let $\mathfrak{L} \in \PP_{\overline{\mathfrak{A}}_p}\big(\PP_p(C_T(X))\big)$. Define $\boldsymbol{M} = \mathfrak{E}(\mathfrak{L})$, take any $\Lambda \in \operatorname{Lift}(\boldsymbol{M})$ and define $\mathfrak{L}_{\Lambda} = G_\sharp  \Lambda$, according to the following diagram
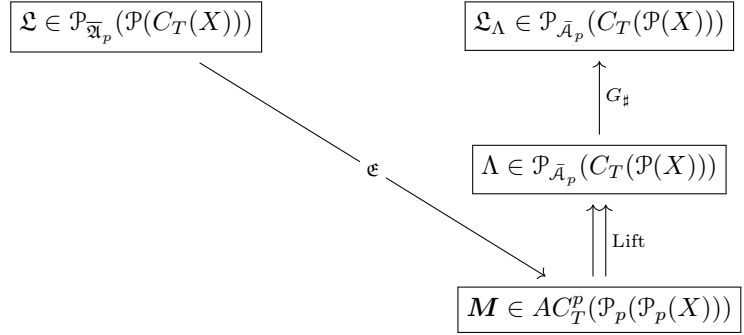
\begin{center}
    \color{black}
    \begin{tikzcd}[column sep=large,row sep = large]
        & \boxed{ \mathfrak{L} \in \PP_{\overline{\mathfrak{A}}_p}(\PP(C_T(X))) }\arrow[ddr,"\mathfrak{E}" description]
            &[1cm] \boxed{ \mathfrak{L}_{\Lambda} \in \PP_{\bar{\mathcal{A}}_p}(C_T(\PP(X)))}
        \\
        &
            & \boxed{\Lambda \in \PP_{\bar{\mathcal{A}}_p}(C_T(\PP(X))) }\arrow[u,"G_\sharp "']
        \\
        & 
            & \boxed{ \boldsymbol{M} \in AC_T^p(\PP_p(\PP_p(X))) }\arrow[u,"\operatorname{Lift}"', shift right=1]\arrow[u, shift left=1]
    \end{tikzcd}
    \normalcolor
    \end{center}
    Then $(E_t)_\sharp \mathfrak{L} = (E_t)_\sharp \mathfrak{L}_{\Lambda} = M_t$ for all $t\in [0,T]$, and moreover
    \begin{equation}
        |\dot{\boldsymbol{M}}|_{\mathcal{W}_p}^p(t) = \int \int |\dot{\boldsymbol{x}}|^p(t)d\lambda(\boldsymbol{x}) d\mathfrak{L}_\Lambda(\lambda) \leq  \int \int |\dot{\boldsymbol{x}}|^p(t) d\lambda(\boldsymbol{x}) d\mathfrak{L}(\lambda),
    \end{equation}
    for a.e. $t\in [0,T]$.
\end{prop}

\begin{proof}
    The first part follows the same strategy above. Regarding the second part, from (\ref{equality energy}) and Proposition \ref{starting from Lambda}, we know that for any $s<t$
    \[\int \int \int_s^t |\dot{\boldsymbol{x}}|_d^p(r) dr d\lambda(\boldsymbol{x}) d\mathfrak{L}_\Lambda(\lambda) = \int \int_s^t |\dot{\boldsymbol{\mu}}|_{W_{p}}^p(r)drd\Lambda(\boldsymbol{\mu}) = \int_s^t |\dot{\boldsymbol{M}}|_{\mathcal{W}_p}^p(r) dr.\]
    Using Lebesgue theorem, we easily obtain that for a.e. $t\in [0,T]$ it holds 
    \[\int \int  |\dot{\boldsymbol{x}}|_d^p(t) d\lambda(\boldsymbol{x}) d\mathfrak{L}_\Lambda(\lambda) = \int |\dot{\boldsymbol{\mu}}|_{W_{p}}^p(t)d\Lambda(\boldsymbol{\mu}) =  |\dot{\boldsymbol{M}}|_{\mathcal{W}_p}^p(t) dt.\]
    Moreover, from Proposition \ref{well def from L to M}, for a.e. $t\in [0,T]$ it holds
    $ |\dot{\boldsymbol{M}}|_{\mathcal{W}_p}^p(t) dt \leq \int \int |\dot{\boldsymbol{x}}|_d^p(t) d\lambda(\boldsymbol{x}) d\mathfrak{L}(\lambda).$
\end{proof}

\subsection{Geodesics of random measures}\label{subsec: geodesics}
In this subsection, we want to give a characterization for the \textit{geodesics} in the space of random measures $(\PP_p(\PP_p(X)), \mathcal{W}_p)$. Assume that $(X,d)$ is a complete, separable and geodesic metric space. It is well known that, under these assumptions on $X$, the space $(\PP_p(X),W_p)$ is geodesic as well, for $p>1$ (see e.g. \cite[Theorem 10.6]{ambrosio2021lectures}). Reiterating it, we already know that $(\PP_p(\PP_p(X)),\mathcal{W}_p)$ is geodesic. 

With the notation introduced in §\ref{subsub: geod}, thanks to Corollary \ref{push-forward by a souslin-borel map}, we define
\begin{equation}
\begin{aligned}
    \hfill\operatorname{GEO}: \PP(&X\times X) \to \PP(C([0,1],X)), \quad \operatorname{GEO}(\pi):= \operatorname{geo}_\sharp \pi, \hfill
    \\
    \hfill \operatorname{GEO}_t & : \PP(X\times X) \to \PP(X), \quad \operatorname{GEO}_t(\pi):= (\operatorname{geo}_t)_\sharp \pi. \hfill
\end{aligned}
\end{equation}

Notice that, since $\operatorname{geo}$ is not Borel measurable in general, we cannot use Proposition \ref{meas of push forward} to conclude that $\operatorname{GEO}$ (and $\operatorname{GEO}_t$) is measurable. One way to deal with this would be to obtain $\operatorname{GEO}$ again using a measurable selection argument. For the sake of the presentation, we do not enter into details and we stick with the previous definitions, that work well whenever the geodesics are uniquely determined so that $\operatorname{geo}$ and $\operatorname{GEO}$ are Borel measurable.

On the other hand, we denote by $\operatorname{geo}_{\PP_p}:\PP_p(X)\times \PP_p(X) \to \operatorname{Geo}(\PP_p(X))$ the map defined in \eqref{eq: geo map} in the geodesic space $\PP_p(X)$. Similarly, $\operatorname{geo}_{t,\PP_p}$ is its evaluation at time $t$, for all $t\in[0,1]$.

First, we show how a geodesic connecting two given random measures $M_0,M_1\in \PP_p(\PP_p(X))$ can be obtained either from $\Pi\in \Gamma_0(M_0,M_1)$ or $\mathfrak{P}\in \mathrm{R}\Gamma_0(M_0,M_1)$ (see §\ref{couplings}).

\begin{lemma}\label{lemma: geodesics through opt couplings}
    Let $M_0,M_1\in \PP_p(\PP_p(X))$. If $\mathfrak{P}\in \mathrm{R}\Gamma_0(M_0,M_1)$, then $M_t^{\mathfrak{P}}:=(\operatorname{GEO}_t)_\sharp \mathfrak{P}$ is a geodesic in $(\PP_p(\PP_p(X)),\mathcal{W}_p)$ connecting $M_0$ and $M_1$. 
    \\
    If $\Pi\in \Gamma_0(M_0,M_1)$, then $M_t^{\Pi}:= (\operatorname{geo}_{t,\PP_p})_\sharp \Pi$ is a geodesic in $(\PP_p(\PP_p(X)),\mathcal{W}_p)$ connecting $M_0$ and $M_1$.
\end{lemma}
\begin{proof}
    Regarding the first claim, for all $0\leq s<t \leq 1$, it holds
    \begin{align*}
        \mathcal{W}_p^p & (M_s^{\mathfrak{P}},M_t^{\mathfrak{P}}) \leq  \int \int d^p(\operatorname{geo}_s(x,y), \operatorname{geo}_t(x,y)) d\pi(x,y)d\mathfrak{P}(\pi) 
        \\
        = & (t-s)^p \int \int d^p(x,y)d\pi(x,y)d\mathfrak{P}(\pi) = (t-s)^p \mathcal{W}_p^p(M_0,M_1).
    \end{align*}
    Similarly for the second claim:
    \begin{align*}
        \mathcal{W}_p^p & (M_s^{\Pi}, M_t^{\Pi}) \leq  \int W_p^p (\operatorname{geo}_{s,\PP}(\mu,\nu),\operatorname{geo}_{t,\PP}(\mu,\nu)) d\Pi(\mu,\nu) 
        \\
        = & (t-s)^p \int  W_p^p(\mu,\nu)d\Pi(\mu,\nu) = (t-s)^p \mathcal{W}_p^p(M_0,M_1).
    \end{align*}
\end{proof}

The next proposition completely characterizes geodesics in terms of their liftings $\mathfrak{L}\in \PP(\PP(C([0,1],X)))$ and $\Lambda \in \PP(C([0,1],\PP(X)))$.

\begin{prop}\label{prop: char geo}
    Let $\boldsymbol{M} = (M_t)_{t\in[0,1]} \in C([0,1],\PP_p(\PP_p(X)))$. The following are equivalent
    \begin{enumerate}
        \item $\boldsymbol{M}$ is a geodesic in $(\PP_p(\PP_p(X)),\mathcal{W}_p)$;
        \item there exists $\mathfrak{L}\in \PP(\PP(C([0,1],X)))$ lifting of $\boldsymbol{M}$, concentrated over $\lambda \in \PP(C([0,1],X))$ that are, in turn, supported over $\operatorname{Geo}(X)$, and such that $(E_{0,1})_\sharp \mathfrak{L} \in \mathrm{R}\Gamma_0(M_0,M_1)$, where 
        \begin{equation}
            E_{0,1}: \PP(C([0,1],X)) \to \PP(X \times X), \quad E_{0,1}(\lambda) := (e_0,e_1)_\sharp \lambda;
        \end{equation}
        \item there exists $\Lambda \in \PP(C([0,1],\PP(X)))$ lifting of $\boldsymbol{M}$ that is supported on $\operatorname{Geo}(\PP_p(X))$ and such that $(\mathfrak{e}_0,\mathfrak{e}_1)_\sharp \Lambda \in \Gamma_0(M_0,M_1)$.
    \end{enumerate}
\end{prop}
\begin{proof}
    (1)$\implies$(2): $\boldsymbol{M}$ being a geodesic implies that $\boldsymbol{M}\in AC_1^p(\PP_p(\PP_p(X)))$. Then, consider any $\mathfrak{L}\in \PP(\PP(C([0,1],X)))$ built as in Proposition \ref{3.10}, and we show it shares the wanted properties. By construction, it is a lifting of $\boldsymbol{M}$. Regarding the geodesic property, it holds
    \begin{align*}\mathcal{W}_p^p (M_0,M_1) \leq & \int \int d^p(x_0,x_1)d\lambda(\boldsymbol{x}) d\mathfrak{L}(\lambda) \leq \int \int \left( \int_0^1 |\dot{\boldsymbol{x}}|(r)dr\right)^pd\lambda(\boldsymbol{x}) d\mathfrak{L}(\lambda)
    \\ \leq &\int \int \int_0^1 |\dot{\boldsymbol{x}}|^p(r)dr d\lambda(\boldsymbol{x}) d\mathfrak{L}(\lambda) = \int_0^1 |\dot{\boldsymbol{M}}|_{\mathcal{W}_p}^p (r) dr = \mathcal{W}_p^p(M_0,M_1),
    \end{align*}
    where the last equality follows from the fact that $\boldsymbol{M}$ is a geodesic and the second last one from Proposition \ref{3.10}. The previous computation forces all inequalities to be equalities. In particular, the first implies that $(E_{0,1})_\sharp \mathfrak{L}\in \mathrm{R}\Gamma_0(M_0,M_1)$, while the second and the third one imply, respectively, that for $\mathfrak{L}$-a.e. $\lambda$ and for $\lambda$-a.e. $\boldsymbol{x}$, it holds 
    \[d(x_0,x_1) = \int_0^1 |\dot{\boldsymbol{x}}|(r)dr \quad \text{ and } \quad |\dot{\boldsymbol{x}}| \text{ is constant},\]
    that together imply that $\boldsymbol{x}\in \operatorname{Geo}(X)$.

    (2)$\implies$(3): define $\Lambda:= E_\sharp \mathfrak{L}$. It is immediate to verify that, if $\lambda \in \PP(C([0,1],X))$ is supported over $\operatorname{Geo}(X)$ and $(e_0,e_1)_\sharp \lambda$ is optimal, then $\boldsymbol{\mu}:= E(\lambda)\in \operatorname{Geo}(\PP_p(X))$. We conclude observing $(\mathfrak{e}_0, \mathfrak{e}_1)_\sharp \Lambda = P_\sharp \big((E_{0,1})_\sharp \mathfrak{L}\big)$, where $P$ is defined as in \eqref{eq: 3.6}.

    (3)$\implies$(1): thanks to Proposition \ref{starting from Lambda} and the assumptions, it holds
    \begin{align*}
    \mathcal{W}_p^p(M_0,M_1)\leq & \int_0^1|\dot{\boldsymbol{M}}|^p(r)dr \leq \int \int_0^1 |\dot{\boldsymbol{\mu}}|_{W_p}^p(r)dr d\Lambda(\boldsymbol{\mu}) 
    \\
    = & \int W_p^p(\mu_0,\mu_1)d\Lambda(\boldsymbol{\mu}) = \mathcal{W}_p^p(M_0,M_1).
    \end{align*}
    Reasoning as above, it implies that $\boldsymbol{M}$ is a constant speed geodesic in $(\PP_p(\PP_p(X)), \mathcal{W}_p)$.
\end{proof}

The previous result can be improved when $X = \R^d$ (for simplicity, but the same holds whenever $X$ is a Polish space for which there is uniqueness of geodesics), showing that all the objects involved are unique whenever we restrict ourselves to either $[0,t]$ or $[t,1]$, for $t\in(0,1)$, giving a non-branching property for geodesics of random measures.

First we need the operation of composition of random couplings, following \cite[Lemma 5.3.2]{ambrosio2005gradient}.

\begin{prop}\label{prop: comp rand couplings}
    Let $X$ be a Polish space, and $M_1,M_2,M_3 \in \PP(X)$. Let $\mathfrak{P}_{1,2} \in \mathrm{R}\Gamma(M_1,M_2)$ and $\mathfrak{P}_{2,3} \in \mathrm{R}\Gamma(M_2,M_3)$. Then, there exists $\mathfrak{P}_{1,2,3}\in \PP(\PP(X\times X \times X))$ such that 
    \begin{equation}
        P^{1,2}_\sharp  \mathfrak{P}_{1,2,3} = \mathfrak{P}_{1,2}, \quad P^{2,3}_\sharp  \mathfrak{P}_{1,2,3} = \mathfrak{P}_{2,3},
    \end{equation}
    where, for all $i,j = 1,2,3$ 
    \[P^{i,j}:\mathcal{P}(X\times X \times X) \to \PP(X\times X), \quad P^{i,j}(\theta) = p^{i,j}_\sharp \theta.\]
    In particular, $P^{1,3}_\sharp \mathfrak{P}_{1,2,3}\in \mathrm{R}\Gamma(M_1,M_3)$.
\end{prop}

\begin{proof}
    The proof again uses a measurable selection argument. Let $\mathcal{A}\subset\PP(X\times X)\times \PP(X\times X)$ be the Borel subset defined as 
    \[\mathcal{A} = \left\{(\pi,\sigma) \in \PP(X\times X) \times \PP(X\times X) \, : \, p^2_\sharp \pi=p^1_\sharp \sigma \right\}.\]
    Define the map 
    \[\mathrm{P} = (P^{1,2},P^{2,3}):\PP(X\times X \times X) \to \mathcal{A}, \quad \mathrm{P}(\theta):= (p^{1,2}_\sharp \theta, p^{2,3}_\sharp \theta).\]
    It is Borel thanks to Proposition \ref{meas of push forward}. Moreover, thanks to \cite[Lemma 5.3.2 \& Remark 5.3.3]{ambrosio2005gradient}, the map $\mathrm{P}$ is surjective. Thus, we can apply Theorem \ref{measurable selection} to obtain a Souslin-Borel measurable map $\mathrm{Q}:\mathcal{A}\to \PP(X\times X \times X)$ that is a right-inverse of $\mathrm{P}$.
    \\
    Now, consider the disintegration of $\mathfrak{P}_{1,2}$ and $\mathfrak{P}_{2,3}$, respectively, with respect to the maps $P^1$ and $P^2$, to obtain that
    \[\mathfrak{P}_{1,2} = \int_{\PP} \mathfrak{P}_{1,2,\mu} dM_t(\mu), \quad \mathfrak{P}_{2,3} = \int_{\PP} \mathfrak{P}_{2,3,\mu} dM_t(\mu),\]
    where $\mathfrak{P}_{1,2,\mu} \in \PP(\PP(X\times X))$ is concentrated over couplings $\pi$ for which $p^2_\sharp \pi = \mu$, for $M_t$-a.e. $\mu\in \PP(X)$, and similarly for $\mathfrak{P}_{2,3,\mu}$. Notice that, for $M_t$-a.e. $\mu$, the product measure $\mathfrak{P}_{1,2,\mu} \otimes \mathfrak{P}_{2,3,\mu} $ is concentrated over $\mathcal{A}$. Thus, we can define 
    \begin{equation}
        \mathfrak{P}_{1,2,3}:= \mathrm{Q}_\sharp \left( \int_\PP \mathfrak{P}_{1,2,\mu} \otimes \mathfrak{P}_{2,3,\mu} dM_t(\mu) \right) =  \int_\PP \mathrm{Q}_\sharp \left(\mathfrak{P}_{1,2,\mu} \otimes \mathfrak{P}_{2,3,\mu} \right) dM_t(\mu),
    \end{equation}
    It is not hard to show that $\mathfrak{P}_{1,2,3}$ shares the wanted properties.
\end{proof}

\begin{teorema}
    Let $\boldsymbol{M} = (M_t)_{t\in[0,1]} \in C([0,1],\PP_p(\PP_p(\R^d)))$ be a geodesic in $(\PP_p(\PP_p(\R^d)),\mathcal{W}_p)$. Then, for every $t\in(0,1)$, $\mathrm{R}\Gamma_0(M_0,M_t)$ (resp. $\mathrm{R}\Gamma_0(M_t,M_1)$) contains a unique optimal random coupling $\mathfrak{P}_{0,t}$ (resp. $\mathfrak{P}_{t,1}$). 
    \\
    Moreover, there exists a unique $\mathfrak{L}_{0,t} \in \PP(\PP(C([0,t],\R^d)))$ (resp. $\mathfrak{L}_{t,1} \in \PP(\PP(C([t,1],\R^d)))$) lifting of $(M_s)_{s\in[0,t]}$ (resp. $(M_s)_{s\in[t,1]}$) satisfying property (2) in Proposition \ref{prop: char geo}.
    \\
    Similarly, there exists a unique $\Lambda_{0,t} \in \PP(C([0,t],\PP(\R^d)))$ (resp. $\Lambda_{t,1}\in \PP(C([t,1],\PP(\R^d))))$) lifting of $(M_s)_{s\in[0,t]}$ (resp. $(M_s)_{s\in[t,1]}$) satisfying property (3) in Proposition \ref{prop: char geo}.
\end{teorema}
\begin{proof}
    Let $\mathfrak{P}_{0,t} \in \mathrm{R}\Gamma_0(M_0,M_t)$ and $\mathfrak{P}_{t,1}\in \mathrm{R}\Gamma_{0}(M_t,M_1)$ be any optimal random coupling. Applying Proposition \ref{prop: comp rand couplings}, we obtain $\mathfrak{P}_{0,t,1}\in \PP(\PP(\R^d \times \R^d \times \R^d))$ satisfying
    \begin{equation}
        P^{1,2}_\sharp  \mathfrak{P}_{0,t,1} =\mathfrak{P}_{0,t}, \quad P^{2,3}_\sharp  \mathfrak{P}_{0,t,1} = \mathfrak{P}_{t,1}, \quad \mathfrak{P}_{0,1}:=P^{1,3}_\sharp  \mathfrak{P}_{0,t,1} \in \mathrm{R}\Gamma(M_0,M_1).
    \end{equation}
    Exploiting the geodesic property, we can also show that $\mathfrak{P}_{0,1}\in \mathrm{R}\Gamma_0(M_0,M_1)$:
    \begin{align*}
        \mathcal{W}_p& (M_0,M_1) \leq \left(\iint |x-y|^p d\pi(x,y) d\mathfrak{P}_{0,1}(\pi)\right)^{\frac{1}{p}} = \left( \iint |x_1-x_3|^p d\theta(x_1,x_2,x_3) d\mathfrak{P}_{0,t,1}(\theta) \right)^{\frac{1}{p}}
        \\
        \leq &
        \left(\iint  |x_1-x_2|^pd\theta(x_1,x_2,x_3) d\mathfrak{P}_{0,t,1}(\theta) \right)^{\frac{1}{p}} + \left(\iint  |x_2-x_3|^pd\theta(x_1,x_2,x_3) d\mathfrak{P}_{0,t,1}(\theta) \right)^{\frac{1}{p}} 
        \\
        = & \mathcal{W}_p(M_0,M_t) + \mathcal{W}_p(M_t,M_1) = \mathcal{W}_{p}(M_0,M_1).
    \end{align*}
    An important property is also hidden in the previous computation: since the last inequality is actually an equality, we know that for $\mathfrak{P}_{0,t,1}$-a.e. $\theta\in \PP(\R^d \times \R^d \times \R^d)$ and for $\theta$-a.e. $(x_1,x_2,x_3)$, the three points are aligned, in the sense that there exists $\alpha\in(0,1)$ such that $x_2 - x_1 = \alpha(x_3-x_1)$. Moreover, using again the geodesic property, it is not hard to show that $\alpha = t$. In particular, for $\mathfrak{P}_{0,t,1}$-a.e. $\theta\in \PP(\R^d \times \R^d \times \R^d)$ and for $\theta$-a.e. $(x_1,x_2,x_3)$, it holds $x_1 = \frac{x_2 - t x_3}{1-t}$ and $x_3 = \frac{x_2 - (1-t)x_1}{t}$. Thus, define the functions
    \[\ell_1(x_2,x_3) := \left(  \frac{x_2 - t x_3}{1-t},x_2,x_3 \right), \quad \mathrm{L}_1:= (\ell_1)_\sharp : \PP(\R^d \times \R^d) \to \PP(\R^d \times \R^d \times \R^d), \]
    \[\ell_3(x_1,x_2) := \left(  x_1,x_2,\frac{x_2 - (1-t)x_1}{t} \right), \quad \mathrm{L}_3:= (\ell_3)_\sharp : \PP(\R^d \times \R^d) \to \PP(\R^d \times \R^d \times \R^d).\]
    Thanks to the previous observations, we can conclude that $\mathfrak{P}_{0,t} = P^{1,2}_\sharp \mathfrak{P}_{0,t,1} = P^{1,2}_\sharp \big( (\mathrm{L}_1)_\sharp  \mathfrak{P}_{t,1} \big)$, that makes us conclude that $\mathfrak{P}_{0,t}$ is unique, since it has been chosen independently of $\mathfrak{P}_{t,1}$. Similarly, the uniqueness for $\mathfrak{P}_{t,1}$ is implied by $\mathfrak{P}_{t,1} = P^{2,3}_\sharp \big( (\mathrm{L}_3)_\sharp  \mathfrak{P}_{0,t} \big)$.

    The uniqueness of $\mathfrak{L}_{0,t}$ follows from Proposition \ref{prop: char geo} and observing that the required properties force it to be equal to $P^{0,t}_\sharp  \mathfrak{P}_{0,t}$, where 
    \[P^{0,t}:= p^{0,t}_\sharp , \quad p^{0,t}:\R^d \times \R^d \to C([0,t],\R^d), \ p^{0,t}(x_1,x_2) = [s\mapsto (1-s)x_1 + s x_2].\]
    The uniqueness of $\Lambda_{0,t}$ now follows, and similarly, the same holds for $\mathfrak{L}_{t,1}$ and $\Lambda_{t,1}$.
\end{proof}

\subsection{Barycenters}\label{subsec: barycenters}
In this subsection, we apply our result to show that, if $\boldsymbol{M}\in AC^p_T(\PP(\PP(X)))$, then the curves of its barycenters is in $AC_T^p(\PP(X))$.
\begin{df}\label{def: barycenter}
    Let $M\in \PP(\PP(X))$. The barycenter of $M$ is the measure $\operatorname{bar}[M]\in \PP(X)$ defined such that it satisfies 
    \begin{equation}\label{mean def}
        \int_X f(x) d\big(\operatorname{bar}[M]\big)(x) = \int \int_X f(x) d\mu(x) dM(\mu) 
    \end{equation}
    for all Borel functions $f:X \to[0,+\infty]$.
\end{df}

Recalling the definition of $\widetilde{M}\in \PP(X\times \PP(X))$ given in Remark \ref{rem: tilde{M}}, we notice that the first marginal of $\widetilde{M}$ is indeed $\operatorname{bar}[M]$. Now, we show a nice property of the barycenter with respect to the nested push-forward described in \eqref{push forward of push forward}.

\begin{lemma}\label{lemma barycenters and push-forward}
    Let $X,Y$ be two Polish spaces, $M\in \PP(\PP(X))$ and $f:X\to Y$ a Borel map. Define the map $F:\PP(X) \to \PP(Y)$ as the push-forward of $f$, i.e. $F= f_\sharp $. Then it holds 
    \begin{equation}
        \operatorname{bar}[F_\sharp  M] = f_\sharp (\operatorname{bar}[M]).
    \end{equation}
\end{lemma}

\begin{proof}
    Let $N = F_\sharp M$. For all $g: Y \to [0,+\infty]$ Borel measurable, thanks to \eqref{property of push for of push for}, it holds
    \begin{align*}
        \int_Y g(y) d\operatorname{bar}[N](y) = & \int_{\PP(Y)} \int_Y g(y) d\nu(y) dN(\nu)
        =  
        \int_{\PP(X)} \int_X g(f(x)) d\mu(x) dM(\mu)
        \\
        = &
        \int_X g(f(x)) d\big(\operatorname{bar}[M]\big)(x) = \int_Y g(y) df_\sharp (\operatorname{bar}[N])(y).
        \end{align*}
\end{proof}

Given a curve $\boldsymbol{M} = (M_t)_{t\in [0,T]} \in C_T(\PP(\PP(X)))$, we indicate with $\boldsymbol{\operatorname{bar}[M]}= (\operatorname{bar}[M_t])_{t\in [0,T]} \in C_T(\PP(X))$ the curve of the barycenters.

\begin{prop}
    If $\boldsymbol{M}\in AC^p_T\big(\PP_p(\PP_p(X))\big)$, then $\boldsymbol{m}:=\boldsymbol{\operatorname{bar}[M]}\in AC^p_T(\PP_p(X))$, and, for a.e. $t\in(0,T)$, it holds
    \begin{equation}
        | \dot{\boldsymbol{m}} |_{W_p}^p(t) \leq \int \int |\dot{\boldsymbol{x}}|^p(t) d\lambda(\boldsymbol{x}) d\mathfrak{L}(\lambda) = |\dot{\boldsymbol{M}}|_{\mathcal{W}_p}^p(t).
    \end{equation}
\end{prop}

\begin{proof}
    For all $t\in [0,T]$, $\operatorname{bar}[M_t]\in \PP_p(X)$, indeed
    \[\int_X  d^p(x,\bar{x}) d\operatorname{bar}[M_t](x) = \int \int_X d^p(x,\bar{x}) d\mu(x) dM_t(\mu) <+\infty. \]
    Consider $\mathfrak{L}\in \PP_{\mathcal{A}_p}\big(\PP(C_T(X))\big)$ defined as $\mathfrak{L}:= (G_\sharp )\circ \operatorname{Lift}(\boldsymbol{M})$. Define the map 
    \begin{align*}
        (E_t,E_s): \PP(C_T(X)) & \to \PP(X\times X)
        \\
        \lambda \ & \mapsto (e_t,e_s)_\sharp \lambda,
    \end{align*}
    and $\mathfrak{P}_{t,s}:= (E_t,E_s)_\sharp  \mathfrak{L} \in \PP\big(\PP(X\times X)\big)$. Notice that the barycenter of $\mathfrak{P}_{t,s}$, indicated as $\pi_{t,s}\in \PP(X\times X)$, is a coupling between $\operatorname{bar}[M_t]$ and $\operatorname{bar}[M_s]$: using Lemma \ref{lemma barycenters and push-forward}, indeed $\operatorname{bar}[(E_t)_\sharp  \mathfrak{L}] = (e_t)_\sharp (\operatorname{bar}[\mathfrak{L}])$, and same for $s$. Thus, we have
    \begin{align*}
        W_p^p(\operatorname{bar}[M_t] & ,\operatorname{bar}[M_s])
        \leq 
        \int \int d^p(x,y) d\pi(x,y) d\mathfrak{P}_{t,s}(\pi)
        \\
        = &
        \int \int d^p(\gamma_t,\gamma_s) d\lambda(\gamma) d\mathfrak{L}(\lambda)
        \leq 
        |t-s|^{p-1} \int_s^t \int \int |\dot\gamma|^p(r) d\lambda(\gamma) d\mathfrak{L}(\lambda) dr,
    \end{align*}
    which implies that $\boldsymbol{m}:=\boldsymbol{\operatorname{bar}[M]} \in AC_T(\PP_p(X))$, and using Lebesgue theorem, it holds 
    \begin{equation}
        | \dot{\boldsymbol{m}} |_{W_p}^p(t) \leq \int \int |\dot{\boldsymbol{x}}|^p(t) d\lambda(\boldsymbol{x}) d\mathfrak{L}(\lambda) = |\dot{\boldsymbol{M}}|_{\mathcal{W}_p}^p(t) \quad \text{for a.e. $t\in [0,T]$.}
    \end{equation}
\end{proof}
\section{Continuity equation on random measures}\label{CE: derivations vs VF}

\subsection{Derivations, continuity equation and a first superposition result}\label{subsec: first superposition result}
Derivations are the natural objects that can be used to define an abstract continuity equation over a metric space $X$, see \cite{stepanov2017three}. Here, we adapt this definition in the case $X= \PP(\R^d)$, endowed with the narrow topology.

\begin{df}[Cylinder functions and Wasserstein gradient]\label{def cyl func and wass grad}
    A functional $F:\PP(\R^d)\to\R$ is called a $C^1_c$-cylinder function if there exists $k\in \N$, $\Phi = (\phi_1,\dots,\phi_k)\in C^1_c(\R^d;\R^k)$ and $\Psi\in C^1_b(\R^k)$ such that 
    \begin{equation}\label{cyl functions}
        F(\mu) = \Psi\left( L_\Phi(\mu) \right), \quad L_\Phi(\mu) =  \big(L_{\phi_1}(\mu) ,\dots , L_{\phi_k}(\mu)\big), \quad  L_{\phi_i}(\mu) := \int_{\R^d} \phi_i(x) d\mu(x).
    \end{equation}
    Its Wasserstein gradient is then defined as 
    \begin{equation}\label{wass gradient cyl}
        \nabla_W F(x,\mu):= \sum_{i=1}^k \partial_{i}\Psi\left( L_{\Phi}(\mu)\right) \nabla \phi_i(x) \quad \forall x\in \R^d, \ \forall\mu \in \PP(\R^d).
    \end{equation}
    The collection of all the cylinder functions is called $\operatorname{Cyl}^1_c(\PP(\R^d))$. If $F = \Psi\circ L_\Phi$ with $\Phi \in C_b^1(\R^d)$ and $\Psi \in C_b^1(\R^k)$, we say that $F \in \operatorname{Cyl}^1_b(\PP(\R^d))$. 
\end{df}

Notice that, when $F\in \operatorname{Cyl}^1_c(\PP(\R^d))$, we can consider the outer function $\Psi \in C_c^1(\R^k)$. Moreover, we have the inclusion $\operatorname{Cyl}^1_c(\PP(\R^d))\subset \operatorname{Cyl}^1_b(\PP(\R^d)).$

\begin{oss}
    Since $\nabla_W F(x,\mu) = \nabla_x \left(\frac{d^+}{d\varepsilon}|_{\varepsilon=0} \ F((1-\varepsilon)\mu+ \varepsilon\delta_x)\right)$, the Wasserstein gradient does not depend on the representation chosen for $F$.
\end{oss}

\begin{df}[$L^p$ derivations]\label{def derivations}
    Let $M\in \PP(\PP(\R^d))$ and $p\geq 1$. An $L^p(M)$-derivation is a linear operator $B:\operatorname{Cyl}_c^1(\PP(\R^d))\to L^p(M)$ such that 
    \begin{equation}\label{leibniz rule}
        B[FG] = GB[F] + FB[G] \quad M\text{-a.e.} 
    \end{equation}
    and there exists a non-negative function $c\in L^p(M)$ such that 
    \begin{equation}\label{eq: int condition for derivations}
        |B[F](\mu)| \leq c(\mu) \|\nabla_WF(\cdot,\mu)\|_{L^{p'}(\mu)} \quad\forall F \in \operatorname{Cyl}_c^1(\PP(\R^d)), \ \text{ for } M\text{-a.e. } \mu \in \PP(\R^d),
    \end{equation}
    where $p'$ is the conjugate exponent of $p$.
\end{df}

We will extensively work with families of $L^p$-derivations: to be more specific, let $(M_t)_{t\in [0,T]}\subset \PP(\PP(\R^d))$ be a Borel family of random measures and $(B_t)_{t\in [0,T]}$ a family of $L^p(M_t)$-derivations such that $(t,\mu)\mapsto B_t[F](\mu)$ is Borel measurable for all $F\in \operatorname{Cyl}_c^1(\PP(\R^d))$ and there exists a non-negative function $c\in L^p(M_t\otimes dt)$ (recall the notation introduced in §\ref{subsub cont curve in prob space}) such that
\begin{equation}\label{time integrability family of derivations}
    |B_t[F](\mu)| \leq c_t(\mu) \|\nabla_W F\|_{L^{p'}(\mu)} \  M_t\otimes dt\text{-a.e.} \quad \text{and} \quad \int_0^T \int_{\PP(\R^d)} c_t^p(\mu)dM_t(\mu) dt <+\infty.
\end{equation}
We will refer to such a kind of family of derivations as $L^p(M_t\otimes dt)$-derivations. Notice that, since all the measures involved are probability measures, an $L^p(M)$-derivation is also an $L^q(M)$-derivation for each $q\in[1,p]$.

\begin{oss}\label{rem: negligible sets derivation}
    Thanks to \eqref{eq: int condition for derivations}, given a representative of $c$, there exists $\mathcal{N}\subset \PP(\R^d)$ such that $M(\mathcal{N}) = 0$ and $|B[F](\mu)|\leq c(\mu) \|\nabla_WF(\cdot,\mu)\|_{L^{p'}(\mu)}< +\infty$ for all $F \in \operatorname{Cyl}_c^1(\PP(\R^d))$ and $\mu\in\mathcal{N}^c$. 
    Similarly for $L^p(M_t\otimes dt)$-derivations, exploiting \eqref{time integrability family of derivations}, there exists $\widetilde{\mathcal{N}}\subset \PP(\R^d)\times [0,T]$ such that $M_t\otimes dt(\widetilde{\mathcal{N}}) = 0$ and $|B_t[F](\mu)|\leq c_t(\mu) \|\nabla_WF(\cdot,\mu)\|_{L^{p'}(\mu)}<+\infty$ for all $F \in \operatorname{Cyl}_c^1(\PP(\R^d))$ and $(\mu,t)\in\widetilde{\mathcal{N}}^c$. 
\end{oss}

\begin{df}[Continuity equation on random measures]\label{def CE}
    Let $(M_t)_{t\in[0,T]}$ be a continuous curve of Borel probability measures over $\mathcal{P}(\R^d)$, i.e. $(M_t)_{t\in[0,T]}\in C_T(\PP(\PP(\R^d)))$. Let $(B_t)_{t\in[0,T]}$ be a family of $L^1(M_t)$-derivations, according to \eqref{time integrability family of derivations}. We say that the continuity equation 
    \begin{equation}\label{CE meas}
        \frac{d}{dt}M_t + \operatorname{div}_{\mathcal{P}}\big( 
        B_t M_t \big)=0
    \end{equation}
    is satisfied if 
    \begin{equation}\label{CE meas 2}
    \begin{aligned}
        \forall F \in \operatorname{Cyl}_c^1(\PP(\R^d)) \quad \frac{d}{dt}\int_{\mathcal{P}} F(\mu) \ dM_t(\mu) 
        = \int_{\mathcal{P}}B_t[F](\mu) dM_t(\mu) 
    \end{aligned}
    \end{equation}
    in the sense of distributions in $(0,T)$, where $\PP$ is short for $\PP(\R^d)$.
\end{df}

\begin{oss}
    The assumption that the curve $t\mapsto M_t$ is continuous, is not restrictive. Indeed, if such curve is just Borel measurable, thanks to Lemma \ref{abs cont sol of CE}, we can always find a (unique) continuous representative of it.
\end{oss}

\noindent Before proceeding, we state and prove two technical, but useful, lemmas.

\begin{lemma}[Chain rule]\label{chain rule}
    Let $M\in \PP(\PP(\R^d))$ and $B$ be an $L^1(M)$-derivation. Then, for all $F\in \operatorname{Cyl}_c^1(\PP(\R^d))$ of the form $F = \Psi (L_\Phi(\mu))$ as in \eqref{cyl functions},  for $M$-a.e. $\mu\in \PP(\R^d)$ it holds 
    \begin{equation}\label{chain rule lemma}
        B[F](\mu) = \sum_{i=1}^k \partial_i\Psi(L_\Phi(\mu)) B[L_{\phi_i}](\mu).
    \end{equation}
\end{lemma}

\begin{proof}
    Thanks to the Leibniz rule, by linearity and induction \eqref{chain rule lemma} holds when $\Psi$ is a polynomial. When $\Psi$ is not a polynomial, consider $(p_n)$ a sequence of polynomial approximating $\Psi$ uniformly on compact sets, together with its first derivatives. By the boundedness of $\Phi $ and its first derivatives, we conclude.
\end{proof}

\begin{lemma}\label{lemma: well posed for Lambda-a.e. curve}
    Let $(M_t)_{t\in [0,T]} \in C_T(\PP(\PP(\R^d)))$ and $\Lambda \in \PP(C_T(\PP(\R^d)))$ such that $(\mathfrak{e}_t)_\sharp \Lambda = M_t$ for all $t\in [0,T]$. Let $p\geq 1$. Let $(B_t)_{t\in [0,T]}$ be a family of $L^p(M_t \otimes dt)$-derivations and $c\in L^p(M_t\otimes dt)$ as in \eqref{time integrability family of derivations}. 
    Then the functions 
    \begin{equation}\label{eq: meas in (t,boldsymbol(mu))}(t,\boldsymbol{\mu})\mapsto B_t[F](e_t(\boldsymbol{\mu})), \quad (t,\boldsymbol{\mu})\mapsto c_t^p(e_t(\boldsymbol{\mu}))
    \end{equation}
    are $\mathcal{L}^1_T\otimes\Lambda$-measurable and well-defined. Moreover, for $\mathcal{L}^1_T\otimes\Lambda$-a.e. $(t,\boldsymbol{\mu})$ and for all $F\in \operatorname{Cyl}_c^1(\PP(\R^d))$, it holds
    \begin{equation}\label{eq: thesis lemma well posed Lambda curve}
        |B_t[F](e_t(\boldsymbol{\mu}))| \leq c_t(e_t(\boldsymbol{\mu}))|\nabla_W F(\cdot,e_t(\boldsymbol{\mu})\|_{L^{p'}(e_t(\boldsymbol{\mu}))}<+\infty, \quad \int_0^T c_t(e_t(\boldsymbol{\mu}))dt <+\infty.
    \end{equation}
\end{lemma}

\begin{proof}
The functions in \eqref{eq: meas in (t,boldsymbol(mu))} are composition of measurable maps, so they are measurable. Now, consider $\widetilde{\mathcal{N}}$ as in Remark \ref{rem: negligible sets derivation} and define the function 
    \[
    \begin{aligned}
        \widetilde{E}:[0,T]\times C_T(\PP(\R^d)) & \to [0,T]\times \PP(\R^d)
        \\
        (t,\boldsymbol{\mu}) & \mapsto (t,\mu_t),
    \end{aligned}\]
    where we mean that $\boldsymbol{\mu} = (\mu_t)_{t\in[0,T]}$, so that $\mu_t = e_t(\boldsymbol{\mu})$.
    Notice that $\widetilde{E}_\sharp (\mathcal{L}_T^1\otimes \Lambda) = dt\otimes M_t$, so $\widetilde{E}^{-1}(\widetilde{\mathcal{N}})$ is a negligible set w.r.t. $\mathcal{L}_T^1 \otimes \Lambda$, which implies \eqref{eq: thesis lemma well posed Lambda curve}. To conclude, notice that 
    \[\int \int_0^T c_t^p(\mu_t) dt d\Lambda(\boldsymbol{\mu}) = \int_0^T \int c_t^p(\mu_t)  d\Lambda(\boldsymbol{\mu}) dt =  \int_0^T \int_\PP c_t^p(\mu)  dM_t(\mu) dt<+\infty.\]
\end{proof}

Now, we prove a first superposition result in terms of derivations.
Our proof strongly relies on the techniques developed in \cite{ambrosio2014well}: indeed, we embed the space $\PP(\R^d)$ in $\R^\infty$, where an infinite-dimensional version of Theorem \ref{finite dim superposition principle} holds (see Appendix \ref{appendix R^infty}). 

We are going to use a similar notation as for the purely metric setting:
\begin{itemize}
    \item $\boldsymbol{\gamma}:= (\gamma_t)_{t\in[0,T]} \subset C_T(\R^d)$;
    \item $\lambda \in \PP\big(C_T(\R^d)\big)$;
    \item $\boldsymbol{\mu}:=(\mu_t)_{t\in[0,T]}\in C_T\big( 
    \PP(\R^d) \big)$;
    \item $M\in \PP(\PP(\R^d))$;
    \item $\boldsymbol{M}:=(M_t)_{t\in[0,T]} \in C_T\big(\PP(\PP(\R^d))\big)$;
    \item $\Lambda \in \PP\big(
    C_T(\PP(\R^d))\big)$;
    \item $\mathfrak{L}\in \PP\big(\PP(C_T(\R^d))\big)$.
\end{itemize}

\begin{teorema}\label{superposition principle}
    Let $\boldsymbol{M} = (M_t)_{t\in [0,T]}$ be a continuous curve of probability measures over $\mathcal{P}(\R^d)$. Let $(B_t)_{t\in[0,T]}$ be a Borel family of derivations and $c \in L^1(M_t\otimes dt)$ as in \eqref{time integrability family of derivations} with $p=1$.
    Assume that $(M_t)_{t\in[0,T]}$ satisfies the continuity equation $\partial_tM_t+\operatorname{div}_\PP(B_tM_t) = 0$ in the sense of Definition \ref{def CE}.
    Then there exists a probability measure $\Lambda \in \mathcal{P}\big( C_T(\mathcal{P}(\R^d) \big)$, such that 
    \begin{itemize}
        \item $(\mathfrak{e}_t)_{\sharp }\Lambda = M_t$ for all $t\in [0,T]$;
        \item for $\Lambda$-a.e. curve $(\mu_t)_{t\in [0,T]}$ and a.e. $t\in[0,T]$, $B_t[F](\mu_t)$ and $c_t(\mu_t)$ are well defined and it holds
        \begin{equation}\label{CE classic}
            \int_0^T c_t(\mu_t) dt<+\infty \ \ \text{ and } \ \ \partial_t\mu_t + \operatorname{div}\big( B_t \mu_t \big)= 0 \quad \text{for } \Lambda\text{-a.e. }(\mu_t)_{t\in[0,T]}
        \end{equation}
        in the sense of distributions in duality with $C_c^1$ functions, i.e. for all $\phi \in C_c^1(\R^d)$ and $\psi \in C_c^1(0,T)$
        \begin{equation}
            \int_0^T \int_{\R^d} \psi'(t) \phi(x) d\mu_t(x) dt = -\int_0^T \psi(t) B_t[L_{\phi}](\mu_t) dt. 
        \end{equation}
    \end{itemize}
\end{teorema}

\begin{proof}
    \textbf{Step 1}: here we use the result presented in Appendix \ref{appendix R^infty}. Let $\mathcal{A}=\{\phi_1,\phi_2,\dots \}\subset C_c^1(\R^d)$ and $\iota$ as in Appendix \ref{appendix R^infty}, i.e.
    \begin{itemize}
        \item $\phi_k$ is Lipschitz with respect to $|\cdot |\wedge 1$ for all $k\in \N$. In particular, $\|\nabla\phi_k\|_\infty\leq 1$ for all $k\in \N$;
        \item $\operatorname{Span}(\mathcal{A})$ dense in $C_0^1(\R^d)$, i.e. functions that, together with their first derivatives, converge to $0$ at infinity;
        \item $\hat{W}_1(\mu,\nu):= W_{1,|\cdot|\wedge 1}(\mu,\nu) = \sup_k\int_{\R^d}\phi_k d(\mu-\nu)$;
    \end{itemize}
    and
    \begin{equation}\label{iota Rd}
    \iota:\PP(\R^d) \to \R^\infty, \quad \iota(\mu) = \big(L_{\phi_1}(\mu), L_{\phi_2}(\mu),\dots \big).
    \end{equation}

    \textbf{Step 2}: define $\mathtt{m}_t:= \iota_\sharp  M_t$. We prove that $(\mathtt{m}_t)\in AC_T(\PP(\R^\infty),W_{1,D_\infty})$ (see \eqref{eq: distance D_infty}).
    \\
    For any $s,t\in[0,T]$, consider $\Pi_{s,t}$ an optimal transport plan between $M_t$ and $M_s$ realizing $\hat{\mathcal{W}}_1(M_t,M_s):=W_{1,\hat{W}_1}(M_t,M_s)$. Then $(\iota,\iota)_\sharp \Pi_{t,s}$ is a transport plan between $\mathtt{m}_t$ and $\mathtt{m}_s$, so that 
    \begin{align*}
        W_{1,D_\infty} &(\mathtt{m}_t,\mathtt{m}_s) 
        \leq  \int D_{\infty}(\mathrm{x},\mathrm{y}) \ d(\iota,\iota)_\sharp \Pi_{t,s}(x,y)
        =  \int \sup_k |L_{\phi_k}(\mu) - L_{\phi_k}(\nu)| \wedge 1 \ d\Pi_{t,s}(\mu,\nu)
        \\
        = & \int \sup_k |L_{\phi_k}(\mu) - L_{\phi_k}(\nu)| d\Pi_{t,s}(\mu,\nu)
        =
        \int \hat{W}_1(\mu,\nu)d\Pi_{t,s}(\mu,\nu) 
        = \hat{\mathcal{W}}_1(M_t,M_s),
    \end{align*}
    and we are done thanks to Lemma \ref{abs cont sol of CE}.
    
    \textbf{Step 3}: define, component-wisely, the vector field $v:[0,T]\times \R^\infty \to \R^\infty$ as 
    \begin{equation}
        v_t^{(k)}(\mathrm{x}):=
        \begin{cases}
            B_t[L_{\phi_k}](\iota^{-1}(\mathrm{x})) & \text{ if }\mathrm{x}\in \iota(\PP(\R^d))
            \\
            0 & \text{ if }\mathrm{x}\notin \iota(\mathcal{P}(\R^d)).
        \end{cases}
    \end{equation}
    Notice that
    \begin{equation*}\label{eq: ineq superposition}
    \begin{aligned}
        \int_0^T\int_{\R^\infty}| v_t^{(k)}(\mathrm{x})|d\mathtt{m}_t(\mathrm{x}) dt
        =  
        \int_0^T \int_{\mathcal{P}}|B_t[L_{\phi_k}](\mu_t)|dM_t(\mu)dt
        \leq 
        \int_0^T \int_{\mathcal{P}} c_t(\mu)dM_t(\mu)dt<+\infty,
    \end{aligned} 
    \end{equation*}
    and for any cylinder function $F:\R^\infty\to \R$, i.e. such that there exists $n \in \N$ for which $F(\mathrm{x}) = F(x_1,\dots,x_n)$, it holds (thanks to Lemma \ref{chain rule lemma}) 
    \begin{align*}
        \frac{d}{dt}&\int_{\R^\infty} F(\mathrm{x})d\mathtt{m}_t(\mathrm{x}) 
        = 
        \frac{d}{dt} \int_{\mathcal{P}}F(\iota(\mu))dM_t(\mu) 
        = 
        \frac{d}{dt}\int_{\mathcal{P}}F(L_{\phi_1}(\mu),\dots,L_{\phi_n}(\mu)) dM_t(\mu)
        \\
        = & 
        \sum_{i=1}^n \int_{\mathcal{P}}\partial_i F(L_{\phi_1}(\mu),\dots,L_{\phi_n}(\mu)) \ B_t[L_{\phi_k}](\mu_t) dM_t(\mu) 
        =  
        \int_{\R^\infty}\nabla F(\mathrm{x}) \cdot \big( v_t^{(1)},\dots, v_t^{(n)}\big) d\mathtt{m}_t(\mathrm{x}).
   \end{align*}
   Then, we can apply Theorem \ref{superposition R^infty}, to obtain the existence of a measure $\mathtt{L} \in \mathcal{P}\big( C_T(\R^\infty,d_\infty) \big)$ satisfying
   \[(e_t)_\sharp \mathtt{L}  = \mathtt{m}_t\]
   and $\mathtt{L}$-a.e. $\tilde{\gamma} \in C_T(\R^\infty)$ is weakly absolutely continuous with
   \[\frac{d}{dt}\tilde{\gamma}(t) = v_t(\tilde{\gamma}(t)), \]
   in the sense that each component $\tilde{\gamma}^{(k)}$ of the curve is in $AC([0,T],\R)$ and $\frac{d}{dt}\tilde{\gamma}^{(k)} = v_t^{(k)}(\tilde{\gamma}(t))$.

    \textbf{Step 4}: we prove that for $\mathtt{L}$-a.e. $\tilde{\gamma}$, $\tilde{\gamma} \in AC_T(\iota(\PP(\R^d)), D_\infty)$. In particular, for every $t\in [0,T]$, $\tilde{\gamma}(t) \in \iota(\mathcal{P})$. First, observe that for $\mathtt{L}$-a.e. $\tilde\gamma$, $\tilde\gamma(t) \in \iota(\PP(\R^d))$ for any $t\in \mathbb{Q}\cap[0,T]$.
        Fix $t\in[0,T]\cap \mathbb{Q}$, then 
        \[\mathtt{L}\big( \{\tilde\gamma \ : \ \tilde\gamma(t)\in \iota(\PP(\R^d))\}\big) = \mathtt{m}_t(\iota(\PP(\R^d))) = 1,\]
        so we conclude because $[0,T]\cap\mathbb{Q}$ is countable.

        Now, for $\mathtt{L}$-a.e. $\big(\tilde\gamma(t)\big)_{t\in [0,T]}$, $\tilde\gamma\in AC_T(\R^\infty,D_\infty)$. In fact, $\mathtt{L}$-a.e. curve $\tilde\gamma$ and for any $s,t\in[0,T]$, it holds
        \begin{align*}
            D_{\infty}(\tilde\gamma(t),\tilde\gamma(s)) = & \sup_n |\tilde\gamma_n(t) - \tilde\gamma_n(s)|\wedge 1 
            \leq 
            \sup_n \int_s^t |v_r^{(n)}(\tilde\gamma(r))| dr \leq \int_s^t \sup_n|v_r^{(n)}(\tilde{\gamma}_r)|dr.
        \end{align*}
        Moreover, notice that
        \begin{align*}
            \int \int_0^T & \sup_n|v_r^{(n)}(\tilde\gamma_r) |dr d\mathtt{L}(\tilde\gamma) =  \int_0^T \int  \sup_n|v_r^{(n)}(\tilde\gamma_r)| d\mathtt{L}(\tilde\gamma) dr 
            = 
            \int_0^T \int \sup_n|v_r^{(n)}(\mathrm{x})| d\mathtt{m}_r(\mathrm{x}) dr
            \\
            \leq & \int_0^T \int \sup_n |v_r(n)(\iota(\mu))|  dM_r(\mu) dr
            \leq
            \int_0^T \int c_r(\mu) dM_r(\mu) dr
            <+\infty ,
        \end{align*}
        where we used that $ \sup_n |v_r(n)(\iota(\mu))| = \sup_n |B_r[L_{\phi_n}](\mu)| \leq c_r(\mu)$. This implies that for $\mathtt{L}$-a.e. $\tilde\gamma$ it holds that 
        \[\int_0^T \sup_n|v_r^{(n)}(\tilde\gamma_r) |dr<+\infty\]
        and thanks to the inequalities above, we have that for $\mathtt{L}$-a.e. $\tilde\gamma$, $\tilde\gamma\in AC_T(\R^\infty,D_\infty)$.
   The above properties show that for $\mathtt{L}$-a.e. $\tilde\gamma$, $\tilde\gamma(t) \in \iota(\PP(\R^d))$ for any $t\in [0,T]$, thanks to Lemma \ref{lemma: im of iota is closed}
   
    \textbf{Step 5}: thanks to the previous step, $\tilde{\Lambda}$ is concentrated over $C_T(\iota(\PP),D_\infty)$, thus, thanks to Lemma \ref{lemma: meas of curves in R^infty}, it can be seen as a probability measure over it, with its natural induced compact-open topology. Now, consider the function
    \begin{equation}
    \begin{aligned}
        \Theta: C_T(\iota(\mathcal{P}),D_\infty ) & \to  C_T(\mathcal{P}(\R^d))
        \\ \tilde{\gamma} \quad & \mapsto \quad \mu_t:= \iota^{-1}(\tilde{\gamma}(t)).        
    \end{aligned}
    \end{equation}
    It is well-defined, and thanks to the previous considerations we are allowed to define $\Lambda:= \Theta_\sharp  \mathtt{L}$, because $\mathtt{L}$ is concentrated over curves that are absolutely continuous with respect to $D_\infty$, and in particular on curves that are continuous with respect to it.

    By the properties of $\mathtt{L}$ and the fact that $e_t\circ \Theta (\tilde{\gamma})= \iota^{-1} (\tilde{\gamma}(t))$, it is straightforward that $(e_t)_\sharp \Lambda = M_t$, so that we can apply Lemma \ref{lemma: well posed for Lambda-a.e. curve} and then for $\Lambda$-a.e. $(\mu_t)_{t\in [0,T]}$ it holds
    \begin{equation}\label{eq: ce against phi_k}
    \frac{d}{dt}\int_X \phi_k(x) d\mu_t(x) = \frac{d}{dt}L_{\phi_k}(\mu_t) = B_t[L_{\phi_k}](\mu_t)\end{equation}
    in the sense of distribution in $(0,T)$, for all $k\in \N$. 
    
   By the density of $\operatorname{Span}(\mathcal{A})$ in $C_c^1(\R^d)$, it is immediate to prove that, if a curve $\boldsymbol{\mu} = (\mu_t)_{t\in[0,T]}$ satisfies \eqref{eq: ce against phi_k}, then \eqref{CE classic} holds. 
\end{proof}

\subsection{Correspondence between AC curves and solution to the continuity equation}\label{subsec: corr AC with CERM}

In this subsection we find a natural correspondence between solutions to the continuity equation associated with a family of $L^p$-derivations, and curves in $AC_T^p(\PP_p(\PP_p(\R^d)))$.

\begin{prop}[From CE to AC]\label{p-superposition}
    In the setting of Theorem \ref{superposition principle}, assume $p\geq 1$, $M_0 \in \PP_p(\PP_p(\R^d))$ and  
    \begin{equation}
        \int_0^T \int_\PP c_t^p(\mu) dM_t(\mu) dt <+\infty,
    \end{equation}
    i.e. $(B_t)_{t\in [0,T]}$ is a family of $L^p(M_t)$ derivations.
    Then, the probability measure $\Lambda$ given by Theorem \ref{superposition principle} is concentrated over $\boldsymbol{\mu}\in AC_T^p(\PP_p(\R^d))$ and
    \begin{equation}\label{eq: energy ineq derivations}
        \int_0^T |\dot{\boldsymbol{\mu}}|_{W_p}^p(t) dt  \leq \int_0^T c_t^p(\mu_t) dt <+\infty \quad \text{ for }\Lambda\text{-a.e. } \boldsymbol{\mu}.
    \end{equation}
    In particular, $\Lambda \in \PP_{\bar{\mathcal{A}}_p}(C_T(\PP(\R^d)))$ and $\boldsymbol{M} \in AC_T^p(\PP_p(\PP_p(\R^d)))$.
\end{prop}

\begin{proof}
    Thanks to the hypothesis $M_0\in \PP_p(\PP_p(\R^d))$, we already know that 
    \[\int W_p^p(\mu_0,\delta_{\bar{x}}) d\Lambda(\boldsymbol{\mu})<+\infty.\]
    Moreover, thanks to Lemma \ref{lemma: well posed for Lambda-a.e. curve} and Fubini's theorem, it holds that 
    \[\int_0^T\int c_t^p(\mu_t) d\Lambda(\boldsymbol{\mu}) dt = \int_0^T \int_\PP c_t^p(\mu) dM_t(\mu) dt <+\infty,\]
    so it is left to prove that $\Lambda$-a.e. $\boldsymbol{\mu}$ is in $AC_T(\PP_p(\R^d))$ and \eqref{eq: energy ineq derivations} holds.

    Case $p=1$: let $\phi \in C_b^1(\R^d)$ with $\|\nabla \phi\|_\infty \leq 1$. Then for any $s<t$ it holds
    \[\int_{\R^d} \phi d(\mu_t-\mu_s) \leq \int_s^t B_r[L_\phi](\mu_r) dr \leq \int_s^t c_r(\mu_r)dr,\]
    so that passing to the supremum w.r.t. $\phi$ on the left hand side, we have $W_1(\mu_s,\mu_t)\leq \int_s^t c_r(\mu_r)dr$, which implies \eqref{eq: energy ineq derivations}. Then, $\Lambda \in \PP_{\overline{\mathcal{A}}_1}(C_T(\PP(\R^d)))$ by definition and $\boldsymbol{M}\in AC_T^1(\PP_1(\PP_1(\R^d)))$ easily follows.

    Case $p>1$: for $\Lambda$-a.e. $\boldsymbol{\mu}$, we find a vector field $v:[0,T]\times \R^d \to \R^d$ for which $\partial_t\mu_t + \operatorname{div}(v_t\mu_t)=0$ is satisfied. 
    We know that $\Lambda$-a.e. $\boldsymbol{\mu}$ solves the continuity equation $\partial_t\mu_t +\operatorname{div}(B_t\mu_t) =0 $, in the sense that for all $\psi\in C_c^1(0,T)$ and $\phi \in C^1_c(\R^d)$ it holds
    \begin{equation}
        \int_0^T c_t^p(\mu_t) dt <+\infty \quad \text{and} \quad
        \int_0^T \psi'(t) \int_{\R^d} \phi(x) d\mu_t(x) dt = - \int_0^T \psi(t) B_t[L_\phi](\mu_t) dt.
    \end{equation}
    Fix a curve $\boldsymbol{\mu}= (\mu_t)_{t\in [0,T]} \in C_T(\PP(\R^d))$ with such properties and let $\mathcal{S}$ be the collection of real valued functions from $[0,T]\times \R^d$ defined as
    \[\mathcal{S}:= \operatorname{Span}\left( \left\{ (t,x) \mapsto \psi(t)\nabla_W F(x,\mu_t) \ : \ \psi\in C^1([0,T]), \ F\in \operatorname{Cyl}_c^1(\PP(\R^d)) \right\}\right).\]
    Define the functional 
    \begin{equation}
        \begin{aligned}
            \mathcal{B} :\mathcal{S}  \to \R, \quad
            \mathcal{B}(H) :=  \sum_{k=1}^n \int_0^T \psi_k(t)B_t[F_k](\mu_t) dt,
        \end{aligned}
    \end{equation} 
    where the general form for $H$ is $H (t,x) = \sum_{k=1}^n \psi_k(t) \nabla_WF_k(x,\mu_t)$.
    We have:
    \begin{itemize}
        \item $\mathcal{B}$ is linear (due to the linearity of all the $B_t$'s) and well defined, i.e. it does not depend on the representation chosen for $H\in \mathcal{S}$ indeed if \[H = \sum_{i=1}^m \zeta_i(t) \nabla_W G_i(x,\mu_t) = \sum_{k=1}^n \psi_k(t) \nabla_WF_k(x,\mu_t), \]
        this implies that for each $t\in [0,T]$ the Wasserstein gradient of the cylinder function $\mu \mapsto \sum_i \zeta_i(t) G_i(\mu) - \sum_k \psi_k(t)F_k(\mu)$ is null in $\mu_t$. Then by linearity of the integral and of each $B_t$, it holds
        \begin{align*}
             & \left| \sum_{k=1}^n \int_0^T \psi_k(t)B_t[F_k](\mu_t) dt - \sum_{i=1}^m \int_0^T \zeta_i(t)B_t[G_i](\mu_t) dt \right| \\
            & \quad \quad \quad \quad = \left| \int_0^T B_t\left[\sum_{i=1}^m \zeta_i(t) G_i - \sum_{k=1}^n \psi_k(t)F_k \right](\mu_t)dt \right| 
            \\
            & \quad \quad \quad \quad \leq \int_0^T c_t(\mu_t) \left\|\nabla_W \left(\sum_{i=1}^m\zeta_i(t) G_i - \sum_{k=1}^n \psi_k(t)F_k\right)(\cdot,\mu_t) \right\|_{L^{p'}(\mu_t)} dt = 0;
        \end{align*}
        \item for any $H \in \mathcal{S}$, it holds \[|\mathcal{B}(H)| \leq \left( 
        \int_0^T c_t^p(\mu_t) dt \right)^{\frac{1}{p}} \left\| 
        H\right\|_{L^{p'}(\mu_t\otimes dt;\R^d)}. \]
    \end{itemize}
        This implies that, by Hahn-Banach theorem, $\mathcal{B}$ can be extended to a continuous and linear functional on the set $L^{p'}(\mu_t\otimes dt; \R^d)$ and it can be represented by a Borel measurable vector field $v:[0,T]\times \R^d \to \R^d$. In particular, $v$ satisfies 
        \begin{equation}\label{eq: HB for B}\|v\|_{L^p(\mu_t\otimes dt;\R^d)} \leq\left( 
        \int_0^T c_t^p(\mu_t) dt \right)^{\frac{1}{p}} \quad \text{and} \quad \mathcal{B}(H) = \int_0^T \int_{\R^d} v(t,x) \cdot H(t,x) d\mu_t(x) dt,\end{equation}
        for all $H \in \mathcal{S}$. This implies that $\partial_t\mu_t + \operatorname{div}(v_t\mu_t) = 0$, indeed for all $\psi\in C_c^1(0,T)$ and $\phi\in C_c^1(\R^d)$, considering function $H(t,x):= \psi(t) \nabla\phi(x)$ and substituting in \eqref{eq: HB for B}, we have 
        \begin{align*}
            \int_0^T \psi'(t) & \int_{\R^d} \phi(x) d\mu_t(x) dt =  - \int_0^T \psi(t) B_t[L_\phi](\mu_t) dt 
            \\
            = &
            - \mathcal{B}(H) = - \int_0^T \psi(t) \int_{\R^d} v(t,x)\cdot \nabla\phi(x) d\mu_t(x) dt.
        \end{align*}
    Thanks to \cite[Theorem 8.3.1]{ambrosio2005gradient}, we conclude that $\boldsymbol{\mu}\in AC_T(\PP_p(\R^d))$
    \[\int_0^T |\dot\mu|_{W_p}^p(t) dt \leq \int_0^T \int_{\R^d} |v(t,x)|^p d\mu_t dt \leq \int_0^T c_t^p(\mu_t) dt<+\infty.\]
    Then, $\Lambda \in \PP_{\bar{\mathcal{A}}_p}(C_T(\PP(\R^d)))$ and $\boldsymbol{M}\in AC_T^p(\PP_p(\PP_p(\R^d)))$ follow, respectively, from Lemma \ref{lemma: well posed for Lambda-a.e. curve} and Proposition \ref{starting from Lambda}.
\end{proof}

\subsubsection{Non-local vector fields}
Before proceeding, we introduce the notion of \textit{non-local vector field}, and we see how it is connected to the one of derivation. As for derivations, we introduce 
`$L^p$-non-local vector fields'.
We will use the notation $\widetilde{M}$ and $\widetilde{M_t}\otimes dt$ introduced in Remark \ref{rem: tilde{M}} and §\ref{subsub cont curve in prob space}.

\begin{df}[$L^p$-non-local vector fields]\label{vf inducing derivation}
    Let $M\in\PP(\PP(\R^d))$ and $p\geq 1$. We say that $b:\R^d\times \PP(\R^d)\to \R^d$ is an $L^p(\widetilde{M})$-non-local vector field if 
    \begin{equation}
        \int_{\PP}\int_{\R^d} |b(x,\mu)|^p d\mu(x)dM(\mu)  <+\infty .
    \end{equation}
\end{df}

As for derivations, we will often work with a family of non-local vector field, indexed by time $t\in [0,T]$. In particular, given a Borel family of random measures $(M_t)_{t\in[0,T]}\subset \PP(\PP(\R^d))$, an $L^p(\widetilde{M}_t\otimes dt)$-non-local vector field is a Borel measurable function $b:[0,T]\times \R^d \times \PP(\R^d)\to \R^d$ such that 
\begin{equation}\label{time integrability for non-local vf}
    \int_0^T \int_{\PP}\int_{\R^d} |b(t,x,\mu)|^p d\mu(x) dM_t(\mu) dt<+\infty.
\end{equation}

\begin{df}\label{def CERM non-local vf}
    Let $(M_t)_{t\in [0,T]} \subset \PP(\PP(\R^d))$ be a curve of random measures and $b:[0,T]\times \R^d \times \PP(\R^d) \to \R^d$ an $L^p(\widetilde{M}_t\otimes dt)$ non-local vector field. We say that $\partial_t M_t + \operatorname{div}_\PP(b_tM_t) = 0$ holds, if for all $F \in \operatorname{Cyl}_c^1(\PP(\R^d))$ it holds
    \begin{equation}\label{cerm with non-local vf}
        \frac{d}{dt}\int_{\PP(\R^d)} F(\mu) dM_t(\mu) = \int_{\PP} \int_{\R^d} \nabla_W F(x,\mu) \cdot b_t(x,\mu) d\mu(x) dM_t(\mu),
    \end{equation}
    in the sense of distribution in $(0,T)$.
\end{df}

\begin{oss}\label{rem: vf induces deriv}
    Notice that a non-local vector field always induces a derivation. To be more specific, let $M\in \PP(\PP(\R^d))$ and $b:\R^d \times \PP(\R^d)\to\R^d$ be an $L^p(\widetilde{M})$ non-local vector field, then for $M$-a.e. $\mu\in \PP(\R^d)$, the following quantity defines an $L^p(M)$-derivation
    \begin{equation}
        B^{(b)}[F](\mu) := \int_{\R^d} b(x,\mu)\cdot \nabla_W F(x,\mu)d\mu(x)
    \end{equation}
    where a feasible $c$ is given by 
    \[c^{(b)}(\mu) = \|b(\cdot,\mu)\|_{L^p(\mu)}.\]
    The same relation holds between $L^p(\widetilde{M}_t\otimes dt)$-non-local vector field and $L^p(M_t\otimes dt)$-derivations. In particular, in the context of Definition \ref{def CERM non-local vf}, if $p\geq1$ and $M_0\in \PP_p(\PP_p(\R^d))$, then thanks to Proposition \ref{p-superposition}, it always holds that
    \begin{equation}\label{eq: energy ineq CERM nlvf}
        |\dot{M}|_{\mathcal{W}_p}^p(t) \leq \int_{\PP(\R^d)} \int_{\R^d} |b_t(x,\mu)|^pd\mu(x)dM_t(\mu) <+\infty \quad \text{ for a.e. }t\in[0,T].
    \end{equation}
\end{oss}

\begin{lemma}\label{lemma: CERM duality with cyl_b}
     Let $(M_t)_{t\in [0,T]} \subset \PP(\PP(\R^d))$ be a curve of random measures and $b:[0,T]\times \R^d \times \PP(\R^d) \to \R^d$ an $L^p(\widetilde{M}_t\otimes dt)$ non-local vector fields, with $p\geq 1$, satisfying $\partial M_t + \operatorname{div}_\PP(b_tM_t) = 0$. Then, \eqref{cerm with non-local vf} is satisfied also for $F \in \operatorname{Cyl}_b^1(\PP(\R^d))$.
\end{lemma}

\begin{proof}
    Let $\Psi\in C_b^1(\R^k)$ and $\phi_i \in C_b^1(\R^d)$ for $i\leq k$. Consider a cut-off function $\rho\in C^1_c(\R^d)$ such that $0\leq\rho\leq 1$, $\rho(x) = 1$ for all $|x|\leq1$ and $\rho=0$ for all $|x|\geq 2$. Then, for all $R>1$ define $\rho_R(x) := \rho(x/R)$, $\phi_{i,R}(x) = \phi_i(\rho_R(x))$ and $F_R = \Psi(L_{\phi_{1,R}},\dots,L_{\phi_{k,R}}) \in \operatorname{Cyl}_c^1(\PP(\R^d))$. We know that for all $\xi\in C_c^1(0,T)$ it holds
    \[\int_0^T \xi'(t)\int_{\PP(\R^d)} F_R(\mu) dM_t(\mu) dt = \int_0^T \xi(t)\int_{\PP} \int_{\R^d} b_t(x,\mu)\cdot \nabla_W F_R(x,\mu)d\mu(x)dM_t(\mu) dt,\]
    so we want to pass to the limit on both sides as $R\to+\infty$.
    Regarding the LHS:
    \begin{align*}
        & \left|  \int_0^T \xi'(t)\int_{\PP(\R^d)} F_R(\mu) dM_t(\mu) dt - \int_0^T \xi'(t)\int_{\PP(\R^d)} F(\mu) dM_t(\mu) dt \right|
        \\
        & \ \leq 
        \|\xi'\|_\infty \|\nabla\Psi\|_\infty \sum_{i=1}^k \int_0^T \int_{\PP} \int_{\R^d} |\phi_{i,R}(x) - \phi_i(x)| d\mu(x) dM_t(\mu) dt \to 0,
    \end{align*}
    thanks to dominated convergence theorem. Regarding the RHS:

    \begin{align*}
        \left|\int_0^T \xi(t)\int_{\PP} \int_{\R^d} b_t(x,\mu)\cdot \big(\nabla_W F_R(x,\mu)-\nabla_W F(x,\mu)\big)d\mu(x) dM_t(\mu) dt \right|\to 0
    \end{align*}
    again by dominated convergence theorem, since $\nabla_W F_R(x,\mu) \to \nabla_WF(x,\mu)$ pointwise in $(x,\mu)$, and the domination is given by $2\|\xi\|_\infty\sum_{i=1}^k\|\partial_i \Psi\|_\infty\|\nabla\phi_i\|_\infty \|b_t(\cdot,\mu)\|_{L^p(\mu)}$. 
\end{proof}

\begin{es}\label{example: N particles CERM}
    The curve $(M_t)_{t\in [0,T]}$ introduced in \eqref{eq: random measures associated to N-particle} solves the continuity equation $\partial M_t+\operatorname{div}_\PP(b_tM_t)=0$, with $b$ as in \eqref{non-local vf intro finite part system}. Indeed, using the notation $\underline{x} = (x_1,\dots,x_N)\in (\R^{d})^N$ for all $F = \Psi
    \circ L_{\Phi} \in \operatorname{Cyl}_c^1(\PP(\R^d))$ we have
    \begin{align*}
        \frac{d}{dt}\int_{\PP(\R^d)} F(\mu) dM_t(\mu)
        = &
        \frac{d}{dt} \int_{\R^{dN}} \Psi \left( \frac{1}{N} \sum_{i=1}^N \phi_1(x_i) , \dots, \frac{1}{N} \sum_{i=1}^N \phi_k(x_i)\right) dm_t(\underline{x})
        \\
        = &
        \frac{1}{N}\sum_{j=1}^k \int_{\R^{dN}} \sum_{i=1}^N \partial_j \Psi\left( L_{\Phi}(\iota(\underline{x})) \right) \nabla \phi_j(x_i)\cdot b_t(x_i,\underline{x}) dm_t(\underline{x})
        \\
        = &
        \sum_{j=1}^k \int_{\R^{dN}} \int_{\R^d} \partial_j \Psi(L_\Phi(\iota(\underline{x}))) \nabla\phi_j(x) \cdot b_t(x,\iota(\underline{x})) d[\iota(\underline{x})](x) dm_t(\underline{x})
        \\
        = & 
        \int_{\PP(\R^d)} \int_{\R^d}  \nabla_W F(x,\mu) \cdot b_t(x,\mu) d\mu(x) dM_t(\mu).
    \end{align*}
\end{es}

Consider now an absolutely continuous curve of random measures $\boldsymbol{M} \in AC_T^p(\PP_p(\PP_p(\R^d)))$, our goal is to build a non-local vector field such that the curve solves the continuity equation associated to it. Before proceeding, we need to define some useful objects in the following:
\begin{itemize}
    \item given a curve $\boldsymbol{M}\in C_T(\PP(\PP(\R^d)))$, define $\Xi^{\boldsymbol{M}} \in \mathcal{M}_+([0,T]\times \R^d \times \PP(\R^d))$ such that for all $F:[0,T] \times \R^d \times \PP(\R^d) \to [0,1]$ Borel measurable, it holds
    \begin{equation}\label{def of Xi^M}
        \int F(t,x,\mu) d\Xi^{\boldsymbol{M}}(t,x,\mu) = \int_0^T \int_{\PP} \int_{\R^d} F(t,x,\mu) d\mu(x) dM_t(\mu) dt.
    \end{equation}
    It coincides with the measure already indicated as $\widetilde{M_t}\otimes dt$, we just use this in some contexts for the sake of notation;
    \item given a measure $\mathfrak{L}\in \PP(\PP(C_T(\R^d))) $, define the measure $\Xi^{\mathfrak{L}} \in \mathcal{M}_+([0,T]\times C_T(\R^d) \times \PP(C_T(\R^d)))$ such that for all $H:[0,T]\times C_T(\R^d) \times \PP(C_T(\R^d)) \to [0,1]$ Borel measurable, it holds
    \begin{equation}\label{def of Xi^L}
        \int H(t,\gamma,\lambda) d\Xi^{\mathfrak{L}}(t,\gamma,\lambda) = \int_0^T \int \int H(t,\gamma,\lambda) d\lambda(\gamma) d\mathfrak{\L}(\lambda) dt.
    \end{equation}
\end{itemize}

\begin{prop}[From AC to CE]\label{correspondence AC to CERM}
    Let $\boldsymbol{M}=(M_t)_{t\in [0,T]} \in AC_T^p(\PP_p(\PP_p(\R^d)))$ for some $p>1$. Then, there exists an $L^p(\widetilde{M}_t\otimes dt)$ non-local vector field $b:[0,T],\times \R^d \times \PP(\R^d) \to \R^d$ such that
    \begin{equation}\label{eq: energy equality non-local vf}
        \int_\PP\int_{\R^d} |b(t,x,\mu)|^pd\mu(x) dM_t(\mu) = |\dot{M}|^p_{\mathcal{W}_p}(t) \quad \text{ for a.e. }t\in [0,T],
    \end{equation}
    and satisfying the continuity equation $\partial M_t + \operatorname{div}_\PP(b_tM_t) = 0$, in the sense of \eqref{cerm with non-local vf}. 
\end{prop}

\begin{proof}
    Using the results of Section \ref{metric structure}, let $\Lambda_{\boldsymbol{M}} \in \operatorname{Lift}(\boldsymbol{M})$ and $\mathfrak{L} := G_\sharp (\Lambda_{\boldsymbol{M}}) \in \PP_{\bar{\mathfrak{A}}_p}(\PP(C_T(\R^d)))$, and consider the measures $\Xi^{\boldsymbol{M}}$ and $\Xi^{\mathfrak{L}}$ defined above. Thanks to propositions \ref{3.8} and \ref{3.10}, we have 
    \[\int_0^T \int \int |\dot\gamma|^p(t) d\lambda(\gamma) d\mathfrak{L}(\lambda) dt = \int_0^T|\dot{M}|_{\mathcal{W}_p}^p(t) dt <+\infty,\]
    which implies that the map $(t,\gamma,\lambda) \mapsto D(t,\gamma):=\dot\gamma(t)$ is in $L^p(\Xi^{\mathfrak{L}};\R^d)$ (see Lemma \ref{meas of D(t,gamma)}). Consider the map
    \begin{equation}\label{eq: def of mathcal E}
        \begin{aligned}
            \mathcal{E}: [0,T]\times C_T(\R^d) \times \PP(C_T(\R^d)) & \to [0,T]\times \R^d \times \PP(\R^d)
            \\
            (t,\gamma,\lambda) & \mapsto (t,\gamma(t), (e_t)_\sharp  \lambda),
        \end{aligned}
    \end{equation}
    and notice that $\Xi^{\boldsymbol{M}} = \mathcal{E}_\sharp  \Xi^{\mathfrak{L}}$. Then, thanks to \cite[Lemma 17.3]{ambrosio2021lectures} (see also  Remark \ref{remark density push-forward}), there exists a function $b:[0,T]\times \R^d \times \PP(\R^d) \to \R^d$ such that 
    \begin{equation}\label{eq: def of b thanks to lemma}
    \mathcal{E}_\sharp (D \ \Xi^{\mathfrak{L}}) = b \ \Xi^{\boldsymbol{M}}.
    \end{equation}
    We show that $b$ is the non-local vector field we are looking for: indeed 
    \begin{equation}\label{energy inequality vf}
    \begin{aligned}
        \int_0^T \int_\PP \int_{\R^d} |b(t,x,\mu)|^p d\mu(x) dM_t(\mu) dt 
        = & 
        \int  |b(t,x,\mu)|^pd\Xi^{\boldsymbol{M}}(t,x,\mu)dt 
        \leq 
        \int |\dot\gamma(t)|^p d\Xi^{\mathfrak{L}}(t,\gamma,\lambda) <+\infty,
    \end{aligned}
    \end{equation}
    again thanks to \cite[Lemma 17.3]{ambrosio2021lectures}. Moreover, for any $\psi\in C^1_c(0,T)$ and $F = \Psi\circ L_\Phi \in \operatorname{Cyl}_c^1(\PP(\R^d))$, it holds
    \begin{align*}
        \int_0^T \psi'(t) & \int_\PP F(\mu) dM_t(\mu) dt 
        =  
        \int \psi'(t) F(\mu) d\Xi^{\boldsymbol{M}}(t,x,\mu) 
        = 
        \int \psi'(t) F((e_t)_\sharp \lambda) d\Xi^{\mathfrak{L}}(t,\gamma,\lambda)  
        \\
        = &
        \int \int \int_0^T \psi'(t) \Psi\left(\int \phi_1(\gamma(t)) d\lambda(\gamma), \dots , \int \phi_k(\gamma(t)) d\lambda(\gamma)\right) dt d\lambda(\gamma) d\mathfrak{L}(\lambda) 
        \\
        = &
        - \int \int \int_0^T \psi(t) \sum_{i=1}^k \partial_i\Psi(L_\Phi((e_t)_\sharp \lambda)) \nabla\phi_i(\gamma(t)) \cdot \dot\gamma(t) \ dtd\lambda(\gamma) d\mathfrak{L}(\lambda)
        \\
        = & 
        - \int \psi(t) \nabla_W F \big(\gamma(t),(e_t)_\sharp \lambda\big) \cdot d\big(D \  \Xi^{\mathfrak{L}}\big)(t,\gamma,\lambda)
        \\
        = & 
        - \int \psi(t) \nabla_W F(x,\mu) \cdot d\big(b \ \Xi^{\boldsymbol{M}} \big)(t,x,\mu)
        \\
        = &
        - \int_0^T \psi(t) \int_\PP \int_{\R^d} \nabla_W F (x,\mu) \cdot b(t,x,\mu) d\mu(x) dM_t(\mu) dt,
    \end{align*}
    where in the second last equality we exploited the definition of $\mathcal{E}$ and the characterization of $b$ given by \eqref{eq: def of b thanks to lemma}. Putting together \eqref{eq: energy ineq CERM nlvf}, \eqref{energy inequality vf} and Proposition \eqref{well def from L to M}, then \eqref{eq: energy equality non-local vf} follows.
\end{proof}

\begin{oss}\label{rem: nlvf of minimal energy}
Notice that, thanks to \eqref{eq: energy equality non-local vf}, the vector field we built is minimal in an $L^p$-sense, and because of the strict convexity of $|\cdot |^p$ for $p>1$, such vector field is unique, in the sense that any other $L^p(M_t\otimes dt)$-non-local vector field $\tilde{b}$ satisfying $\partial_tM_t + \operatorname{div}_\PP(\tilde{b}_tM_t)=0$ and \eqref{eq: energy equality non-local vf} coincides with $b$ for $\Xi^{\boldsymbol{M}}$-a.e. $(t,x,\mu)\in[0,T]\times \R^d\times \PP(\R^d)$.
\end{oss}

A consequence of the characterization of absolutely continuous curves with curves that solve a continuity equation, is a Benamou-Brenier-type formula for random measures. 

\begin{teorema}[Benamou-Brenier formula]\label{thm: BB2}
    Let $p>1$. For all $M_0,M_1\in \PP_p(\PP_p(\R^d))$ it holds
    \begin{equation}\label{eq: BB}
    \begin{aligned}
        \mathcal{W}_p^p(M_0,M_1) = \operatorname{min}\bigg\{ \int_0^1 \int_{\PP}\int_{\R^d} |b_t(x,\mu)|^p d\mu(x)dM_t(\mu)dt \, : \, \partial_t M_t + \operatorname{div}_\PP(b_tM_t) = 0 \bigg\}.
    \end{aligned}
    \end{equation}
\end{teorema}

\begin{proof}
    Thanks to \eqref{eq: energy ineq CERM nlvf}, all the competitors for the right-hand side satisfy the inequality 
    \[\mathcal{W}_p^p(M_0,M_1) \leq \int_0^1 \int_\PP\int_{\R^d} |b_t(x,\mu)|^p d\mu(x)dM_t(\mu)dt.\]
    On the other hand, Lemma \ref{lemma: geodesics through opt couplings} gives the existence of a constant speed geodesic $(M_t)_{t\in[0,1]}\in C([0,1],\PP_p(\PP_p(\R^d)))$. Then, Proposition \ref{correspondence AC to CERM} gives the existence of a non-local vector field $b:[0,1]\times \R^d \times \PP(\R^d) \to \R^d$ satisfying $\partial_tM_t + \operatorname{div}_\PP(b_tM_t) = 0$ and \eqref{eq: energy equality non-local vf}, from which it follows that the curve $(M_t)$ and the non-local vector field $b$ are optimal for \eqref{eq: BB}.
\end{proof}

\subsection{The tangent and cotangent bundle to $\PP_p(\PP_p(\R^d))$}\label{subsec: tangent}
In this subsection, we define the tangent and cotangent bundle as closure in a suitable Lebesgue space of the Wasserstein gradient of cylinder functions. Then, following the same argument of \cite[§8.4]{ambrosio2005gradient} we characterize the non-local vector fields of minimal energy (see Remark \ref{rem: nlvf of minimal energy}) as elements of the tangent bundle.

Before proceeding, let us recall the duality pairing map between Lebesgue spaces: given any measurable space $(X,\mathcal{F})$ endowed with a finite positive measure $\sigma$, then for any $p\in (1,+\infty)$ the duality pairing is defined as
\begin{equation}
    j_p: L^p(\sigma;\R^d) \to L^{p'}(\sigma;\R^d), \quad j_p(V)(x):= \begin{cases}
        |V(x)|^{p-2}V(x) \quad & \text{ if } V(x)\neq 0
        \\
        0 & \text{ otherwise }
    \end{cases}
\end{equation}


\begin{df}
    Let $p\in(1,+\infty)$ and $M\in \PP_p(\PP_p(\R^d))$ and recall the definition of $\widetilde{M}\in \PP(\R^d\times \PP(\R^d))$ from Remark \ref{rem: tilde{M}}. Then we define, respectively, the cotangent and the tangent space of $\PP_p(\PP_p(\R^d))$ at $M$ as
    \begin{equation}\label{eq: cotangent}
        \operatorname{CoTan}_M \PP_p(\PP_p(\R^d)) := \operatorname{Clos}_{L^{p'}(\widetilde{M};\R^d)} \left\{ \nabla_W F \ : \ F\in\operatorname{Cyl}_c(\PP(\R^d)) \right\} \subseteq L^{p'}(\widetilde{M};\R^d);
    \end{equation}
    \begin{equation}\label{eq: tangent}
        \operatorname{Tan}_M \PP_p(\PP_p(\R^d)) := \operatorname{Clos}_{L^{p}(\widetilde{M};\R^d)} \left\{ j_{p'}\left(\nabla_W F\right) \ : \ F\in\operatorname{Cyl}_c(\PP(\R^d)) \right\} \subseteq L^{p}(\widetilde{M};\R^d).
    \end{equation}
\end{df}
Notice that the tangent and the cotangent space are in duality by the maps $j_p$ and $j_{p'}$, i.e. $\operatorname{Tan}_M \PP_p(\PP_p(\R^d)) = j_{p'}(\operatorname{CoTan}_M \PP_p(\PP_p(\R^d)))$ and $\operatorname{CoTan}_M \PP_p(\PP_p(\R^d)) = j_p(\operatorname{Tan}_M \PP_p(\PP_p(\R^d)))$.

\begin{oss}
    The tangent space could be defined only considering infinitely-smooth cylinder functions, that is 
    \[\operatorname{Tan}_M \PP_p(\PP_p(\R^d)) = \operatorname{Clos}_{L^{p}(\widetilde{M};\R^d)} \left\{ j_{p'}\left(\nabla_W (\Psi\circ L_\Phi)\right)  :   k\in \mathbb{N}, \, \Psi\in C_c^\infty(\R^k), \, \Phi \in C_c^\infty(\R^d;\R^k) \right\},\]
    since any $C^1_c$ function can be uniformly approximated by functions in $C_c^\infty$. 
\end{oss}

\begin{lemma}\label{lemma: minimality in the tangent}
    Let $p>1$, $M\in \PP_p(\PP_p(\R^d))$ and $b\in L^p(\widetilde{M};\R^d)$. Then $b\in \operatorname{Tan}_M \PP_p(\PP_p(\R^d))$ if and only if $\|b+ b'\|_{L^p(\widetilde{M};\R^d)} \geq \|b\|_{L^p (\widetilde{M};\R^d)}$ for all $b'\in L^p(\widetilde{M};\R^d)$ such that $\langle b',\omega \rangle = 0$ for all $\omega \in \operatorname{CoTan}_M \PP_p(\PP_p(\R^d))$.    
    \\
    In particular, for every $b \in L^p(\widetilde{M};\R^d)$ there exists a unique element $\Pi(b) \in \operatorname{Tan}_M \PP_p(\PP_p(\R^d))$ in the set of vector fields $b' \in L^p(\widetilde{M};\R^d)$ satisfying $\langle b,\omega \rangle = \langle b',\omega \rangle$ for all $\omega \in \operatorname{CoTan}_M \PP_p(\PP_p(\R^d))$ and $\Pi(b)$ is the element of minimal norm in this class.
\end{lemma}

\noindent Notice that the condition that $\langle b',\omega \rangle = 0$ for all $\omega \in \operatorname{CoTan}_M \PP_p(\PP_p(\R^d))$ is equivalent to ask that $\langle b',\nabla_W F \rangle = 0$ for all $F \in \operatorname{Cyl}_c(\PP(\R^d))$, which, consistently with our notation, can be written in the compact form $\operatorname{div}_{\PP}(b'M) = 0$.

\begin{proof}
    As in \cite[Lemma 8.4.2]{ambrosio2005gradient}, by convexity of the $L^p$-norm to the power $p$ and the fact that $pj_{p}(b)$ belongs to its subdifferential at the function $b$, we have that $\|b+b'\|_{L^p}^p \geq \|b\|_{L^p}$ for all $b'$ satisfying $\operatorname{div}_{\PP}(b'M)=0$ if and only if $\langle j_p(b),b' \rangle = 0$ for all $b'$ as before, and by the Hahn-Banach theorem this happens if and only if $j_p(b) \in \operatorname{CoTan}_M \PP_p(\PP_p(\R^d))$. This is equivalent to say $b = j_{p'}(j_p(b)) \in \operatorname{Tan}_M \PP_p(\PP_p(\R^d))$. The last part follows from the fact that the class of vector fields $b'$ satisfying $\langle b,\omega \rangle = \langle b' ,\omega \rangle$ for all $\omega \in \operatorname{CoTan}_M \PP_p(\PP_p(\R^d))$ is closed and convex, so that by strict convexity of the $L^p$-norm there exists a unique element of minimum norm in it and by the previous characterization it belongs to $\operatorname{Tan}_M \PP_p(\PP_p(\R^d))$.
\end{proof}

\begin{prop}\label{prop:tangent}
    Let $\boldsymbol{M}=(M_t)_{t\in[0,T]} \in C_T(\PP(\PP(\R^d)))$ such that $M_0 \in \PP_p(\PP_p(\R^d))$ and $b:[0,T]\times \R^d \times \PP(\R^d)\to \R^d$ be an $L^p(\widetilde{M}_t\otimes dt)$-non-local vector field. Assume that the continuity equation $\partial_t M_t +\operatorname{div}_{\PP}(b_tM_t) = 0$ holds, in the sense of Definition \ref{def CERM non-local vf}. Then $M_t\in \PP_p(\PP_p(\R^d))$ for all $t\in [0,T]$ and the following are equivalent:
    \begin{enumerate}
        \item[(1)] $b_t \in \operatorname{Tan}_{M_t}\PP_p(\PP_p(\R^d))$ for a.e. $t\in [0,T]$;
        \item[(2)] $\int_{\PP(\R^d)}\int_{\R^d} |b_t(x,\mu)|^p d\mu(x)dM_t(\mu) \leq |\dot{\boldsymbol{M}}|_{\mathcal{W}_p}^p(t)$ for a.e. $t\in [0,T]$;
        \item[(3)] $\int_0^T\int_{\PP(\R^d)}\int_{\R^d} |b_t(x,\mu)|^p d\mu(x)dM_t(\mu)dt = \int_0^T |\dot{\boldsymbol{M}}|_{\mathcal{W}_p}^p(t)dt$.
    \end{enumerate}
\end{prop}

\begin{proof}
The equivalence between (2) and (3) follows from \eqref{eq: energy ineq CERM nlvf}. Regarding (1)$\implies$(2), consider $\tilde{b}$ be the $L^p(\widetilde{M}_t\otimes dt)$-non-local vector field given by Proposition \ref{correspondence AC to CERM}. The goal is to show that $b_t=\tilde{b}_t$ for a.e. $t\in[0,T]$, as functions of $L^p(\widetilde{M}_t)$. By \eqref{eq: energy equality non-local vf} and \eqref{eq: energy ineq CERM nlvf}, it holds that 
\begin{equation}\label{eq: ineq proposition}
\int_{\PP}\int_{\R^d} |\tilde{b}_t(x,\mu)|^p d\mu(x) dM_t(\mu) \leq \int_{\PP}\int_{\R^d}|b_t(x,\mu)|^pd\mu(x)dM_t(\mu) \quad \text{ for a.e. }t\in[0,T].
\end{equation}
Moreover, the curve $(M_t)_{t\in[0,T]}$, by assumption and by construction of $\tilde{b}$, satisfies the continuity equations $\partial_tM_t + \operatorname{div}_{\PP}(b_tM_t) = 0$ and $\partial_tM_t + \operatorname{div}_{\PP}(\tilde{b}_t M_t) = 0$, which implies that for all $\xi\in C_c^1(0,T)$ and $F\in \operatorname{Cyl}_c(\PP(\R^d))$ it holds
\[\int_0^T \xi(t)\int_{\R^d\times\PP(\R^d)} b_t\cdot \nabla_WF d\widetilde{M}_t \ dt = \int_0^T \xi(t)\int_{\R^d\times\PP(\R^d)} \tilde{b}_t\cdot \nabla_WF d\widetilde{M}_t \ dt. \]
Localizing this equality in time, it holds 
\[\int_{\R^d\times\PP(\R^d)} b_t\cdot \nabla_WF d\widetilde{M}_t  = \int_{\R^d\times\PP(\R^d)} \tilde{b}_t\cdot \nabla_WF d\widetilde{M}_t  \quad \text{ for a.e. }t\in[0,T],\]
and together with \eqref{eq: ineq proposition}, since $b_t \in \operatorname{Tan}_{M_t}\PP_p(\PP_p(\R^d))$ for a.e. $t\in [0,T]$, we can apply Lemma \ref{lemma: minimality in the tangent} to conclude that $b_t(x,\mu) = \tilde{b}_t(x,\mu)$ for $\widetilde{M}_t$-a.e. $(x,\mu)$ and for a.e. $t\in[0,T]$.
\\
Regarding (2)$\implies$(1), let us introduce the auxiliar space
\[\mathcal{V}:= \operatorname{Clos}_{L^p(\Xi^{\boldsymbol{M}};\R^d)} \operatorname{Span}\left\{ j_p\left( \xi(t)\nabla_W F(x,\mu) \right) \ : \ \xi \in C_c^1(0,T), \ F \in \operatorname{Cyl}_c(\PP(\R^d)) \right\}.\]
Following the argument of Lemma \ref{lemma: minimality in the tangent}, it is not hard to prove that a non-local vector field $v:[0,T]\times \R^d \times \PP(\R^d) \to \R^d$ belongs to $\mathcal{V}$ if and only if $\|v+v'\|_{L^p} \geq \|v\|_{L^p}$ for all $v'\in L^{p}(\Xi^{\boldsymbol{M}};\R^d)$ satisfying 
\[\int_0^T \xi(t)\int_{\PP(\R^d)} \int_{\R^d} \nabla_W F(x,\mu)\cdot v'(t,x,\mu) d\mu(x) dM_t(\mu) dt = 0.\]
Thus, if we assume condition (2), because of its equivalence to (3) and the minimality given by Remark \ref{rem: nlvf of minimal energy}, it holds $b\in \mathcal
V$. Thus we conclude proving that $v\in\mathcal{V} \implies v_t \in \operatorname{Tan}_{M_t}\PP_p(\PP_p(\R^d))$ for a.e. $t\in[0,T]$. This easily follows by a pointwise argument: fix two sequences of functions $\xi_n\in C_c^1(0,T)$ and $F_n\in\operatorname{Cyl}_c(\PP(\R^d))$ such that $j_p(\xi_n \nabla_W F_n) \to v$ in $L^p(\Xi^{\boldsymbol{M}};\R^d)$. Up to consider a subsequence, it holds that 
\[j_p(\xi_n \nabla_WF_n)(t,\cdot,\cdot) = |\xi_n(t)|^{p-2}\xi_n(t) |\nabla_W F_n(\cdot,\cdot)|^{p-2} \nabla_WF_n(\cdot,\cdot) \to v(t,\cdot,\cdot) \quad\text{ in }L^p(\widetilde{M}_t),\]
for a.e. $t\in[0,T]$. In particular, fixing a time $t\in [0,T]$ for which the one above holds, the sequence of cylinder functions $F_{t,n} \in \operatorname{Cyl}_c(\PP(\R^d))$ defined by $F_{t,n}(x,\mu):= \xi_n(t) F_n(x,\mu)$ is such that $j_p(\nabla_W F_{t,n}) \to v_t$ in $L^p(\widetilde{M}_t)$. By definition of tangent space, this proves that if $v\in \mathcal{V}$, then $v_t \in \operatorname{Tan}_{M_t} \PP_p(\PP_p(\R^d))$ for a.e. $t\in[0,T]$. 
\end{proof}

\subsection{Derivations and vector fields}

In this subsection, we show that any family of derivations $(B_t)_{t\in[0,T]}$ is induced by a family of non-local vector fields, as in Remark \ref{vf inducing derivation}, whenever the continuity equation $\partial_tM_t + \operatorname{div}_\PP(B_tM_t) = 0$ is satisfied.
\begin{teorema}\label{derivation induced by a vf theorem}
    Let $(M_t)_{t\in [0,T]}\in C_T(\PP(\PP(\R^d)))$ and $(B_t)_{t\in [0,T]}$ an $L^p(M_t\otimes dt)$-derivation, such that 
    $\partial M_t + \operatorname{div}_{\PP} (B_tM_t) = 0$.
    Then there exists an $L^p(\widetilde{M}_t\otimes dt)$-non-local vector field $b:[0,T]\times \R^d \times \PP(\R^d)$ such that for $M_t\otimes dt$-a.e. $(\mu,t)$ it holds
    \begin{equation}\label{vf inducing derivation thesis theorem}
        B_t[F](\mu) = \int_{\R^d} b(t,x,\mu)\cdot \nabla_W F(x,\mu) d\mu(x), \quad \forall F\in \operatorname{Cyl}_c^1(\PP(\R^d)).
    \end{equation}
\end{teorema}

The proof of this result is very similar to the one of Proposition \ref{correspondence AC to CERM}. Indeed, putting together Proposition \ref{p-superposition} and Proposition \ref{correspondence AC to CERM}, we have a vector field $v:[0,T]\times \R^d \times \PP(\R^d)\to \R^d$ such that for any $F\in \operatorname{Cyl}_c^1(\PP(\R^d))$ and for a.a. $t\in (0,T)$, it holds 
    \[\int_{\PP}\int_{\R^d} v(t,x,\mu)\cdot \nabla_W F(x,\mu) d\mu(x) dM_t(\mu) = \int_{\PP} B_t[F](\mu) dM_t(\mu).\]
    \\
    The non-trivial part is to localize this equality with respect to the variable $\mu$, to prove \eqref{vf inducing derivation thesis theorem}. In the proof, we are going to see how this localization can be done in various steps, mainly in steps 2 and 3 below.

\begin{proof}[Proof of Theorem \ref{derivation induced by a vf theorem}]
    \textbf{Step 1}: \textit{superposition and nested metric lifting.} Using first Proposition \ref{p-superposition} and then Theorem \ref{metric theorem from Lambda to mathfrak{L}}, we obtain a probability measure $\mathfrak{L}\in\PP_{\mathfrak{A}_p}(\PP(C_T(\R^d)))$ that satisfies: 
    \begin{itemize}
        \item[(i)] $(E_t)_\sharp \mathfrak{L} = M_t$;
        \item[(ii)]$\mathfrak{L}$-a.e. $\lambda\in \PP(C_T(\R^d))$ is such that $\mu_t:= (e_t)_\sharp  \lambda$ solves the continuity equation \eqref{CE classic}, since $E_\sharp  \mathfrak{L}  = \Lambda$; \label{condition ii L} 
        \item[(iii)]
        $\int_0^T \int \int |\dot\gamma|^p(t)  d\lambda(\gamma) d\mathfrak{L}(\lambda) dt<+\infty.$
    \end{itemize}
    Consider its associated measure $\Xi^{\mathfrak{L}} \in \mathcal{M}_+([0,T]\times C_T(\R^d) \times \PP(C_T(\R^d)))$ as in \eqref{def of Xi^L}.
    \\
    \textbf{Step 2}: \textit{localization step}. Consider the function $\mathcal{F}: (t,\gamma,\lambda) \mapsto (t,\gamma(t),\lambda) \in [0,T]\times \R^d \times \PP(C_T(\R^d))$, and notice that $\mathcal{F}_\sharp \Xi^{\mathfrak{L}} = \Xi^{\mathfrak{L},\boldsymbol{M}}$, where $\Xi^{\mathfrak{L},\boldsymbol{M}}\in \mathcal{M}_+\big( [0,T]\times \R^d \times \PP(C_T(\R^d)) \big)$ is defined such that for each $G:[0,T]\times \R^d \times \PP(C_T(\R^d)) \to [0,1]$ Borel measurable it holds
    \[\int G(t,x,\lambda) d\Xi^{\mathfrak{L},\boldsymbol{M}}(t,x,\lambda) = \int_0^T \int \int G(t,\gamma(t),\lambda) d\lambda(\gamma) d\mathfrak{\L}(\lambda) dt.\]
    Since the function $\dot\gamma \in L^p(\Xi^{\mathfrak{L}})$, then thanks to \cite[Lemma 17.3]{ambrosio2021lectures}, it holds that there exists a function $\tilde{b} \in L^p(\Xi^{\mathfrak{L},\boldsymbol{M}};\R^d)$ such that 
    \[\mathcal{F}_\sharp (D \ \Xi^{\mathfrak{L}})(dt,dx,d\lambda) = \tilde{b}(t,x,\lambda) \ \Xi^{\mathfrak{L},\boldsymbol{M}}(dt,dx,d\lambda),\]
    where $D(t,\gamma,\lambda)= \dot\gamma(t)$ as before.
    
    \textit{Claim}: for $\mathfrak{L}$-a.e. $\lambda$, $\mu_t:=(e_t)_\sharp \lambda$ solves $\partial_t\mu_t + \operatorname{div}_x(\tilde{b}(t,x,\lambda) \mu_t) = 0$.
    \\
    Indeed, for any $F:\PP(C_T(\R^d))\to [0,1]$ Borel measurable, $\psi \in C_c^1(0,T)$ and $\phi\in C_c^1(\R^d)$ it holds
    \begin{align*}
        \int F(\lambda) & \left[\int_0^T\psi'(t) \int_{\R^d}\phi(x) d\mu_t(x) dt\right]d\mathfrak{L}(\lambda) = 
        \\
        & = 
        \int F(\lambda)  \left[\int_0^T\psi'(t) \int_{C_T(\R^d)}\phi(\gamma(t)) d\lambda(\gamma) dt\right]d\mathfrak{L}(\lambda) = 
        \\
        & = 
        - \int F(\lambda) \left[\int_0^T\psi(t) \int_{C_T(\R^d)}\nabla\phi(\gamma(t))\cdot \dot\gamma(t) d\lambda(\gamma) dt\right]d\mathfrak{L}(\lambda) 
        \\
        & = 
        - \int F(\lambda) \psi(t) \nabla\phi(\gamma(t))\cdot 
        d\left(D  \Xi^{\mathfrak{L}} \right)(t,\gamma,\lambda)
        \\
        & = 
        - \int F(\lambda) \psi(t) \nabla\phi(x)\cdot \tilde{b}(t,x,\lambda)
        d\Xi^{\mathfrak{L},\boldsymbol{M}} (t,x,\lambda)
        \\
        & = 
        \int F(\lambda) \left[-
        \int_0^T \psi(t) \int_{\R^d} \nabla \phi(x) \cdot \tilde{b}(t,x,\lambda) d\mu_t(x) dt \right] d\mathfrak{L}(\lambda).
    \end{align*}

\textbf{Step 3}: \textit{definition of the non-local vector field}. Define the continuous map $\mathcal{G}:[0,T]\times \R^d \times \PP(C_T(\R^d)) \to [0,T]\times \R^d \times \PP(\R^d)$ as
    \[\mathcal{G}(t,x,\lambda) = (t,x,(e_t)_\sharp \lambda),\]
    and notice that $\mathcal{G}_\sharp \Xi^{\mathfrak{L},M} := \Xi^{\boldsymbol{M}}$.
    At this point, consider the disintegration of $\Xi^{\mathfrak{L},\boldsymbol{M}}$ w.r.t. $\mathcal{G}$, i.e. the Borel map $(t,x,\mu) \mapsto \Xi^{\mathfrak{L},\boldsymbol{M}}_{t,x,\mu}$, that is well-defined $\Xi^{\boldsymbol{M}}$-almost everywhere. Then define the non-local vector field as
    \[b(t,x,\mu) = \int \tilde{b}(s,z,\lambda) d\Xi^{\mathfrak{L},\boldsymbol{M}}_{t,x,\mu}(s,z,\lambda),\]
    which is measurable thanks to the measurability of $\tilde{b}$ and of the disintegration $(t,x,\mu) \mapsto \Xi^{\mathfrak{L},\boldsymbol{M}}_{t,x,\mu}$ (see Lemma \ref{measurability int g d mu}). By construction, it's easy to verify that $b\in L^p(\Xi^{\boldsymbol{M}};\R^d)$.

    \textbf{Step 4}: \textit{representation for the disintegration} $\Xi^{\mathfrak{L},\boldsymbol{M}}_{t,x,\mu}$. For any $t\in [0,T]$, disintegrate the measure $\mathfrak{L}$ with respect to the map $E_t= (e_t)_\sharp $, to obtain that there exists a family of probability measures $\{\mathfrak{L}_{t,\mu}\}_{\mu\in \PP(\R^d)}\subset \PP(\PP(C_T(\R^d)))$ such that 
    \[\mathfrak{L} = \int_{\PP(\R^d)} \mathfrak{L}_{t,\mu} dM_t(\mu).\]
    Then $\Xi^{\mathfrak{L},\boldsymbol{M}}_{t,x,\mu} = \delta_t\otimes \delta_x \otimes \mathfrak{L}_{t,\mu}$ for $\Xi^{\boldsymbol{M}}$-a.e. $(t,x,\mu)$. Indeed for any $G:[0,T]\times \R^d \times \PP(C_T(\R^d))\to [0,+\infty]$ Borel measurable map, we have
    \begin{align*}
        \int G d\Xi^{\mathfrak{L},\boldsymbol{M}} 
        = &
        \int_0^T \int \left( \int G(t,x,\lambda) d\big((e_t)_\sharp \lambda\big)(x) \right) d\mathfrak{L}(\lambda) dt
        \\
        = &
        \int_0^T \left[ \int_{\PP(\R^d)} \left( 
        \int_{\R^d} \int G(t,x,\lambda) d\mathfrak{L}_{t,\mu}(\lambda) d\mu(x) \right) dM_t(\mu)\right] dt
        \\
        = &
        \int \left[ \int G(t,x,\lambda) d\mathfrak{L}_{t,\mu}(\lambda) \right] d\Xi^{\boldsymbol{M}}(t,x,\mu) 
        \\
        = &
        \int \left[\int \int \int G(s,z,\lambda) \ d\delta_t(s) d\delta_x(z) d\mathfrak{L}_{t,\mu}(\lambda)\right] d\Xi^{\boldsymbol{M}}(t,x,\mu),
    \end{align*}
    so we conclude by mean of the uniqueness of the disintegration. In particular, for $\Xi^{\boldsymbol{M}}$-a.e. $(t,x,\mu)$, it holds
    \begin{equation}\label{repr vf with new disintegration}
        b(t,x,\mu) = \int \tilde{b}(t,x,\lambda) d\mathfrak{L}_{t,\mu}(\lambda).
    \end{equation}

    \textbf{Step 5}: \textit{conclusion}. We are left to prove that the non-local vector field $b$ satisfies \eqref{vf inducing derivation thesis theorem}. From Step 2 and the properties of $\mathfrak{L}$, we know that for $\mathfrak{L}$-a.e. $\lambda$, for any $\psi\in C_c^1(0,T)$ and $\phi\in C_c^1(\R^d)$, it holds
    \begin{align*}
        \int_0^T \psi(t) B_t[L_\phi]\big((e_t)_\sharp \lambda\big) dt & = - \int_0^T \psi'(t)\int_{\R^d}\phi(x)d\big((e_t)_\sharp \lambda\big)(x) dt 
        \\
        & = \int_0^T \psi(t) \int_{\R^d} \tilde{b}(t,x,\lambda) \cdot \nabla \phi(x) d\big((e_t)_\sharp \lambda\big)(x) dt,
    \end{align*}
    which implies, thanks to Lemma \ref{chain rule}, that for all $F\in \operatorname{Cyl}_c^1(\PP(\R^d))$ it holds
    \[B_t[F]\big((e_t)_\sharp \lambda\big) = \int_{\R^d} \tilde{b}(t,x,\lambda)\cdot \nabla_WF(x,(e_t)_\sharp \lambda) d\big((e_t)_\sharp \lambda\big)(x),\]
    for a.e. $t\in [0,T]$ and $\mathfrak{L}$-a.e. $\lambda$. At this point, for a.e. $t\in [0,T]$ and $M_t$-a.e. $\mu\in \PP(\R^d)$, we can integrate both sides w.r.t. $\mathfrak{L}_{t,\mu}$: since $(e_t)_\sharp \lambda = \mu$ for $\mathfrak{L}_{t,\mu}$-a.e. $\lambda$, the left hand side is constant, while on the right hand side, we can switch the order of integration to obtain, thanks to \eqref{repr vf with new disintegration}, that
    \begin{align*}
        B_t[F](\mu) = \int_{\R^d} \left(\int \tilde{b}(t,x,\lambda) d\mathfrak{L}_{t,\mu}(\lambda)\right)\cdot \nabla_W F(x,\mu) d\mu(x)
        = 
        \int_{\R^d} b(t,x,\mu) \cdot \nabla_W F(x,\mu) d\mu(x),
    \end{align*}
    for a.e. $M_t\otimes dt$-a.e. $(\mu,t)$ and for all $F\in \operatorname{Cyl}_c^1(\PP(\R^d))$.
    \end{proof}

\section{Nested superposition principle}\label{superposition}
The main goal of this section is to prove Theorem \ref{main theorem}. The strategy is similar to the one used in Section \ref{metric structure} to prove Theorem \ref{metric superposition}. In particular, given a Borel measurable non-local vector field $b:[0,T]\times \R^d \times \PP(\R^d) \to \R^d$, the main objects of our study are:
\begin{enumerate}[label=(\roman*)]
    \item a curve of random measures $\boldsymbol{M} = (M_t)_{t\in [0,T]}\in C_T(\PP(\PP(\R^d)))$ for which $b$ is an $L^1(\widetilde{M}_t\otimes dt)$-non-local vector field and such that, according to Definition \ref{def CERM non-local vf}, it holds 
    \begin{equation}
        \partial_t M_t+\operatorname{div}_\PP(b_tM_t) = 0;
    \end{equation}
    \item\label{item ii} a probability measure $\Lambda \in \PP(C_T(\PP(\R^d)))$ satisfying $\int \int_0^T  \int |b(t,x,\mu_t)| d\mu_t(x) dt d\Lambda(\boldsymbol{\mu}) <+\infty$ and concentrated over curves of measures $\boldsymbol{\mu}=(\mu_t)_{t\in [0,T]}\in C_T(\PP(\R^d))$ that solves
    \begin{equation}
        \partial_t \mu_t + \operatorname{div}(b_t(\cdot,\mu_t)\mu_t) = 0;
    \end{equation}
    \item\label{item iii} a probability measure $\mathfrak{L}\in \PP(\PP(C_T(\R^d)))$ satisfying $\int \int a_1(\gamma) d\lambda(\gamma)d\mathfrak{L}(\lambda)<+\infty$ (see Definition \ref{ac energies}) and concentrated over $\lambda \in \PP(C_T(\R^d))$ that, in turn, are concentrated over $\gamma\in AC_T(\R^d)$ that are solutions of 
    \begin{equation}\label{particle systems section 5}
        \dot\gamma(t) = b(t,\gamma_t,(e_t)_\sharp \lambda).
    \end{equation}
\end{enumerate}

We introduce two sets associated with a generic non-local vector field, that will play a fundamental role in the proof of Theorem \ref{main theorem}.

\begin{df}\label{def: CE and SCE}
    Let $b:[0,T]\times \R^d \times \PP(\R^d) \to \R^d$ be a Borel map. Then, the set $\operatorname{CE}(b)\subset C_T(\PP(\R^d))$ of solutions to the continuity equation driven by the non-local vector field $b$ is defined by
    \begin{equation}\label{CEb}
\begin{aligned}
    \CE(b):= \bigg\{(\mu_t)_{t\in[0,T]} \in C_T(\PP(\R^d)) \ : \ 
    \int \int |b_t(x,\mu_t)|d\mu_t(x) dt <+\infty, & \\
    \partial_t\mu_t + \operatorname{div}(b_t(\cdot,\mu_t)\mu_t)=0& \bigg\}.
\end{aligned}
\end{equation}
The set $\operatorname{SPS}(b)\subset \PP(C_T(\R^d))$ of superposition solutions of the particle systems \eqref{particle systems section 5} is defined by
\begin{equation}\label{SPSb}
\begin{aligned}
    \operatorname{SPS}(b):= \{\lambda \in & \PP(C_T(\R^d)) \ : \ \int \int |b_t(\gamma(t),(e_t)_\sharp \lambda)| d\lambda(\gamma) dt<+\infty ,
    \\
    &\lambda \big( \operatorname{AC}_T(\R^d) \big) =1, \ \dot{\gamma}(t) = b_t(\gamma_t,(e_t)_\sharp \lambda) \ \mathcal{L}^1_T\otimes\lambda\text{-a.e.} \}.
\end{aligned}
\end{equation}
\end{df}

Note that the properties of $\Lambda\in\PP(C_T(\PP(\R^d)))$ and $\mathfrak{L}\in \PP(\PP(C_T(\R^d)))$ listed in \ref{item ii} and \ref{item iii}, can be summarized by saying that $\Lambda$ is concentrated over $\operatorname{CE}(b)$ and $\mathfrak{L}$ is concentrated over $\operatorname{SPS}(b)$. A crucial point will be the Borel measurability of these sets, for which we refer to Proposition \ref{meas of CE_b} and Proposition \ref{meas of SPSb}. 

Let us start noticing that the hierarchy described in the introduction between general objects $\boldsymbol{M}\in C_T(\PP(\PP(\R^d)))$, $\Lambda \in \PP(C_T(\PP(\R^d)))$ and $\mathfrak{L} \in \PP(\PP(C_T(\R^d)))$, is preserved when we require them to satisfy the conditions listed above. Before proceeding recall that 
\[E:\PP(C_T(\R^d)) \to C_T(\PP(\R^d)), \quad E(\lambda):= ((e_t)_\sharp \lambda)_{t\in [0,T]};\]
\[\mathfrak{e}_t:C_T(\PP(\R^d)) \to \PP(\R^d), \ \  \mathfrak{e}_t(\boldsymbol{\mu}) := \mu_t;\quad E_t : \PP(C_T(\R^d)) \to \PP(\R^d), \ \  E_t(\lambda) = (e_t)_\sharp \lambda.\]

\begin{prop}\label{L to Lambda section 5}
    If $\lambda \in \operatorname{SPS}(b)$, then $E(\lambda) \in \operatorname{CE}(b)$. In particular, if
    $\mathfrak{L}\in \PP(\PP(C_T(\R^d)))$ is concentrated over $\operatorname{SPS}(b)$, then $\Lambda := E_\sharp \mathfrak{L} \in \PP(C_T(\PP(\R^d)))$ is concentrated over $\operatorname{CE}(b)$.
\end{prop}

\begin{proof}
    Let $\lambda \in \operatorname{SPS}(b)$ and $\mu_t:= (e_t)_\sharp \lambda$. Then for any $\xi \in C_c^1(0,T)$ and $\phi \in C_c^1(\R^d)$, it holds 
    \begin{align*}
        \int_0^T & \xi'(t)  \int_{\R^d} \phi(x) d\mu_t(x) dt 
        =
        \int_0^T \xi'(t) \int \phi(\gamma(t)) d\lambda(\gamma) dt
        = 
        - \int \int_0^T\xi(t) \nabla \phi(\gamma_t) \cdot \dot\gamma(t) dt d\lambda(\gamma)
        \\
        = & 
        - \int_0^T \xi(t) \int \nabla\phi(\gamma_t) \cdot b(t,\gamma_t,(e_t)_\sharp \lambda) d\lambda(\gamma) dt
        = 
        - \int_0^T \xi(t) \int_{\R^d} \nabla\phi(x) \cdot b(t,x,\mu_t) d\mu_t(x) dt.
    \end{align*}
\end{proof}

\begin{prop}\label{Lambda to M_t section 5}
    Let $\Lambda\in \PP(C_T(\PP(\R^d)))$ be concentrated on $\operatorname{CE}(b)$ and such that \[\int \int_0^T \int |b_t(x,\mu_t)|d\mu_t(x) dt d\Lambda(\boldsymbol{\mu})<+\infty.\] Then the curve of random measures defined as 
    $M_t:= (\mathfrak{e}_t)_\sharp \Lambda,$
    solves the continuity equation $\partial_tM_t + \operatorname{div}_\PP(b_t M_t) = 0$, in the sense of Definition \ref{def CERM non-local vf}.
\end{prop}

\begin{proof}
    Let $\xi\in C_c^1(0,T)$ and $F = \Psi (L_\Phi) \in \operatorname{Cyl}_c^1(\PP(\R^d))$, then 
    \begin{align*}
        \int_0^T \xi'(t) & \int_{\PP(\R^d)} F(\mu) dM_t(\mu) dt 
        = 
        \int_0^T \xi'(t) \int F(\mu_t) d\Lambda(\boldsymbol{\mu}) dt 
        \\
        = &
        \int \int_0^T \xi'(t) \psi\left( \int \phi_1 d\mu_t, \dots , \int \phi_k d\mu_t \right) dt d\Lambda(\boldsymbol{\mu})
        \\
        = &
        - \int \int_0^T \xi(t) \sum_{i=1}^k \partial_i \psi(L_\Phi(\mu_t)) \int_{\R^d} \nabla \phi_i(x)\cdot b_t(x,\mu_t) d\mu_t(x) dt d\Lambda(\boldsymbol{\mu})
        \\
        = & 
        - \int_0^T \xi(t) \int \int_{\R^d} \nabla_W F(x,\mu_t)\cdot b_t(x,\mu_t) d\mu_t(x) d\Lambda(\boldsymbol{\mu}) dt 
        \\
        = & 
        - \int_0^T \xi(t) \int_{\PP(\R^d)} \int_{\R^d} \nabla_W F(x,\mu)\cdot b_t(x,\mu) d\mu(x) dM_t(\mu) dt.
    \end{align*}
\end{proof}

\begin{co}
    Let $\mathfrak{L}\in \PP(\PP(AC_T(\R^d)))$ be concentrated on $\operatorname{SPS}(b)$ and such that 
    \[\int \int_0^T \int |b_t(\gamma_t,(e_t)_\sharp \lambda)| d\lambda(\gamma) dt d\mathfrak{L}<+\infty.\] Then the curve of random measures defined as 
    $
        M_t:= (E_t)_\sharp  \mathfrak{L}, 
    $
    solves $\partial_tM_t + \operatorname{div}_\PP(b_t M_t) = 0$.
\end{co}

\begin{proof}
    Simply notice that $E_t = \mathfrak{e}_t \circ E$, and we conclude thanks to Propositions \ref{L to Lambda section 5} and \ref{Lambda to M_t section 5}.
\end{proof}

\subsection{Superposition: from \texorpdfstring{$\boldsymbol{M}$}{} to \texorpdfstring{$\Lambda$}{}}

Here, similarly to Section \ref{metric structure}, we would like to somehow invert the result of Proposition \ref{Lambda to M_t section 5}, i.e. to define a measure $\Lambda \in \PP(C_T(\PP(\R^d)))$ concentrated over $\operatorname{CE}(b)$ starting from a curve $\boldsymbol{M} \in C_T(\PP(\PP(\R^d)))$ that solves the continuity equation $\partial_tM_t + \operatorname{div}_\PP(b_tM_t)=0$. This is a simple consequence of the superposition theorem \ref{superposition principle} for derivations.

\begin{teorema}\label{superposition principle for non-local vf}
    Let $\boldsymbol{M}=(M_t)_{t\in [0,T]} \in C_T(\PP(\PP(\R^d)))$ and $b:[0,T]\times \R^d \times \PP(\R^d) \to \R^d$ be an $L^1(\widetilde{M}_t \otimes dt)$-non-local vector fields, according to Definition \ref{vf inducing derivation}. 
    \\
    Assume that the continuity equation $\partial_tM_t + \operatorname{div}_\PP(b_tM_t) = 0$ is satisfied. 
    Then there exists a measure $\Lambda \in \PP(C_T(\PP(\R^d)))$ that is concentrated over $\operatorname{CE}(b)$ and $(\mathfrak{e}_t)_\sharp \Lambda = M_t$ for all $t\in [0,T]$.
\end{teorema}

\begin{proof}
    Consider the family of $L^1(M_t\otimes dt)$-derivations induced by the family of non-local vector fields $b$, i.e. 
\begin{equation}
    B_t[F](\mu) := \int_{\R^d} b(t,x,\mu) \cdot \nabla_W F(x,\mu) d\mu(x) \ \quad \forall F\in \operatorname{Cyl}_c^1(\PP(\R^d)),
\end{equation}
that is well-defined $M_t$-a.e. and for almost all $t\in (0,T)$, since 
\begin{equation}
    \int_0^T \int_{\PP(\R^d)}\int_{\R^d}|b(t,x,\mu)| d\mu(x) dM_t(\mu) dt<+\infty.
\end{equation}
It is an $L^1(M_t\otimes dt)$ family of derivations, so we can apply Theorem \ref{superposition principle} to obtain $\Lambda \in \PP(C_T(\PP(\R^d)))$ such that $(\mathfrak{e}_t)_\sharp \Lambda = M_t$ for all $t\in [0,T]$ and $\Lambda$-a.e. $\boldsymbol{\mu}= (\mu_t)_{t\in [0,T]}$ satisfies
\[\frac{d}{dt}\int_{\R^d}\phi(x) d\mu_t(x) = B_t[L_\phi](\mu_t) = \int_{\R^d} b(t,x,\mu_t)\cdot \nabla\phi(x) d\mu_t(x), \]
for all $\phi\in C_c^1(\R^d)$, in the sense of distributions in $(0,T)$.
\end{proof}

\subsection{Nested superposition: from \texorpdfstring{$\Lambda$}{} to \texorpdfstring{$\mathfrak{L}$}{}}\label{subsec: from mathfrak L to Lambda} 
Here, we want to invert the result of Proposition \ref{L to Lambda section 5}, i.e. to build a measure $\mathfrak{L}\in \PP(\PP(C_T(\R^d)))$ concentrated over $\operatorname{SPS}(b)$ to a given measure $\Lambda \in \PP(C_T(\PP(\R^d)))$ concentrated over $\operatorname{CE}(b)$.
Let us introduce some notations: we will denote 
\begin{equation}\label{eq: def of Y and Z}
    Y:= [0,T]\times \R^d \times \PP(\R^d), \quad \quad Z:= [0,T]\times C_T(\R^d) \times \PP(C_T(\R^d)).
\end{equation}

Then, we define the maps
\begin{equation}\label{eq: def of kappa and mathfrak K}
    \begin{aligned}
        &\kappa: C_T(\PP(\R^d)) \to \mathcal{M}_+(Y), \quad  \kappa(\boldsymbol{\mu}) = dt\otimes \big( \mu_t \otimes\delta_{\mu_t} \big),
        \\
        & \mathfrak{K}:\PP(C_T(\R^d)) \to \mathcal{M}_+(Z), \quad \mathfrak{K}(\lambda):= \mathcal{L}^1_T\otimes \lambda\otimes\delta_{\lambda}.
    \end{aligned}
\end{equation}
For a fixed Borel non-local vector field $b:[0,T]\times \R^d \to \PP(\R^d) \to \R^d$, we define the subset $\hat{\CE}(b)\subset \mathcal{M}_+(Y)\times \mathcal{M}(Y;\R^d)$ as
\begin{equation}\label{eq: def of hat ce(b)}
    \begin{aligned}
    & \hat{\operatorname{CE}}(b):= \big\{(\hat{\mu},\hat{\nu})\in \mathcal{M}_+(Y)\times\mathcal{M}(Y;\R^d) \ : \ \hat{\mu} \in \operatorname{Im}(\kappa), \ b\in L^1(\hat{\mu}), \\
    & \hspace{6.68cm} \hat{\nu} = b\hat{\mu},\ \partial_t\hat{\mu}+ \operatorname{div}(\hat{\nu})= 0 \big\},
    \end{aligned}
\end{equation}
where the equation $\partial_t \hat{\mu} + \operatorname{div}(\hat{\nu}) = 0$ must be understood as 
\[\forall \xi \in C_c((0,T)\times\R^d) \ \int \partial_t\xi(t,x) d\hat{\mu}(t,x,\mu) + \int \nabla_x \xi(t,x)\cdot d\hat{\nu}(t,x,\mu) = 0,\]
and, lastly, define the set $\hat{\operatorname{SPS}}(b)\subset\mathcal{M}_+(Z)$ as
\begin{equation}\label{eq: def of hat sps(b)}
\begin{aligned}
    & \hat{\operatorname{SPS}}(b):= \bigg\{\hat{\lambda} \in \mathcal{M}_+(Z) \ : \ \hat{\lambda} \in \operatorname{Im}(\mathfrak{K}), \  \hat{\lambda}\big(([0,T]\times AC_T(\R^d) \times \PP(\R^d))^c\big)=0, \\
        & \hspace{4.15cm} \hat{D}\in L^1(\hat{\lambda}), \ \int |\hat{D}-b\circ \hat{E}| d\hat{\lambda}=0\bigg\},
    \end{aligned}
\end{equation}
where $\hat{D}(t,\gamma,\lambda) = D(t,\gamma)$, according to Lemma \ref{meas of D(t,gamma)}, for all $(t,\gamma,\lambda) \in Z$, and 
\begin{equation}\label{eq: def of hat E}
    \hat{E}: Z \to [0,T]\times \R^d \times \PP(\R^d), \quad \hat{E}(t,\gamma,\lambda) = (t,\gamma(t),(e_t)_\sharp \lambda).
\end{equation}
Notice that $\hat{E}$ coincide with the map $\mathcal{E}$ already defined in \eqref{eq: def of mathcal E}, we just call it $\hat{E}$ here for the sake of notation.

The proof of the nested superposition principle relies on the measurability of the sets $\CE(b)$, $\operatorname{SPS}(b)$, $\hat{\CE}(b)$ and $\hat{\operatorname{SPS}}(b)$, for which the following lemma will be fundamental.

\begin{lemma}\label{cont of kappa and mathfrak k}
    The maps $\kappa$ and $\mathfrak{K}$ are continuous and injective. In particular, $\operatorname{Im}(\kappa)\subset \mathcal{M}_+(Y)$ and $\operatorname{Im}(\mathfrak{K})\subset \mathcal{M}_+(Z)$ are Borel subsets.
\end{lemma}

\begin{proof}
    The two functions are clearly continuous, and the injectivity follows by uniqueness of the disintegration with respect to the projection on the time variable. The last part of the statement follows from Lemma \ref{lusin iff borel} and the definition of Lusin set.
\end{proof}

\begin{prop}\label{meas of CE_b}
    Let $b:[0,T]\times \R^d \times \PP(\R^d) \to \R^d$ be Borel. 
    It holds 
    \begin{equation}\label{eq: CE(b) = kappa^-1 hat CE(b)}
        \operatorname{CE}(b) = \kappa^{-1}(\pi^1(\hat{\CE}(b))),
    \end{equation} 
    where $\pi^1(\hat{\mu},\hat{\nu}) = \hat{\mu}$ is the projection on the first coordinate.
    Moreover, the set $\hat{\CE}(b) \subset \mathcal{M}_+(Y)\times \mathcal{M}(Y;\R^d)$ is Borel and, in particular, $\operatorname{CE}(b)\subset C_T(\PP(\R^d))$ is Borel. 
\end{prop}

\begin{proof}
    Thanks to Lemma \ref{measurability int g d mu}, Lemma \ref{measurability nu=f mu} and Lemma \ref{cont of kappa and mathfrak k}, the sets 
    \[B_1:= \{\hat{\mu}\in \mathcal{M}_+(Y) \ : \ \hat{\mu}\in \operatorname{Im}(\kappa), \ b \in L^1(\hat{\mu})\}\]\[ B := \{(\hat{\mu},\hat{\nu})\in \mathcal{M}_+(Y)\times\mathcal{M}(Y;\R^d) \ : \ \hat{\mu}\in \operatorname{Im}(\kappa), \ b\in L^1(\hat{\mu}), \ \hat{\nu} = b\hat{\mu}\}\]
    are Borel, respectively, in $\mathcal{M}_+(Y)$ and $\mathcal{M}_+(Y)\times \mathcal{M}(Y;\R^d)$. Moreover, notice that $B_1 = \pi^1(B)$. Then, the set  
    \[\hat{\operatorname{CE}}(b):= \{(\hat{\mu},\hat{\nu})\in \mathcal{M}_+(Y)\times\mathcal{M}(Y;\R^d) \ : \ b\in L^1(\hat{\mu}), \ \hat{\nu} = b\hat{\mu},\ \partial_t\hat{\mu}+ \operatorname{div}(\hat{\nu})=0\}\]
    is Borel, since it is a relatively closed subset of $B$.
    Notice that the projection on the first coordinate restricted to $B$, i.e. $\pi^1|_{B} : B \to \mathcal{M}_+(Y)$, is continuous and injective, then it satisfies $\pi^1(B')\in \mathcal{B}\big(\mathcal{M}_+(Y)\big)$ for all $B'\subset B$ Borel set (see Lemma \ref{lusin iff borel} and Corollary \ref{cont and inj image of Borel sets}).
    Thus, we conclude proving \eqref{eq: CE(b) = kappa^-1 hat CE(b)}: if $\boldsymbol{\mu} \in \operatorname{CE}(b)$, then $(\kappa(\boldsymbol{\mu}), b \kappa(\boldsymbol{\mu})) \in \hat{\operatorname{CE}}(b)$; on the other hand, if $\boldsymbol{\mu}\in \kappa^{-1}(\pi^1(\hat{\operatorname{CE}}(b)))$, then $(\hat{\mu},\hat{\nu}) = \big(\kappa(\boldsymbol{\mu}),b \kappa(\boldsymbol{\mu})\big)\in \hat{\operatorname{CE}}(b)$, which implies that 
    \[\int_0^T \int |b_t(x,\mu_t)|d\mu_t(x) dt = \int_Y |b_t(x,\mu)|d\kappa(\boldsymbol{\mu})(t,x,\mu) <+\infty,\]
    and
    \begin{align*}
        0 = & \int_Y \frac{\partial}{\partial t}\xi(t,x) d\hat{\mu}(t,x,\mu) + \int_Y \nabla_x\xi(t,x)\cdot d\hat{\nu}(t,x,\mu)
        \\
        = & 
        \int_Y \frac{\partial}{\partial t}\xi(t,x) d\hat{\mu}(t,x,\mu) + \int_Y \nabla_x\xi(t,x)\cdot b(t,x,\mu)d\hat{\mu}(t,x,\mu)
        \\
        = & 
        \int_0^T \int_{\R^d} \frac{\partial}{\partial t} \xi(t,x) d\mu_t(x) dt + \nabla_x\xi(t,x)\cdot b(t,x,\mu_t)d\mu_t(x) dt,
    \end{align*}
    which means that $\boldsymbol{\mu} = (\mu_t)_{t\in [0,T]}$ satisfies the continuity equation, and so $\boldsymbol{\mu}\in \operatorname{CE}(b)$.
\end{proof}

\begin{prop}\label{meas of SPSb}
    Let $b:[0,T]\times \R^d \times \PP(\R^d) \to \R^d$ be Borel. It holds
    \begin{equation}\label{eq: sps(b) = k^-1 hat sps(b)}
        \operatorname{SPS}(b) = \mathfrak{K}^{-1}(\hat{\operatorname{SPS}}(b)).
    \end{equation}
    Moreover, the set $\hat{\operatorname{SPS}}(b)\subset \mathcal{M}_+(Z)$ is Borel and, in particular, $\operatorname{SPS}(b) \subset \PP(C_T(\R^d))$ is Borel.
\end{prop}

\begin{proof}
    The set $\hat{\operatorname{SPS}}(b)$ is Borel thanks to Lemma \ref{cont of kappa and mathfrak k}, Lemma \ref{meas of AC^p}, Lemma \ref{meas of D(t,gamma)} and Lemma \ref{measurability int g d mu}, and we conclude by proving \eqref{eq: sps(b) = k^-1 hat sps(b)}. If $\lambda\in \operatorname{SPS}(b)$, then calling $\hat{\lambda}= \mathfrak{K}(\lambda)$ we have:
    \begin{itemize}
        \item since $\lambda$ is concentrated over $AC_T(\R^d)$, it holds $\hat{\lambda}\big(([0,T]\times AC_T(\R^d)\times \PP(\R^d))^c\big)=0$;
        \item $\int |D-b\circ \hat{E}| d\hat{\lambda}=0$ is equivalent to $\dot{\gamma}(t) = b(t,\gamma(t),(e_t)_\sharp \lambda)$ $\mathcal{L}^1_T\otimes \lambda$-a.e.;
        \item thanks to the previous point, $\hat{D}\in L^1(\hat{\lambda})$ is equivalent to $\int \int |b(t,\gamma(t),(e_t)_\sharp \lambda)|  d\lambda dt <+\infty$.
    \end{itemize}
    So, $\mathfrak{K}(\lambda)\in \hat{\operatorname{SPS}}(b)$.
    On the other hand, if $\lambda \in \mathfrak{K}^{-1}(\hat{\operatorname{SPS}}(b))$, then $\hat{\lambda}:= \mathfrak{K}(\lambda) \in \hat{\operatorname{SPS}}(b)$, which means that 
    \begin{itemize}
        \item $\lambda\big((AC_T(\R^d))^c\big) = 0$;
        \item $\dot\gamma(t) = b(t,\gamma(t),(e_t)_\sharp \eta)$ for $\hat{\lambda}$-a.e. $(t,\gamma,\eta)$. Since $\hat{\lambda} = \mathcal{L}^1_T\otimes\lambda\otimes\delta_{\lambda}$, we have that $\dot\gamma(t) = b(t,\gamma(t),(e_t)_\sharp \lambda)$ $\mathcal{L}^1_T\otimes\lambda$-a.e.;
        \item thanks to the previous point, 
        \[\int\int |b(t,\gamma(t),(e_t)_\sharp \lambda)| d\lambda dt = \int | \hat{D}| d\hat{\lambda}<+\infty.\]
    \end{itemize}
    Thus, $\lambda\in \operatorname{SPS}(b)$ and the proof is concluded.
\end{proof}

Now, everything is set to prove the following theorem.

\begin{teorema}\label{from Lambda to L theorem/corollary}
    Let $b:[0,T]\times \R^d \times \PP(\R^d) \to \R^d$ be Borel measurable. Then, there exists a Souslin-Borel measurable map $G_b:\operatorname{CE}(b) \to \PP(C_T(\R^d))$ satisfying $\operatorname{Im} G_b \subset \operatorname{SPS}(b)$ and $E \circ G_b(\boldsymbol{\mu}) = \boldsymbol{\mu}$ for all $\boldsymbol{\mu}\in \operatorname{CE}(b)$. 
    \\
    In particular, if $\Lambda\in \PP(C_T(\PP(\R^d)))$ is concentrated over $\CE(b)$, then the measure $\mathfrak{L}:= (G_b)_\sharp \Lambda \in \PP(\PP(C_T(\R^d)))$ is well defined and concentrated over $\operatorname{SPS}(b)$, satisfying $E_\sharp  \mathfrak{L} = \Lambda$. 
\end{teorema}

\begin{proof}
    Notice that, thanks to Proposition \ref{L to Lambda section 5}, the map 
\[E|_{\operatorname{SPS}(b)} : \operatorname{SPS}(b) \to \operatorname{CE}(b)\]
is well defined. Moreover, thanks to the finite dimensional superposition principle, Theorem \ref{finite dim superposition principle}, it is surjective. Indeed for all $\boldsymbol{\mu}\in \operatorname{CE}(b)$ there exists a lifting $\lambda\in \PP(C_T(\R^d))$ such that $(e_t)_\sharp \lambda = \mu_t$, $\lambda$ is supported over a.c. curves that solves $\dot\gamma(t) = b(t,\gamma_t,\mu_t)$ and 
\[\int \int |b_t(\gamma(t),(e_t)_\sharp \lambda)| d\lambda(\gamma) dt = \int \int |b_t(x,\mu_t)|d\mu_t(x) dt <+\infty.\]
In other words, $\lambda \in \operatorname{SPS}(b)$. Then, thanks to Proposition \ref{meas of CE_b} and Proposition \ref{meas of SPSb}, we can apply Theorem \ref{measurable selection} to obtain a Souslin-Borel measurable map $G_b: \operatorname{CE}(b) \to \operatorname{SPS}(b)$, thanks to which we can define $\mathfrak{L}:= (G_b)_\sharp  \Lambda$, because of Corollary \ref{push-forward by a souslin-borel map}. Moreover, being $G_b$ a right-inverse of $E$, we also have the equality $\Lambda = E_\sharp  \mathfrak{L}$, so it preserves the hierarchy shown in Proposition \ref{L to Lambda section 5}.
\end{proof}

\subsection{Universality in the measurable selection}\label{subsec: universal}

In this subsection, we show a possible universal decomposition for a measurable selection map $G_b$, highlighting the dependence on the non-local vector field $b$ in it. 
\\
We will use the same notations of the previous subsection introduced in \eqref{eq: def of Y and Z}-\eqref{eq: def of hat E}, but here it will be fundamental to introduce also the following sets:
\begin{equation}
    \begin{aligned}
        \hat{\operatorname{CE}}:= \big\{(\hat{\mu},\hat{\nu}) \in \mathcal{M}_+(Y)\times \mathcal{M}(Y;\R^d) \ : \ \hat{\mu} \in \operatorname{Im}(\kappa), \quad \partial_t\hat{\mu} + \operatorname{div}\hat{\nu} = 0, \quad \hat{\nu} \ll \hat{\mu} & \big\},
        \\
        \hat{\operatorname{SPS}} :=\Big\{ \hat{\lambda}\in \mathcal{M}_+(Z) \ : \ \hat{\lambda} \in \operatorname{Im}(\mathfrak{K}), \quad 
        \hat{\lambda} \big(([0,T] \times AC_T(\R^d) \times \PP(C_T(\R^d)))^c \big) =0,& 
        \\
        \hspace{3.8cm} \hat{D}\in L^1(\hat{\lambda}), \quad 
        \exists v: Y \to \R^d \ \text{s.t.} \ \hat{D} = v\circ \hat{E} \ \hat{\lambda}\text{-a.e.} & 
        \Big\}.
    \end{aligned}
\end{equation}

These sets are strictly related to the sets $\hat{\operatorname{CE}}(b)$ and $\hat{\operatorname{SPS}}(b)$ (see \eqref{eq: def of hat ce(b)} and \eqref{eq: def of hat sps(b)}), the difference is that they allow a general Borel field $b$ instead of an a priori fixed one. 
\\
Indeed, for any element $(\hat{\mu},\hat{\nu}) \in \hat{\operatorname{CE}}$ there exists a Borel non-local vector field $v:[0,T]\times \R^d\times \PP(\R^d) \to \R^d$ such that $\hat\nu = v \hat\mu$. Moreover, since $\hat\mu\in \operatorname{Im}(\kappa)$, there exists a curve $\boldsymbol{\mu} \in C_T(\PP(\R^d))$ such that $\hat{\mu} = \kappa(\boldsymbol{\mu})$, and by the first condition, it also holds that $\partial_t \mu_t + \operatorname{div}(v_t(\cdot,\mu_t)\mu_t) = 0$. We remark that the presence of $\hat{\nu}$ is also telling which vector field (defined $\hat{\mu}$-a.e.) drives the continuity equation. In particular, $\hat{\operatorname{CE}}$ is formally the union of all $\hat{\operatorname{CE}}(b)$ as $b$ varies.
\\
Regarding the other set, any $\hat{\lambda}\in \hat{\operatorname{SPS}}$ can be represented as $\hat{\lambda} = \mathcal{L}^1\otimes \lambda \otimes \delta_\lambda$ for some $\lambda \in \PP(C_T(\R^d))$, and by the other conditions it holds that $\lambda$ is concentrated over absolutely continuous curves that solve the ordinary differential equation $\dot\gamma(t) = v_t(\gamma(t),(e_t)_\sharp \lambda)$. As before, $\hat{\operatorname{SPS}}$ is formally the a union of all $\hat{\operatorname{SPS}}(b)$ as $b$ varies.

 At this point, for all $b:[0,T]\times \R^d \times \PP(\R^d) \to \R^d$ Borel measurable we have the following commutative diagram:
 \begin{equation}\label{eq: comm diagram}
\begin{tikzcd}
  \operatorname{SPS}(b) \arrow[r, "\mathfrak{K}"] \arrow[d, "E"'] &  \hat{\operatorname{SPS}} \arrow[d, "\hat{\mathcal{E}}"]\\
  \CE(b) \arrow[r, "\kappa_b"] & \hat{\CE}
\end{tikzcd}
\end{equation}
where 
\[\hat{\mathcal{E}}(\hat{\lambda}) := \big(\hat{E}_\sharp  \hat{\lambda}, \hat{E}_\sharp (\hat{D}\hat{\lambda})\big), \quad \kappa_b(\hat{\mu}):= \big(\kappa(\boldsymbol{\mu}), b\kappa(\boldsymbol{\mu})\big).\]

In the following, we show that we can obtain a right inverse $\hat{\mathcal{G}}:\hat{\CE} \to \hat{\operatorname{SPS}}$ of the map $\hat{\mathcal{E}}$, performing again a measurable selection. Then, we infer that $\hat{\mathcal{G}}\circ \kappa_b(\operatorname{CE}(b)) \subseteq \mathfrak{K}(\operatorname{SPS}(b))$ for any measurable $b$, giving the diagram
\begin{equation}\label{eq: resulting diagram}
\begin{tikzcd}
  \operatorname{SPS}(b)  &  \hat{\operatorname{SPS}} \arrow[l, "\hspace{0.15cm}\mathfrak{K}^{-1}"']\\
  \CE(b) \arrow[r, "\kappa_b"] & \hat{\CE} \arrow[u, "\hat{\mathcal{G}}"']
\end{tikzcd}
\end{equation}
Then, composition $\mathfrak{K}^{-1}\circ \hat{\mathcal{G}}\circ\kappa_b$ will be a right inverse of $E|_{\operatorname{SPS}_b}$.
This strategy is an alternative proof for Theorem \ref{from Lambda to L theorem/corollary}, providing a more refined right inverse of the map $E|_{\operatorname{SPS}_b}$. Moreover, the measurable selection map $\hat{\mathcal{G}}$ is universal, in the sense that it does not depend on $b$.

In the following, we make rigorous this strategy step by step, starting with a more general result that will be useful but could be of independent interest.

\begin{lemma}\label{f = v circ e}
    Let $R,S$ be Polish spaces, $e:S\to R$ Borel measurable and $f:S\to \R^d$ a Borel function. 
    Let $\lambda\in \mathcal{M}_+(S) $ be such that $f\in L^1(\lambda)$. Assume that 
    \[\mathcal{H}\left( e_\sharp (f\lambda) | e_\sharp \lambda \right) = \mathcal{H}(f\lambda|\lambda),\]
    where 
    \[\mathcal{H}(\nu|\mu) := \int \sqrt{1+\left|\frac{d\nu}{d\mu}\right|^2} d\mu + |\nu^\perp|, \quad \nu = \frac{d\nu}{d\mu} \mu + \nu^{\perp} \text{ with }\nu^\perp \perp \mu.\]
    Then, calling $v:R \to \R^d$ a version for the density of $e_\sharp (f\lambda)$ w.r.t. $e_\sharp  \lambda$, we have that 
    \[f = v\circ e \ \lambda\text{-a.e.}\]
\end{lemma}

\begin{proof}
    First of all, notice that 
    \[\mathcal{H}\left( e_\sharp (f\lambda) | e_\sharp \lambda \right) = \int_R \sqrt{1+|v|^2} d(e_\sharp \lambda), \quad \mathcal{H}\left( f\lambda | \lambda \right) = \int_S \sqrt{1+|f|^2} d\lambda.\]

    Consider the disintegration of $\lambda$ with respect to the map $e$, i.e. the family of probability measures $\{\lambda_x\}_{x\in R}\subset \mathcal{P}(S)$ such that 
    \[\lambda = \int_R \lambda_x d(e_\sharp \lambda)(x), \quad \lambda_x(e^{-1}(\{x\}) = 1.\]
    Define the map 
    \[\tilde{v}(x) := \int_{S} f(y) d\lambda_x(y),\]
    which is well-define $e_\sharp \lambda$-a.e. because $f\in L^1(\lambda)$ and Borel measurable thanks to Corollary \ref{equiv of sigma alg cor}

    \textit{Claim}: $v=\tilde{v}$ $e_\sharp \lambda$-a.e. 
    \\
    To prove the claim, it suffices to prove that $\tilde{v}$ is a density for $e_\sharp (f\lambda)$ w.r.t. $e_\sharp \lambda$. For any $g:R\to \R^d$ it holds
    \begin{align*}
        \int_R g(x) \cdot de_{\sharp }(f\lambda)(x) = & \int_S g(e(y)) \cdot f(y) d\lambda(y) 
        = 
        \int_R \left( \int_S g(e(y))\cdot f(y) d\lambda_x(y) \right) d(e_\sharp \lambda)(x)
        \\
        = & 
        \int_R g(x) \cdot \left( 
        \int_S f(y) d\lambda_x(y) \right)d(e_\sharp \lambda) 
        = 
        \int_R g(x) \tilde{v}(x) d(e_\sharp \lambda)(x).
    \end{align*}
    
    Now, for simplicity, we call $H(z) = \sqrt{1+|z|^2}$. Then it always holds
    \begin{align*}
        \mathcal{H}\left( e_\sharp (f\lambda) | e_\sharp \lambda \right) & =  \int_R H(\tilde{v}(x)) d(e_\sharp \lambda)(x) = \int_X H\left( \int_S f(y) d\lambda_x(y) \right) d(e_\sharp \lambda)(x)
        \\
        \leq & \int_R \int_S H(f(y)) d\lambda_x(y) d(e_\sharp \lambda)(x) = \int_S H(f(y)) d\lambda(y) = \mathcal{H}\left(f\lambda |\lambda \right),
    \end{align*}
    where we used Jensen inequality. By hypothesis, we have equality, and by strict convexity of $G$ the equality can hold if and only if for $(e_\sharp \lambda)$-a.e. $x\in R$, $f$ is constant $\lambda_x$-a.e. In particular, $f$ is constant $\lambda_{e(y)}$-a.e., for $\lambda$-a.e. $y\in S$, and looking in the definition of $\tilde{v}$, we conclude that $\tilde{v}(e(y)) = f(y)$ $\lambda$-a.e.
\end{proof}

\begin{oss}\label{remark density push-forward}
    It is known that the same conclusion of the previous lemma is obtained if 
    \begin{equation}\label{eq density}
        \int_R |v|^p d(e_\sharp \lambda) = \int_S |f|^p d\lambda,
    \end{equation}
    for $p\in(1,+\infty)$. Unfortunately, this argument fails when $p=1$, as the following counterexample shows.
    Indeed, under the same assumptions of the lemma, but assuming only (\ref{eq density}) with $p=1$, it is still true that $e_\sharp \eta\ll e_\sharp \lambda$. Anyway, it is not said that $|f|=|v\circ e|$: take $S=\R^2$, $R=\R$, $\lambda = \frac{1}{2}\mathcal{H}^1_{[0,1]\times \{-1\}} + \frac{1}{2}\mathcal{H}^1_{[0,1]\times \{1\}}$, $e(x,y) = x$, $f(x,y) = 1/2$ is $x<0$ and $f(x,y) = 3/2$ if $x\geq 0$. Then 
    \[\eta := f\lambda = \frac{1}{4}\mathcal{H}^1_{[0,1]\times \{-1\}} + \frac{3}{4}\mathcal{H}^1_{[0,1]\times \{1\}}, \quad e_\sharp \lambda = e_\sharp \eta = \mathcal{L}^1_{[0,1]},\]
    which means that $v \equiv 1$. Moreover, it's easy to verify that
    \[\int |f|d\lambda = \frac{1}{2}\int_0^1 \frac{1}{2}dx + \frac{1}{2}\int_0^1\frac{3}{2} dx = 1 = |e_\sharp \eta|(\R),\]
    but it does not hold that $|f| = |v\circ e|$ for $\lambda$-a.e. $y\in S$.
    So, when $p=1$, we can only conclude that 
    \[f = \frac{|f|}{|v\circ e|} v\circ e \quad \lambda\text{-a.e.}\]
    In view of this, our result is a version of the result for $p>1$, when we only have $f\in L^1$. Of course, the function $H(z) = \sqrt{1+z^2}$ can be substituted with any strictly convex function $\tilde{H}$ satisfying $\operatorname{lim}_{z\to+\infty}\tilde{H}(z)/x <+\infty$.
\end{oss}

Thanks to this lemma, we can then rewrite the set $\hat{\operatorname{SPS}}$ as 
\begin{equation}
\begin{aligned}
    & \hat{\operatorname{SPS}} :=\Big\{ \hat{\lambda}\in \mathcal{M}_+(Z) \ : \ \hat{\lambda} \in \operatorname{Im}(\mathfrak{K}),\quad
        \hat{\lambda} \big(([0,T] \times AC_T(\R^d) \times \PP(C_T(\R^d)))^c \big) =0, 
        \\
        &
        \hspace{3.8cm}
        \mathcal{H}\left( \hat{E}_\sharp (\hat{D}\hat{\lambda}) | \hat{E}_\sharp  \hat{\lambda} \right) = \int\sqrt{1+|\hat{D}|^2}d\hat{\lambda} , \quad  \hat{D}\in L^1(\hat{\lambda})
        \Big\}.
\end{aligned}
\end{equation}
Indeed, thanks to Lemma \ref{f = v circ e}, the constraint that there exists $v:Y \to \R^d$ such that $\hat{D} = v\circ \hat{E}$ $\hat{\lambda}$-a.e. can be rewritten by 
\begin{equation}\label{eq: meas cond using H}
\mathcal{H}\left( \hat{E}_\sharp (\hat{D}\hat{\lambda}) | \hat{E}_\sharp  \hat{\lambda} \right) = \int\sqrt{1+|\hat{D}|^2}d\hat{\lambda}.
\end{equation}

We are now ready to prove the following.

\begin{teorema}\label{lemma: properties of hat mathcal E}
    The sets $\hat{\operatorname{CE}} \subset \mathcal{M}_+(Y)\times \mathcal{M}(Y;\R^d)$ and $\hat{\operatorname{SPS}}\subset \mathcal{M}_+(Z)$, and the map $\hat{\mathcal{E}}$ are Borel measurable. Moreover, 
    \begin{equation}
        \hat{\mathcal{E}} |_{\hat{\operatorname{SPS}}}: \hat{\operatorname{SPS}} \to \hat{\operatorname{CE}}
    \end{equation}
    is surjective. In particular, there exists a Souslin-Borel measurable map $\hat{\mathcal{G}}:\hat{\operatorname{CE}} \to \hat{\operatorname{SPS}}$ such that $\hat{\mathcal{E}} \circ\hat{\mathcal{G}} = \operatorname{id}_{\hat{\operatorname{CE}}}$. 
    Moreover, for all $b:[0,T]\times \R^d \times \PP(\R^d) \to \R^d$ Borel, it holds $\hat{\mathcal{E}}(\hat{\operatorname{SPS}}(b) ) = \hat{\operatorname{CE}}(b) $ and $\hat{\mathcal{G}}(\hat{\CE}(b)) \subset \hat{\operatorname{SPS}}(b)$.
\end{teorema}

\begin{proof}
    \textbf{Step 1}: the set $\hat{\operatorname{CE}}$ is measurable, since $\partial_t\hat{\mu} + \operatorname{div}\hat{\nu} = 0$ is a closed condition, $\hat{\nu}\ll\hat{\mu}$ is a Borel condition (see Corollary \ref{abs cont meas R^d}) and $\operatorname{Im}(\kappa)$ is Borel measurable (see Lemma \ref{cont of kappa and mathfrak k}).
    \\
    The set $\hat{\operatorname{SPS}}$ is Borel since: the set $\operatorname{Im}(\mathfrak{K})$ is Borel thanks to Lemma \ref{cont of kappa and mathfrak k}; from Lemma \ref{measurability int g d mu}, we have that the evaluation on Borel sets is Borel and $\hat{D}\in L^1(\hat{\lambda})$ is a Borel condition; the map $\mathcal{M}_+(Z) \ni \hat{\lambda}\mapsto \hat{D}\hat{\lambda} \in \mathcal{M}(Z;\R^d)$ is Borel measurable (see Corollary \ref{measurability vector measure f mu}), as well as the push-forward operation through $\hat{E}$ (see Proposition \ref{meas of push forward}). Then the equality \eqref{eq: meas cond using H} is a Borel condition, again by Lemma \ref{measurability int g d mu} and the fact that $\mathcal{H}$ is jointly l.s.c. (see \cite[Chapter 2]{ambrosio2000functions}). 

    \textbf{Step 2}: as before, $\hat{\mathcal{E}}$ is the composition of Borel measurable, thanks to Corollary \ref{measurability vector measure f mu} and Proposition \ref{meas of push forward}, thus it is Borel measurable. \\
    Moreover, it maps $\hat{\operatorname{SPS}}$ in $\hat{\CE}$: let $\hat{\lambda}\in \hat{\operatorname{SPS}}$ and $(\hat{\mu},\hat{\nu}):= \hat{\mathcal{E}}(\hat{\lambda})$. Since $\hat{\lambda} \in \operatorname{Im}(\mathfrak{K})$, there exists $\lambda \in \PP(C_T(\R^d))$ such that $\hat{\lambda} = \mathfrak{L}^1_T\otimes \lambda \otimes \delta_\lambda$, which implies that $\hat{\mathcal{E}}(\hat{\lambda}) = \mathcal{L}^1_T\otimes (\mu_t\otimes\delta_{\mu_t})$, where $\mu_t := (e_t)_\sharp \lambda$ for all $t\in[0,T]$. Moreover, the equality \eqref{eq: meas cond using H} implies (see Lemma \ref{f = v circ e}) the existence of a vector field $b:[0,T]\times \R^d \times \PP(\R^d) \to \R^d$ such that $\hat{D} = b \circ \hat{E}$ for $\hat{\lambda}$-a.e. $(t,\gamma,\eta)$. By definition of $\hat{\mathcal{E}}$ and again by Lemma \ref{f = v circ e}, $\hat{\nu} = b\hat{\mu}$. Regarding the continuity equation, for all $\xi\in C_c^1((0,T)\times \R^d)$, it holds
    \begin{align*}
        \int \frac{\partial}{\partial t} \xi(t,x) d\hat{\mu} + \int \nabla_x \xi(t,x) d\hat{\nu} & = 
        \int \frac{\partial}{\partial t} \xi(t,\gamma(t)) d\hat{\lambda} + \int \nabla_x \xi(t,\gamma(t)) \cdot D(t,\gamma)d\hat{\lambda}
        \\
        & = 
        \int \int_0^T \frac{\partial}{\partial t}\xi(t,\gamma(t)) dt d\lambda(\gamma) = 0.
    \end{align*}
    Then, $\hat{\mathcal{E}} $ maps $\hat{\operatorname{SPS}}$ in $\hat{\operatorname{CE}}$. 

    \textbf{Step 3}: given $b:[0,T]\times \R^d \times \PP(\R^d) \to \R^d$ and $\hat{\lambda} \in \hat{\operatorname{SPS}}(b)$ of the form $\hat{\lambda} = \mathcal{L}^1_T\otimes \lambda \otimes \delta_\lambda$, for some $\lambda \in \PP(C_T(\R^d))$, using the same notation of the previous step we have 
    \begin{align*}
        0 = \int \frac{\partial}{\partial t} & \xi(t,x) d\hat{\mu} + \int \nabla_x \xi(t,x) d\hat{\nu} = 
        \int \frac{\partial}{\partial t} \xi(t,\gamma(t)) d\hat{\lambda} + \int \nabla_x \xi(t,\gamma(t)) \cdot D(t,\gamma)d\hat{\lambda}
        \\
        & = 
        \int \frac{\partial}{\partial t} \xi(t,\gamma(t)) d\hat{\lambda} + \int \nabla_x \xi(t,\gamma(t)) \cdot b(t,\gamma(t),(e_t)_\sharp \eta)d\hat{\lambda}
        \\
        & = 
        \int_0^T\int \frac{\partial}{\partial t} \xi(t,\gamma(t)) d\lambda dt + \int_0^T\int \nabla_x \xi(t,\gamma(t)) \cdot b(t,\gamma(t),(e_t)_\sharp \lambda)d\lambda dt
        \\
        & = 
        \int_0^T\int \frac{\partial}{\partial t} \xi(t,x) d\mu_t(x) dt + \int_0^T\int \nabla_x \xi(t,x) \cdot b(t,x,\mu_t)d\mu_tdt,
    \end{align*}
    i.e. $\hat{\mathcal{E}}(\hat{\operatorname{SPS}}(b)) \subseteq \hat{\operatorname{CE}}(b)$. The equality follows from the surjectivity, which we prove in the next step. 

    \textbf{Step 4}: we are left to prove the surjectivity of the map $\hat{\mathcal{E}}$. For any $(\hat{\mu},\hat{\nu}) \in \hat{\operatorname{CE}}$, we have that $\hat{\mu}$ is of the form $\mathcal{L}^1_T\otimes (\mu_t\otimes \delta_{\mu_t})$ for some $(\mu_t)_{t\in[0,T]}\in C_T(\PP(\R^d))$, and $\mu_t$ solves the continuity equation $\partial_t\mu_t + \operatorname{div}(b_t(\cdot,\mu_t)\mu_t) = 0$, where $b$ is a density for $\hat{\nu}$ w.r.t. $\hat{\mu}$. Then we can use the finite dimensional superposition principle (Theorem \ref{finite dim superposition principle}) to obtain $\lambda\in \PP(C_T(\R^d))$ that is in $\operatorname{SPS}(b)$. It is not hard to verify that $\hat{\lambda}:= \mathfrak{K}(\lambda) \in \hat{\operatorname{SPS}}$ and $\hat{\mathcal{E}}(\hat{\lambda}) = (\hat{\mu},\hat{\nu})$.

    \textbf{Step 5}: we can finally apply Theorem \ref{measurable selection} to obtain the existence of $\hat{\mathcal{G}}:\hat{\operatorname{CE}} \to \hat{\operatorname{SPS}}$ satisfying the requirements. Then, we are only left to prove that $\hat{\mathcal{G}}(\hat{\CE}(b)) \subset \hat{\operatorname{SPS}}(b)$, for which it is sufficient to show that $\hat{\mathcal{E}} |_{\hat{\operatorname{SPS}}}^{-1} (\hat{\operatorname{CE}}(b)) = \hat{\operatorname{SPS}}(b))$. Let $(\hat{\mu},\hat{\nu}) \in \hat{\operatorname{CE}}(b)$ with $\hat{\mu} = \mathcal{L}^1_T\otimes(\mu_t \otimes \delta_{\mu_t})$ for some $\boldsymbol{\mu}= (\mu_t)_{t\in[0,T]} \in C_T(\PP(\R^d))$ and consider $\hat{\lambda} \in \hat{\operatorname{SPS}}$ such that $\hat{\mathcal{E}}(\hat{\lambda}) = (\hat{\mu},\hat{\nu})$. Since $\hat{\operatorname{SPS}}\subset \operatorname{Im}(\mathfrak{K})$, there exists $\lambda \in \PP(C_T(\R^d))$ such that $\hat{\lambda} = \mathcal{L}^1_T\otimes \lambda \otimes \delta_\lambda$. Moreover, since 
    \[\mathcal{H}\left( \hat{E}_\sharp (\hat{D}\hat{\lambda}) | \hat{E}_\sharp  \hat{\lambda} \right) = \mathcal{H}\left( \hat{D}\hat{\lambda}|  \hat{\lambda} \right),\]
    we have that there exists a Borel measurable function $v:[0,T]\times \R^d \times \PP(\R^d) \to \R^d$ such that $\hat{D} = v\circ \hat{E}$ for $\hat{\lambda}$-a.e. $(t,\gamma,\eta)$. In particular, $\hat{\lambda} \in \hat{\operatorname{SPS}}(v)$, and thanks to Lemma \ref{lemma: properties of hat mathcal E} we have that $\hat{\mathcal{E}}(\hat{\lambda}) = (\hat{\mu},\hat{\nu}) \in \hat{\operatorname{CE}}(v)$. By assumption $(\hat{\mu},\hat{\nu}) \in \hat{\operatorname{CE}}(b)$, thus it holds $b = v$ for $\hat{\mu}$-a.e. $(t,x,\mu)$. Since $\hat{\mu} = \hat{E}_\sharp \hat{\lambda}$, we have that $b\circ \hat{E} = v\circ \hat{E}=D$ $\hat{\lambda}$-almost everywhere. In particular $\hat{\lambda} \in \hat{\operatorname{SPS}}(b)$.
\end{proof}

We are now ready to conclude the argument we presented above, stating the main result.
Let $b:[0,T]\times \R^d \times \PP(\R^d) \to \R^d$ be a Borel measurable map. Define $\mathcal{M}_{+,b} (Y) := \big\{\hat{\mu} \in \mathcal{M}_+(Y) \ : \ b \in L^1(\hat{\mu};\R^d) \big\}$ and the map 
\begin{equation}
\begin{aligned}
    V_b: \mathcal{M}_{+,b} (Y) & \to \mathcal{M}_{+}(Y) \times \mathcal{M}(Y;\R^d)
    \\
    \hat{\mu} & \mapsto (\hat{\mu}, b \hat{\mu}),
\end{aligned}
\end{equation}
which is a Borel map thanks to Corollary \ref{measurability vector measure f mu}. Then, the map $\kappa_b$ introduced in \eqref{eq: comm diagram}, is given by $\kappa_b = V_b \circ \kappa$, resulting Borel measurable. Recall that the (left) inverse of $\mathfrak{K}$ is given by $\mathfrak{K}^{-1} = \frac{1}{T}\pi^2_\sharp $.

\begin{co}
    The function $\mathfrak{K}^{-1} \circ \hat{\mathcal{G}} \circ \kappa_b$ maps $\operatorname{CE}(b)$ to $\operatorname{SPS}(b)$ and is a right inverse of the map  $E|_{\operatorname{CE}(b)}$, i.e. $E\circ \mathfrak{K}^{-1} \circ \hat{\mathcal{G}} \circ \kappa_b = \operatorname{id}_{\operatorname{CE}(b)}$. In particular, for any $\Lambda \in \PP(C_T(\PP(\R^d)))$ concentrated over $\operatorname{CE}(b)$, the random measure $\mathfrak{L}:= (\mathfrak{K}^{-1} \circ \hat{\mathcal{G}} \circ \kappa_b)_\sharp \Lambda$ is concentrated over $\operatorname{SPS}(b)$ and satisfies the properties of Theorem \ref{main theorem}.
\end{co}

\begin{proof}
For simplicity, let us call $G_b:= \mathfrak{K}^{-1} \circ \hat{\mathcal{G}} \circ \kappa_b$.
    First of all, notice that $\kappa_b = V_b \circ k$ is well defined and maps $\boldsymbol{\mu} \in\operatorname{CE}(b)$ in $(\hat{\mu},\hat{\nu}) \in \hat{\operatorname{CE}}(b) \cap \operatorname{Im}(\kappa)$. Then, $\hat{\mathcal{G}}$ selects a measure $\hat{\lambda} \in \hat{\operatorname{SPS}}$ such that $\hat{\mathcal{E}}(\hat{\lambda}) = (\hat{\mu},\hat{\nu})$. 
    
    Since $\hat{\mathcal{G}}(\operatorname{CE}(b)) \subset \hat{\operatorname{SPS}}(b)$, then $\hat{\lambda} :=\hat{\mathcal{G}}(\kappa_b(\boldsymbol{\mu})) \in \hat{\operatorname{SPS}}(b)$. We need to show that $\lambda := \mathfrak{K}^{-1}(\hat{\lambda}) \in \operatorname{SPS}(b)$. All the needed checks are pretty straightforward: 
    \begin{itemize}
        \item since $\lambda = \pi^2_\sharp \hat{\lambda}$, we have that $\lambda(AC_T(\R^d)) = 0$;
        \item for $\hat{\lambda}$-a.e. $(t,\gamma,\eta)$ it holds $D(t,\gamma) = b(t,\gamma(t),(e_t)_\sharp \eta)$, which immediately implies that $D(t,\gamma) = b(t,\gamma(t),(e_t)_\sharp \lambda)$ for $\mathcal{L}^1_T\otimes \lambda$-a.e. $(t,\gamma)$;
        \item $\int \int |b(t,\gamma_t,(e_t)_\sharp \lambda)| d\lambda(\gamma) dt = \int |D| d\hat{\lambda}<+\infty$.
    \end{itemize}
    Thus, we conclude that $G_b$ maps $\operatorname{CE}(b)$ in $\operatorname{SPS}(b)$. 
    We are left to show that $G_b$ is a right inverse of $E$. Let $\lambda = G_b(\boldsymbol{\mu})$, then for all $f:[0,T]\times\R^d \to [0,+\infty]$ Borel, it holds
    \begin{align*}
        \int f(t,\gamma(t)) d\lambda(\gamma) = & \int f(t,\gamma(t)) d\hat{\lambda}(t,\gamma,\eta)
        = 
        \int f(t,x) d\hat{\mu}(t,x,\mu) 
        = 
        \int f(t,x) d\mu_t(x) dt,
    \end{align*}
    where $\hat{\lambda} = \mathfrak{K}(\lambda)$ and $\hat{\mu} = \hat{E}_\sharp \hat{\lambda} = \mathcal{L}^1_T\otimes (\mu_t\otimes \delta_{\mu_t})$. In particular, $(e_t)_\sharp \lambda = \mu_t$ for all $t\in[0,T]$.
\end{proof}

\section{Existence and uniqueness: the Lipschitz case}\label{uniqueness section} In this section, we are going to prove a uniqueness result for the solution of the continuity equation \(\partial_tM_t + \operatorname{div}_\PP(b_tM_t)=0,\) under a Lipschitz assumption of the vector field $b$ with respect to the variables $(x,\mu) \in\R^d \times \PP(\R^d)$. 

Here we state first a useful lemma, and then the main uniqueness theorem of this section, whose proof is divided into several parts and postponed to the next subsection.

\begin{lemma}\label{initial lemma section 6}
    Let $p\geq 1$. Let $\boldsymbol{M} = (M_t)_{t\in [0,T]} \in C_T(\PP(\PP(\R^d)))$ and $b:[0,T]\times \R^d \times \PP(\R^d)\to \R^d$ an $L^p(\widetilde{M}_t\otimes dt)$-non-local vector field (see Definition \ref{vf inducing derivation}) satisfying $\partial_t M_t + \operatorname{div}_\PP(b_tM_t) = 0$. The following properties hold:
    \begin{itemize}
        \item[(i)] if $M_0$ is concentrated over $\PP_p(\R^d)$, then the same holds for $M_t$ for all $t\in [0,T]$;
        \item[(ii)] if $M_0\in \PP_p(\PP_p(\R^d))$, then $M_t\in \PP_p(\PP_p(\R^d))$ for all $t\in [0,T]$.
    \end{itemize}
\end{lemma}

\begin{proof}
    (i) Consider $\Lambda \in \PP(C_T(\PP(\R^d)))$ given by Theorem \ref{superposition principle}. Thanks to \eqref{time integrability for non-local vf} and the properties of $\Lambda$, we have that for $\Lambda$-a.e. $\boldsymbol{\mu} = (\mu_t)_{t\in [0,T]}\in C_T(\PP(\R^d))$ it holds
    \[\int_0^T \int_{\R^d} |b(t,x,\mu_t)|^p d\mu_t(x) dt<+\infty,\]
    and $\partial_t\mu_t + \operatorname{div} (b_t(\cdot,\mu_t)\mu_t )= 0$. Moreover, since $(\mathfrak{e}_0)_\sharp \Lambda = M_0$, we have that $\mu_0\in \PP_p(\R^d)$ for $M_0$-a.e. $\mu_0$. Then, for any $(\mu_t)_{t\in [0,T]}$ with these properties, consider $\lambda \in \PP(C_T(\R^d))$ given by Theorem \ref{finite dim superposition principle}, so that 
    \begin{align*}
        W_p^p(\mu_t,\mu_0) \leq & \int_0^t \int |\dot\gamma|^p(s)d\lambda(\gamma)ds 
        =  
        \int_0^t \int |b(s,\gamma_s,(e_s)_\sharp \lambda)|^p(s) d\lambda(\gamma)ds 
        \\
        = &
        \int_0^t \int_{\R^d} |b(s,x,\mu_s)|^p d\mu_s(x) ds <+\infty
    \end{align*}
    and given a transport plan $\pi_{0,t}$ realizing $W_p^p(\mu_0,\mu_t)$, we have    \begin{equation}\label{p-moment estimate}
        \int_{\R^d} |x|^p d\mu_t(x) \leq 2^p\int_{\R^d\times \R^d} |x-y|^p d\pi_{0,t}(x,y) + \int_{\R^d}|y^p| d\mu_0(y)<+\infty,
    \end{equation}
    which gives that $M_t$ is concentrated over $\PP_p(\R^d)$. Property (ii) follows from Proposition \ref{p-superposition}.
\end{proof}

\begin{teorema}\label{main uniqueness theorem}
    Let $p\geq 1$. Let $\boldsymbol{M} = (M_t)_{t\in [0,T]} \in C_T(\PP(\PP(\R^d)))$ and $b:[0,T]\times \R^d \times \PP(\R^d)\to \R^d$ an $L^p(M_t\otimes dt)$-non-local vector field satisfying:
    \begin{enumerate}
        \item\label{cond 2 weak uniqueness} for any $\mu_0,\mu_1\in \PP_p(\R^d)$ there exists a $W_p$-optimal plan $\pi$ between $\mu_0$ and $\mu_1$ such that \begin{equation}\label{eq: Lip condition}
            \int_{\R^d\times \R^d}|b(t,x_0,\mu_0) - b(t,x_1,\mu_1)|^p d\pi(x_0,x_1) \leq L(t)  W_p^p(\mu_0,\mu_1),
        \end{equation}
        with $L \in L^1(0,T)$;
        \item $\partial_tM_t + \operatorname{div}_\PP(b_t M_t) = 0$ and $M_0=\overline{M}\in \PP(\PP(\R^d))$ is concentrated over $\PP_p(\R^d)$.
    \end{enumerate}
    Then:
    \begin{enumerate}
        \item[(i)] there exists a unique $\Lambda \in \PP(C_T(\PP(\R^d)))$ concentrated on $\operatorname{CE}(b)$ and such that $(\mathfrak{e}_0)_\sharp \Lambda = \overline{M}$;
        \item[(ii)] there exists a unique $\mathfrak{L} \in \PP(\PP(C_T(\R^d)))$ concentrated on $\operatorname{SPS}(b)$ and such that $(E_0)_\sharp \mathfrak{L} = \overline{M}$.
    \end{enumerate}
    In particular, $\boldsymbol{M} \in C_T(\PP(\PP(\R^d)))$ is the unique solution of $\partial_t M_t + \operatorname{div}_\PP(b_tM_t) = 0$ satisfying $M_0 = \overline{M}$.
\end{teorema}

    It is worth commenting on the existence for solutions of the continuity equation for random measures: it can be recovered under Lipschitz assumptions in the spirit of our theorem (see \cite{cavagnari2022lagrangian, bonnet2021differential}) or even under just Carathéodory assumptions (see \cite{bonnet2024caratheodory}). Indeed, if given a non-local vector field $b:[0,T]\times \R^d \times \PP(\R^d) \to \R^d$, we can prove existence for $\partial_t \mu_t + \operatorname{div}(b_t(\cdot,\mu_t)\mu_t) = 0$ for any starting measure $\mu_0\in \PP(\R^d)$, then using a measurable selection argument and the disintegration theorem, for any $M_0\in \PP(\PP(\R^d))$ we can prove existence for $\Lambda \in \PP(C_T(\PP(\R^d)))$ concentrated over $\CE(b)$ and such that $(\mathfrak{e}_0)_\sharp \Lambda = M_0$. At this point, the superposition principle gives existence for: $(M_t)_{t\in [0,T]} \in C_T(\PP(\PP(\R^d)))$ solving $\partial_t M_t + \operatorname{div}_\PP(b_tM_t) = 0$ starting from $M_0$; and $\mathfrak{L}\in \PP(\PP(C_T(\R^d)))$ concentrated over $\operatorname{SPS}(b)$ and such that $(E_0)_\sharp \mathfrak{L} = M_0$.

\subsection{Proof of Theorem \ref{main uniqueness theorem}}
Even if point $(i)$ follows from point $(ii)$, we prove it first using a more classical argument. Consider $\Lambda\in \PP(C_T(\PP(\R^d)))$ given by Theorem \ref{main theorem}. We will prove uniqueness for the trajectories 
\begin{equation}\label{CE uniqueness section}
\partial_t\mu_t + \operatorname{div}(b_t(\cdot,\mu_t)\mu_t) = 0
\end{equation}
for any fixed starting point $\overline{\mu}\in \PP_p(\R^d)$. Indeed, given $\boldsymbol{\mu}^0, \boldsymbol{\mu}^1\in C_T(\PP(\R^d))$  two solutions of \eqref{CE uniqueness section}, thanks to \cite[Theorem 8.4.7 and Remark 8.4.8]{ambrosio2005gradient}, we can differentiate in time the quantity $W_p^p(\mu_t^0,\mu_t^1)$, to obtain that, given a $W_p$-optimal plan $\pi_t^{0,1} $ between $\mu_t^0$ and $\mu_t^1$
    \begin{align*}
        \frac{d}{dt} & W_p^p(
        \mu_t^0,\mu_t^1
        )  = p\int_{\R^d\times\R^d} |x_0-x_1|^{p-2}(x_0-x_1)\cdot \big(b(t,x_0,\mu_t^0) - b(t,x_1,\mu_t^1)\big) d\pi_t^{0,1}(x_0,x_1)
        \\
        & \leq
        pW_p^{p-1}(\mu_t^0,\mu_t^1)\left(\int |b(t,x_0,\mu_t^0) - b(t,x_1,\mu_t^1)|^p d\pi_t^{0,1}(x_0,x_1) \right)^{\frac{1}{p}}
        \leq
        pL(t)W_p^p(\mu_t^0,\mu_t^1).
    \end{align*}
    Using Gr\"onwall lemma, we conclude that for any $\overline{\mu}\in \PP_p(\R^d)$ there exists a unique $\boldsymbol{\mu}_{\overline{\mu}}\in C_T(\PP(\R^d))$ solution of \eqref{CE uniqueness section}. Thus, we have the following representation 
    \begin{equation}
        \Lambda = \int_{\PP_p(\R^d)} \delta_{\boldsymbol{\mu}_{\overline{\mu}}} d\overline{M}(\overline{\mu}),
    \end{equation}
    implying the uniqueness result of the theorem. 
    \\
    Using this result, we can already prove the uniqueness of $\partial_tM_t+\operatorname{div}_\PP(b_tM_t)=0$ given a starting point $\overline{M}\in \PP(\PP(\R^d))$ concentrated over $\PP_p(\R^d)$. Indeed, if we had two different solutions $\boldsymbol{M}^0$ and $\boldsymbol{M}^1$, Theorem \ref{main theorem} would give us two different $\Lambda^0,\Lambda^1 \in \PP(C_T(\PP(\R^d)))$ satisfying property $(i)$ of Theorem \ref{main uniqueness theorem}, which is a contradiction. 

    We are left with the proof of $(ii)$, that, similarly to the proof of $(i)$, passes again through the uniqueness of superposition solutions $\lambda\in\operatorname{SPS}(b)$ with a fixed starting point $(e_0)_\sharp \lambda = \overline{\mu}\in \PP_p(\R^d)$. In particular, we will see how the Lipschitz assumption in Theorem \ref{main uniqueness theorem} implies a Lipschitz assumption in the variable space $x\in \R^d$, giving uniqueness of trajectories at the particle level. Before proceeding, we need some preliminary results, that are the extension to the case $p\geq 1$ of (a part of) \cite[Lemma 6.1, Theorem 6.2 and Theorem 7.6]{cavagnari2023lagrangian}.

    \begin{lemma}
        Let $\mu_0,\mu_1\in \PP_p(\R^d)$ and $\pi \in \Gamma(\mu_0,\mu_1)$. Assume $\mu_0$ has finite support $S = \{\Bar{x}_1,\dots,\Bar{x}_N\}$ with $\delta:= \min\{|\Bar{x}_i - \Bar{x}_j| \ : \ i\neq j\}$ and 
        \[\sup\big\{ |y-x| \ : (x,y) \in \operatorname{supp}\pi \big\} \leq \frac{\delta}{2}.\]
        Then $\pi$ is $W_p$-optimal, i.e. $W_p^p(\mu_0,\mu_1) = \int |y-x|^pd\pi(x,y)$.
    \end{lemma}

    \begin{proof}
        It is sufficient to prove that the support of $\pi $ satisfies the $c$-cyclical monotonicity property, with $c(x,y) := |x-y|^p$. Consider $\{(x_i,y_i)\}_{i=1}^n \subset \operatorname{supp} 
        \pi$, with $x_0 := x_n$. Then
        \begin{align*}
            \sum_{i=1}^n |x_{i-1} - y_i|^p - |x_i-y_i|^p
            \geq \sum_{i=1}^n \left[\big||x_{i-1}-x_i| - |x_i-y_i|\big|^p - \left(\frac{\delta}{2}\right)^p\right]\geq 0,
        \end{align*}
        because $(x_i,y_i)\in \operatorname{supp}\pi$, thus $|x_i-y_i| \leq \delta/2$, and 
        \[|x_{i-1}- y_i| = |x_{i-1} -x_ i + x_i - y_i| \geq |x_{i-1} - x_i| - |x_i - y_i| \geq \delta - \delta/2 = \delta/2. \]
    \end{proof}

    \begin{lemma}\label{interpolating curve lemma}
        Let $\mu_0,\mu_1\in \PP_p(\R^d) $ be two measures with finite support, $\pi\in \Gamma(\mu_0,\mu_1)$ and $\mu_t := (\mathrm{x}^t)_\sharp \pi$, where $\mathrm{x}^t(x_0,x_1) := (1-t)x_0 + t x_1$. Then the following properties hold:
        \begin{itemize}
            \item[(i)] for every $s\in [0,1]$ there exists $\delta > 0$ such that for every $t\in [0,1]$ with $|t-s|\leq \delta$ $\pi^{st}:= (\mathrm{x}^s,\mathrm{x}^t)_\sharp \pi $ is a $W_p$ optimal plan between $\mu_s$ and $\mu_t$. Moreover 
            \begin{equation}
                W_p^p(\mu_s,\mu_t) = |t-s|^p \int |x_0 - x_1|^pd\pi(x_0,x_1);
            \end{equation}
            \item[(ii)] there exist $0=t_0 < \dots < t_K =1$ such that for every $k=1,\dots,K$, $\mu|_{[t_{k-1},t_k]}$ is a constant speed geodesic w.r.t. $W_p$ and 
            \begin{equation}
                W_p^p(\mu_{s},\mu_{r}) = |r - s|^p \int |x_0-x_1|^p d\pi(x_0,x_1) \quad \forall s,r\in[t_{k-1},t_k];
            \end{equation}
            \item[(iii)] the length of the curve $t\mapsto \mu_t$, w.r.t. $W_p$, is $\big(\int |x_0-x_1|^p d\pi(x_0,x_1)\big)^{1/p}$.  \end{itemize}
    \end{lemma}

    \begin{proof}
        It is the very same of \cite[Theorem 6.2]{cavagnari2023lagrangian}.
    \end{proof}

    \begin{lemma}\label{lipschitz along optimal plans implies lipschitz on any plan}
        Let $b:\R^d \times \PP_p(\R^d) \to \R^d$ be such that for all $\mu_0,\mu_1\in \PP_p(\R^d)$ and some $W_p$-optimal plan $\pi\in\Gamma(\mu_0,\mu_1)$ it holds
        \begin{equation}\label{lipschitz condition lemma}
            \int |b(x_0,\mu_0) - b(x_1,\mu_1)|^p d\pi(x_0,x_1) \leq L\int | x_1 - x_0 |^p d\pi(x_0,x_1),
        \end{equation}
        for some $L\in(0,+\infty)$. Then \eqref{lipschitz condition lemma} holds for any transport plan $\pi\in \Gamma(\mu_0,\mu_1)$.
    \end{lemma}

    \begin{proof}
        The proof is divided in two steps: first we prove the result for measures that are supported on finite sets, and then we use an approximation procedure to extend the result to all measures.

        \textbf{Step 1}: assume that $\mu_0,\mu_1$ have finite support and consider a generic transport plan $\pi\in \Gamma(\mu_0,\mu_1)$. Let $0=t_0<\dots<t_K = 1$ be as in Lemma \ref{interpolating curve lemma}, so that $(\mathrm{x}^{t_{k-1}},\mathrm{x}^{t_k})_\sharp \pi$ is a $W_p$-optimal plan between $(\mathrm{x}^{t_{k-1}})_\sharp \pi$ and $(\mathrm{x}^{t_k})_\sharp \pi$. It is also the unique optimal plan, see \cite[Lemma 7.2.1 and Theorem 7.2.2]{ambrosio2005gradient}. Then 
        \begin{align*}
            \bigg( \int |b(x_0, \mu_0) - & b(x_1,\mu_1)|^p d\pi(x_0,x_1)\bigg)^{1/p}
            \\
            & \leq 
            \sum_{k=1}^K \bigg(\int |b(\mathrm{x}^{t_{k-1}},  (\mathrm{x}^{t_{k-1}})_\sharp \pi) - b(\mathrm{x}^{t_{k}},(\mathrm{x}^{t_{k}})_\sharp \pi)|^p d\pi(x_0,x_1)\bigg)^{1/p}
            \\
            & \leq  \sum_{k=1}^K L W_p\big((\mathrm{x}^{t_{k-1}})_\sharp \pi, \mathrm{x}^{t_{k}}_\sharp \pi\big) 
            =  
            \sum_{k=1}^K L(t_k-t_{k-1}) \int|x_1-x_0|^p d\pi(x_0,x_1) 
            \\
            & = L \int|x_1-x_0|^p d\pi(x_0,x_1) .
        \end{align*}

        \textbf{Step 2}: let $\mu_0^n \in \PP(\R^d)$ (resp. $\mu_1^n\in \PP(\R^d)$) have finite support and be such that $W_p(\mu_0^n,\mu_0) \to 0$ (resp. $W_p(\mu_1^n,\mu_1) \to 0$). Let $\pi_0^n \in \Gamma(\mu_0^n,\mu_0)$ and $\pi_1^n \in \Gamma(\mu_1,\mu_1^n)$ be $W_p$-optimal plans for which \eqref{lipschitz condition lemma} is satisfied. Exploiting \cite[Proposition 8.6]{ambrosio2021lectures}, let $\sigma_n \in \PP\big((\R^d)^4\big)$ be such that
        \[p^{12}_\sharp \sigma_n = \pi_0^n, \ p^{23}_\sharp \sigma_n = \pi, \ p^{34}_\sharp \sigma_n = \pi^n_1,\]
        where $p^{ij}$ is the projection on both $i$-th and $j$-th coordinates. In particular, $p^{14}_\sharp \sigma_n \in \Gamma(\mu_0^n,\mu_1^n)$ and converges to $\pi$ w.r.t. $W_p$. Indeed, rearranging the coordinates of $\sigma^n$, we have a transport plan between $\pi$ and $(p_1,p_4)_\sharp \sigma^n$, which is $(p_2,p_3,p_1,p_4)_\sharp \sigma^n \in \Gamma(\pi,(p_1,p_4)_\sharp \sigma^n)$, and \[W_p^p(\pi,(p_1,p_4)_\sharp \sigma^n)\leq \int_{(\R^{d})^4}|y_2-y_1|^p + |y_3-y_4|^p d\sigma^n(y_1,y_2,y_3,y_4) = W_p^p(\mu_0^n,\mu_0) + W_p^p(\mu_1^n,\mu_1) \to 0.\]  
        Then, using the notation $(y_1,y_2,y_3,y_4)\in (\R^d)^4$, we have 
        \begin{align*}
            \bigg(\int & |b(x_0,\mu_0) - b(x_1,\mu_1)|^p d\pi(x_0,x_1)\bigg)^{1/p} = \| b(y_2,\mu_0) - b(y_3,\mu_1) \|_{L^p(\sigma_n;\R^d)} 
            \\
            \leq & 
            \| b(y_2,\mu_0) - b(y_1,\mu_0^n) \|_{L^p} + \| b(y_1,\mu_0^n) - b(y_4,\mu_1^n) \|_{L^p} + \| b(y_4,\mu_1^n) - b(y_3,\mu_1) \|_{L^p} 
            \\
            \leq &
            L^{1/p} \left[ W_p(\mu_0^n,\mu_0) + W_p(\mu_1^n,\mu_1) + \left(\int |x_1-x_0|^p dp^{14}_\sharp \sigma_n(x_0,x_1)\right)^{1/p} \right],
        \end{align*}
        where in the last inequality we used the fact that \eqref{lipschitz condition lemma} holds for $p^{12}_\sharp \sigma_n = \pi_0^n$ and $p^{34}_\sharp \sigma_n = \pi^n_1$, and $p^{14}_\sharp \sigma_n$ is any transport plan between $\mu_0^n$ and $\mu_1^n$, that are finitely supported so that \eqref{lipschitz condition lemma} holds as well, thanks to Step 1. Then, we conclude passing to the limit as $n\to +\infty$.
    \end{proof}

    The next result shows how the Lipschitz property along all the possible transport plans implies a Lipschitz property in space. For a proof, we refer to \cite[Theorem 4.8, (1)]{cavagnari2020extension}.

    \begin{lemma}\label{W_p lipschitz implies spatial lipschitz}
        Let $b:\R^d\times \PP(\R^d)\to\R^d$ be satisfying the hypothesis of Lemma \ref{lipschitz along optimal plans implies lipschitz on any plan}. Then, for all $\mu\in \PP_p(\R^d)$, the map $b(\cdot,\mu) :\operatorname{supp}\mu \to \R^d$ is $L$-Lipschitz.
    \end{lemma}

    Now, we can proceed to prove point (ii) in Theorem \ref{main uniqueness theorem}. Notice that all the previous lemmas apply to any non-local vector field $(x,\mu)\mapsto b(t,x,\mu)$, for any $t$ such that $L(t)<+\infty$ (in particular for a.e. $t\in (0,T)$). So, for a.e. $t\in(0,T)$ and for every $\mu\in \PP_p(\R^d)$, the map $b(t,\cdot,\mu) : \operatorname{supp}\mu \to \R^d$ is $L(t)$-Lipschitz.
    \\
    Now, let $\lambda \in \operatorname{SPS}_b$ and consider $\gamma_0,\gamma_1\in \operatorname{supp}\lambda$ that satisfy $\dot\gamma_i(t) = b(t,\gamma_i(t),(e_t)_\sharp \lambda)$, $i=0,1$. Thanks to Lemma \ref{initial lemma section 6}, for $i=0,1$, we know that $(e_t)_\sharp \lambda \in \PP_p(\R^d)$ for any $t\in (0,T)$ and $\gamma_i(t) \in \operatorname{supp}(e_t)_\sharp \lambda$, because $\gamma_i\in \operatorname{supp}\lambda$. Then, for a.e. $t\in (0,T)$
    \begin{align*}
        \frac{d}{dt} |\gamma_0(t) - \gamma_1(t)|^2 = & 2\big(\gamma_0(t) - \gamma_1(t)\big)\cdot \big(b(t,\gamma_0(t),(e_t)_\sharp \lambda) - b(t,\gamma_1(t),(e_t)_\sharp \lambda )\big)
        \\
        \leq &
        2L(t)|\gamma_1(t) - \gamma_2(t)|^2.
    \end{align*}
    Thus, using Gronwall lemma and the continuity of the curves $\gamma_0$ and $\gamma_1$, we have $\gamma_0 = \gamma_1$, which implies that, defining $\overline{\mu}:= (e_0)_\sharp \lambda$, it holds
    \begin{equation}\label{repr formula for lambda}
        \lambda = \lambda_{\overline{\mu}} := \int \delta_{\gamma_{\overline{x}}} d\overline{\mu}(\overline{x}),
    \end{equation}
    where $\gamma_{\overline{x}}\in C_T(\R^d)$ is the unique curve in the support of $\lambda$ solving $\dot{\gamma}(t) = b(t,\gamma(t),(e_t)_\sharp \lambda)$ starting from $\gamma(0) = \overline{x} \in \R^d$. 
    Then, considering $\mathfrak{L}\in \PP(\PP(C_T(\R^d)))$ given by Theorem \ref{main theorem}, $\mathfrak{L}$-a.e. $\lambda \in \PP(C_T(\R^d))$ must be of the form \eqref{repr formula for lambda}, so that 
    \begin{equation}
        \mathfrak{L} = \int_{\PP(\R^d)} \delta_{\lambda_{\overline{\mu}}} d\overline{M}(\overline{\mu}),
    \end{equation}
    giving the uniqueness of $\mathfrak{L}$ satisfying $(E_0)_\sharp \mathfrak{L} = \overline{M}$.

    \begin{oss}
        The same proof works, in the case $p=2$, only assuming monotonicity of the vector field $b$, i.e. that for any $\mu_0,\mu_1 \in \PP_2(\R^d)$ there exists a $W_2$-optimal plan $\pi$ between $\mu_0$ and $\mu_1$ such that 
        \begin{equation}
            \int (x_0-x_1)\cdot \big( b(t,x_0,\mu_0) - b(t,x_1,\mu_1) \big) d\pi(x_0,x_1) \leq L(t) W_2^2(\mu_0,\mu_1),
        \end{equation}
        with $L\in L^1(0,T)$. Indeed, thanks to the aforementioned \cite{cavagnari2023lagrangian,cavagnari2020extension}, the same proofs work under this hypothesis.
    \end{oss}

\appendix

\section{Lusin sets, Souslin sets and measurable selection theorem}\label{appendix souslin}
Here we collect some results from \cite{schwariz1973radon} and \cite{Bogachev07} about Lusin and Souslin sets, that are persistently used throughout the paper. In particular, we recall (without proof) the universal measurability of Souslin subsets and a measurable selection theorem.

\begin{df}[Lusin sets]\label{def: lusin}
    A subset $L\subset X$ of a Hausdorff topological space $(X,\tau)$, is said to be a Lusin set if there exists a Polish space $(Y,\tau')$ and a continuous and bijective map $i:Y\to L$. The space $X$ is said to be a Lusin space if it is a Lusin set.
\end{df}

Lusin spaces shares many interesting properties. Here we list the ones that are useful for our presentation and we suggest \cite[Chapter 2]{schwariz1973radon} for further reading.

\begin{lemma}\label{lusin iff borel}
    Let $(X,\tau)$ be a Lusin space. Then $L\subset X$ is a Lusin set if and only if it is a Borel set. In particular, the same holds when $(X,\tau)$ is Polish.
\end{lemma}

\noindent A direct consequence of the definition of Lusin set and the previous lemma is the next corollary.

\begin{co}\label{cont and inj image of Borel sets}
    Let $(X,\tau)$ and $(Y,\tau')$ be Lusin spaces and $f:Y \to X$ continuous and injective. Then, $f(B)\subset X$ is Borel for any $B\subset Y$ Borel. 
\end{co}

\begin{lemma}\label{union of lusin}
    Let $(X,\tau)$ be a Hausdorff topological space and $L_n\subset X$ a sequence of Lusin sets of $X$. Then $\bigcup L_n $ is a Lusin set.
\end{lemma}

\begin{df}[Souslin sets]
    A subset $S\subset X$ of a Hausdorff topological space $(X,\tau)$ is said to be a \textit{Souslin set} if there exist a Polish space $Y$ and a continuous map $f:Y\to X$ such that $f(Y) = S$. The space $X$ is said to be a \textit{Souslin space} if it is a Souslin set.
    \\
    Let $\Tilde{\mathcal{S}}(X)$ be the class of all Souslin subsets of $X$ and $\mathcal{S}(X)$ be the $\sigma$-algebra generated by $\Tilde{\mathcal{S}}(X)$.
\end{df}

An important tool to better understand the structure of the Souslin subsets of a given space, is the so-called A-operation (or Souslin operation). In the following theorem, we see how they are connected.

\begin{teorema}
    Let $(X,\tau)$ be a Hausdorff topological space. Every Souslin subset $S\in \tilde{\mathcal{S}}(X)$ can be obtained from closed sets by means of the A-operation, i.e. for any $S\in \tilde{\mathcal{S}}(X)$ there exists a class of closed sets $\{C_{n_1,\dots,n_k}\}$, where $(n_1,\dots,n_k)$ is any possible finite sequence, such that 
    \begin{equation}
        S = \bigcup_{(n_i)\in \N^{\infty}}\bigcap_{k=1}^{+\infty} C_{n_1,\dots,n_k}.
    \end{equation}
    Moreover, the set $\Tilde{\mathcal{S}}(X)$ is closed under the A-operation. 
\end{teorema}

\begin{oss}
    Within this theorem, one can think that $\Tilde{\mathcal{S}}(X) = \mathcal{S}(X)$, but it is known to be false. A consequence of the theorem is that $\Tilde{\mathcal{S}}(X)$ is closed under countable union and intersection. The problems are given by the complement operation: indeed, \cite[Corollary 6.6.10]{Bogachev07}, if both $S$ and $S^c$ are Souslin, then $S$ is a Borel subset, and by \cite[Theorem 6.7.10]{Bogachev07} we know that there exists a Souslin set that is not Borel.
\end{oss}

Now, we show that the Souslin-measurable sets are \textit{universally measurable}.

\begin{prop}\label{univ meas of souslin}
    Let $(X,\tau)$ be a Hausdorff space and $\mu$ a Borel, positive and finite measure over $X$. Then, the $\sigma$-algebra $\mathcal{B}_\mu$ of all the $\mu$-measurable subsets is closed under the $A$-operation.
    \\
    In particular, $\mathcal{S}(X)\subset \mathcal{B}_\mu$ for all Borel, positive and finite measure $\mu$, and we say that the Souslin $\sigma$-algebra is universally measurable.
\end{prop}

In particular, we can use any Souslin-Borel measurable map to define a push-forward of a Borel measure and obtain a Borel measure, as the next Corollary shows.

\begin{co}\label{push-forward by a souslin-borel map}
    Let $X,Y$ be two Souslin topological spaces. Let $\mu\in \mathcal{M}_+(X)$ be  Borel measure on $X$ and $f:X\to Y$ be a Souslin-Borel measurable map, i.e. $f^{-1}(B) \in \mathcal{S}(X)$ for any $B\in \mathcal{B}(Y)$. Then $\nu:= f_\sharp \mu \in \mathcal{M}_+(Y)$ is a well-defined measure over $Y$. 
\end{co}

Finally, we state a version of the \textit{measurable selection theorem}, see \cite[Theorem 6.9.1]{Bogachev07}.

\begin{teorema}\label{measurable selection}
    Let $X$ and $Y$ be two Souslin topological spaces. Let $F:X\to Y$ be a surjective Borel map. Then, there exists a $(\mathcal{S}(Y),\mathcal{B}(X))$-measurable map $G:Y \to X$ that is a right-inverse of $F$, i.e. $F(G(y))=y$ for any $y\in Y$. In addition, the image of $G$ belongs to $\mathcal{S}(X)$.
\end{teorema}

\section{The extended metric-topological structure of \texorpdfstring{$\R^\infty$}{}}\label{appendix R^infty}
In this section, we describe some properties of the space $\R^\infty$, i.e. the space of real sequences $\mathrm{x} = (x_n)_{n\in \N}$, endowed with two different structures:
\begin{itemize}
    \item the metric \begin{equation}\label{eq: distance D_infty}
    D_\infty(\mathrm{x}, \mathrm{y}):= \sup_{n\in \N} |x_n - y_n| \wedge 1,
    \end{equation}
    inducing the uniform convergence;
    \item the topology $\tau_w$ induced by the element-wise convergence, which is the topology given by the metric
    \begin{equation}
        d_\infty(\mathrm{x},\mathrm{y}) := \sum_{n\in \N} \frac{|x_n-y_n|\wedge 1}{2^n},
    \end{equation}
    that is complete, and thus the topological space $(\R^\infty,\tau_w)$ is Polish. We will always consider on $\R^\infty$ the Borel $\sigma$-algebra generated by $\tau_w$.
\end{itemize}

This structure is related to $(\PP(\R^d),\hat{W}_1)$. Let $\mathcal{A}=\{\phi_1,\phi_2,\dots \}\subset C_c^1(\R^d)$ be satisfying:
    \begin{itemize}
        \item[(i)] $\phi_k$ $1$-Lipschitz w.r.t. $|\cdot | \wedge 1$, in particular $\|\nabla\phi_k\|_\infty\leq 1$ for all $k\in \N$;
        \item[(ii)] $\operatorname{Span}(\mathcal{A})$ dense in $C_0^1(\R^d)$;
        \item[(iii)]\label{third property A}$\hat{W}_1(\mu,\nu) = \sup_k\int_{\R^d}\phi_k d(\mu-\nu)$.
    \end{itemize}
    The existence of such family is justified from the fact that the countable class of functions
    \[\mathcal{A}' := \left\{\phi_{k,n}(x) = \big(|x-x_k|\wedge 1 \big)* \rho_{\frac{1}{n}}, \quad k,n\in \N \right\}\]
    satisfies the third condition, where $\{x_k\}\subset \R^d$ is a countable and dense subset of $\R^d$, $\rho(x)$ is a mollifier and $\rho_{\varepsilon}(x) = \frac{1}{\varepsilon^d}\rho(\frac{x}{\varepsilon})$. On the other hand, the unit ball of $C_0^1(\R^d)$, endowed with the $C^1$-norm, is separable. Now, define 
    \begin{equation}\label{appendix: iota Rd}
    \iota:\PP(\R^d) \to \R^\infty, \quad \iota(\mu) = \big(L_{\phi_1}(\mu), L_{\phi_2}(\mu),\dots \big).
    \end{equation}  

    \begin{lemma}\label{lemma: im of iota is closed}
        The map $\iota$ is an isometry between $(\PP(\R^d),\hat{W}_1)$ and $(\iota(\PP(\R^d)),D_\infty)$. In particular, $\iota(\PP(\R^d))$ is bounded and closed w.r.t. $D_\infty$, while it is Borel but not closed with respect to $\tau_w$.
    \end{lemma}

\begin{proof}
    The fact that $\iota$ is an isometry comes from \eqref{third property A} of $\mathcal{A}$, which tells us
    \[D_{\infty}(\iota(\mu),\iota(\nu)) = \sup_n |L_{\phi_n}(\mu) - L_{\phi_n}(\nu)| \wedge 1 =  \sup_n |L_{\phi_n}(\mu) - L_{\phi_n}(\nu)| = \hat{W}_1(\mu,\nu).\]
    Then, the set $\iota(\PP(\R^d)$ must be bounded and closed w.r.t. $
    D_\infty$. On the other hand, consider the sequence of measures $\mu_n := \delta_{x_n}$, with $x_n \in \R^d$ such that $|x_n|\to+\infty$. Then, $\iota(\mu_n) \to 0$, i.e. $\iota(\mu_n)$ converges element-wise to $0$, but the sequence given by all zeros is not in $\iota(\PP(\R^d))$, thus $\iota(\PP(\R^d))$ is not closed in $\tau_w$. Anyway, $\iota$ is continuous when endowing $\PP(\R^d)$ with the narrow topology and $\R^\infty$ with $\tau_w$. Since it is also injective and $(\R^\infty,\tau_w)$ is Polish, then $\iota(\PP(\R^d))$ is Borel because it is Lusin (see Lemma \ref{lusin iff borel} and Corollary \ref{cont and inj image of Borel sets}).
\end{proof}

We proceed by introducing the cylinder functions for $\R^\infty$, the continuity equation over $\R^\infty$ and the superposition principle. For the proofs, we refer to \cite[Chapter 7]{ambrosio2014well}.

\begin{df}[Cylinder functions of $\R^\infty$]\label{cylinder over R^infty}
    We say that a function $F:\R^\infty\to\R$ is a cylinder function, and we write $F\in \operatorname{Cyl}^1(\R^\infty)$, if there exists $k\in \N$ and $\Psi\in C^1(\R^k)$ bounded and with bounded derivatives, such that 
    \begin{equation}
        F(\mathrm{x}) = \Psi(\pi_k(\mathrm{x})) = \Psi(x_1,\dots,x_k) \quad \forall x\in \R^\infty.
    \end{equation}
    Its gradient is then defined as 
    \begin{equation}
        \nabla F(\mathrm{x}) = \left(\partial_1 \Psi(\pi_k(\mathrm{x})),\dots, \partial_k \Psi(\pi_k(\mathrm{x})), 0 ,0 ,\dots\right).
    \end{equation}
\end{df}

\begin{df}[Continuity equation over $\R^\infty$]\label{CE over R^infty}
    Let $v:[0,T]\times \R^\infty \to \R^\infty$ be a Borel vector field (w.r.t. $\tau_w$) and $(\mathtt{m}_t)_{t\in [0,T]} \subset \PP(\R^\infty)$ a $\text{weakly}^*$ continuous curve of probability measures over $\R^\infty$. Then, we say that the continuity equation $\partial_t\mathtt{m}_t + \operatorname{div}(v_t\mathtt{m}_t) = 0$ holds if 
    \begin{equation}
        \int_0^T \int |v_t^{(k)}| d\mathtt{m}_t dt<+\infty \quad \forall k\in \N,
    \end{equation}
    where the superscript $k$ indicates the $k$-th component of $v$, and 
    \begin{equation}
        \frac{d}{dt}\int F d\mathtt{m}_t = \int \nabla F \cdot v_t d\mathtt{m}_t,\quad  \forall F\in \operatorname{Cyl}^1(\R^\infty)
    \end{equation}
    in the sense of distribution of $(0,T)$.
\end{df}

\begin{df}\label{AC_w over R^infty}
    The set $AC_w([0,T],\R^\infty)\subset C([0,T],(\R^\infty,\tau_w))$ is the set of $\tau$-continuous curves $\gamma:[0,T]\to \R^\infty $ that are element-wise absolutely continuous.
\end{df}

\begin{teorema}[Superposition principle over $\R^\infty$]\label{superposition R^infty}
    Let $v:[0,T]\times \R^\infty \to \R^\infty$ and $(\mathtt{m}_t)_{t\in [0,T]}$ as in the previous definition, satisfying $\partial_t\mathtt{m}_t + \operatorname{div}(v_t\mathtt{m}_t) = 0$. Then there exists a probability measure $\mathtt{L} \in \PP(C_w([0,T],\R^\infty))$ such that $(e_t)_\sharp \mathtt{L} = \mathtt{m}_t$, it is concentrated over $AC_w([0,T],\R^\infty) $ and for $\lambda$-a.e. $\tilde\gamma$, it holds
    \begin{equation}
        \frac{d}{dt}\tilde\gamma^{(k)}(t) = v_t^{(k)}(\tilde\gamma_t) \quad \forall k\in \N.
    \end{equation}
\end{teorema}

\section{Measurability in the space of curves}\label{appendix meas for curves}
We collect here some results concerning measurability properties of subsets of $C_T(Y)$, where $Y$ is a Polish space and 
$C_T(Y)$ is endowed with the distance $D_{d}$
where $d$ is a distance on $Y$ generating its topology.

\begin{lemma}\label{meas of AC^p}
    For any $p\in [1,+\infty)$, the space $AC^p_T(Y)$ is a Borel subsets of $C_T(Y)$, both endowed with the sup distance.
\end{lemma}

\begin{proof}
    By Theorem 10.2, \cite{ambrosio2021lectures}, the functional $a_p$ is lower semicontinuous (their proof can be easily extended to the case $p\in(1,+\infty)$ using again Lemma 10.1, \cite{ambrosio2021lectures}). Then, the sublevel sets are closed, and by the definitions above we can write $AC^p$ as the union of the sublevel sets at level $n \in \N$. Then it is $F_\sigma$, and so it is Borel. 
    Regarding the case $p=1$, we refer to \cite[§2.2]{ambrosio2014calculus}
\end{proof}

\begin{lemma}\label{meas of D(t,gamma)}
    Let $D:[0,T]\times C_T(\R^d) \to \big(\R\cup\{\pm \infty\}\big)^d$ defined as the pointwise and component-wise $\limsup$ of the discrete derivatives, i.e. 
        \[\big(D(t,\gamma)\big)_j = \limsup_{h \to 0}\frac{\gamma_j(t+h)- \gamma_j(t)}{h}, \quad j = 1,\dots,d.\]
    Then, the function $D$ is Borel measurable and for all $\lambda \in \PP(C_T(\R^d))$ concentrated over $AC_T(\PP(\R^d))$, $D(t,\gamma)\in \R^d$ for $\mathcal{L}^1_T\otimes \lambda$-a.e. $(t,\gamma)$.
\end{lemma}

\begin{proof}
    Being the $\limsup$ of continuous functions, $D$ is Borel. The last property follows from the fact that for $\lambda$-a.e. $\gamma$ and a.e. $t\in (0,T)$, the derivative of $\gamma$ in $t$ exists and is finite. 
\end{proof}

\subsection{Curves in \texorpdfstring{$\R^\infty$}{}}
Here we show a measurability result, linking the two spaces of continuous curves on $\R^\infty$ with respect to the two topologies presented in Appendix \ref{appendix R^infty}, i.e. $C_T(\R^\infty, D_\infty)$ and $C_T(\R^\infty, \tau_w)$. In particular, we will need the following result.

\begin{lemma}\label{lemma: meas of curves in R^infty}
    Let $\iota:\PP(\R^d)\to\R^\infty$ be as in \eqref{appendix: iota Rd}. The set $C_T(\iota(\PP(\R^d)),D_\infty)$ is a Borel subset of $C_T(\R^\infty,\tau_w)$. Moreover, the Borel $\sigma$-algebra generated by its subspace topology induced by $C_T(\R^\infty,D_\infty)$ coincides with the Borel $\sigma$-algebra generated by the subspace topology induced by $C_T(\R^\infty, \tau_w)$.
\end{lemma}

\begin{proof}
    The space $(\iota(\PP(\R^d)),D_\infty)$ is Polish, since it is isometric to $(\PP(\R^d),\hat{W}_1)$, then the same holds for $C_T(\iota(\PP(\R^d)),D_\infty)$, with its natural compact-open topology, that we call $\pazocal{T}$. Such a topology is clearly stronger than the subspace topology induced by the larger (Polish) space $C_T(\R^\infty, \tau_w)$, that we denote by $\pazocal{T}_w$, so that the map
    \[\operatorname{id}:\big(C_T(\iota(\PP),D_\infty), \pazocal{T}\big) \to \big(C_T(\R^\infty,\tau_w),\pazocal{T}_w\big)\]
    is continuous and injective, i.e. $C_T(\iota(\PP),D_\infty)$, as a topological subspace of $C_T(\R^\infty,\tau_w)$ is Lusin by Definition \ref{def: lusin}, and then Borel by Lemma \ref{lusin iff borel}. 
    To conclude, notice that both the topology over $C_T(\iota(\PP),D_\infty)$ makes it a Lusin space and they are comparable; then the induced Borel $\sigma$-algebras coincide by \cite[Corollary 2, pp. 101]{schwariz1973radon}.
\end{proof}

\subsection{Curve of random measures}

\begin{lemma}\label{lemma: dual cyl R infty}
    For any $\mathtt{m}_0,\mathtt{m}_1\in \PP(\R^\infty)$ it holds
    \begin{equation}\label{dual formula R^infty}
    \begin{aligned}
        W_{1,D_\infty} &(\mathtt{m}_0,\mathtt{m}_1) =  \sup \bigg\{ \int_{\R^\infty} F d(\mathtt{m}_0-\mathtt{m}_1) \ : \ F(x)\in\operatorname{Cyl}_c^1(\R^\infty), \ F \ 1\text{-}\operatorname{Lip} \text{ w.r.t. } D_\infty\bigg\}
        \\
        = & \sup \bigg\{ \int_{\R^\infty}\hspace{-0.2cm} \Psi\circ \pi_k d(\mathtt{m}_0-\mathtt{m}_1)  :  k\in \N, \, \Psi \in C_c^1(\R^k), \, \|\Psi\|_\infty\leq 1/2,\, \bigg\|\sum_{i=1}^k|\partial_i\Psi|\bigg\|_\infty \leq 1\bigg\}
        .
    \end{aligned}
    \end{equation}
\end{lemma}

\begin{proof}
We prove the first inequality, while the second one will be a byproduct of our argument.
    The $\geq$ inequality is trivial, so we focus on the other one. Define the projection functions
    \[\pi^n:\R^\infty \to \R^\infty, \quad \pi^n(\mathrm{x}) = (x_1,\dots,x_n,0,0,\dots)\]
    \[ \hspace{-1,3cm}p^n:\R^\infty\to\R^n, \quad p^n(\mathrm{x})= (x_1,\dots,x_n).\]
    Notice that $W_{1,D_\infty}(\pi^n_\sharp \mathtt{m}_0,\pi^n_\sharp \mathtt{m}_1)\leq W_{1,D_\infty}(\mathtt{m}_0,\mathtt{m}_1)$. Moreover, for all $\mathtt{m}\in \PP(\R^\infty)$, we have $\pi^n_\sharp \mathtt{m}\to\mathtt{m}$ weakly in duality with $C_b(\R^\infty,\tau_w)$. Indeed, for any $\mathrm{x}\in \R^\infty$, $\pi^n(\mathrm{x})\overset{\tau}{\to}\mathrm{x}$, so for any $F\in C_b(\R^\infty,\tau_w)$, by the dominated convergence theorem, it holds
    \[\int F(\mathrm{x}) d(\pi^n_\sharp \mathtt{m})(\mathrm{x}) = \int F(\pi^n(\mathrm{x})) d\mathtt{m}(\mathrm{x}) \to \int F(\mathrm{x}) d\mathtt{m}(x).\]
    Moreover, $D_\infty$ is $\tau_w \otimes \tau_w$ l.s.c.: indeed consider $\mathrm{x}^{(n)},\mathrm{y}^{(n)}\in \R^\infty$ respectively, converging component-wise to $\mathrm{x}$ and $\mathrm{y}$, then 
    \begin{align*}
        \liminf_{n\to+\infty} D_\infty(\mathrm{x}^{(n)},\mathrm{y}^{(n)}) 
        = &
        \liminf_{n\to+\infty} \sup_{j\in \N} |x_j^{(n)} - y_j^{(n)}|\wedge 1 
        \\
        \geq &
        \liminf_{n\to+\infty} \sup_{j\leq m} |x_j^{(n)} - y_j^{(n)}|\wedge 1 = \sup_{j\leq m} |x_j - y_j|\wedge 1,
    \end{align*}
    for any $m\in \N$, and for its arbitrariness we conclude. 
    \\
    Consider a sequence of $W_{1,D_\infty}$-optimal plans $\Pi_n \in \Gamma_0(\pi^n_\sharp \mathtt{m}_0,\pi^n_\sharp \mathtt{m}_1)$. By the convergence of the marginals and the fact that $(\R^\infty,\tau_w)$ is a Polish space, we have that the set of measures $\{\Pi_n\}$ is tight, which implies that 
    \[\forall (n_k) \ \exists (n_{k_j}) \ : \ \Pi_{n_{k_j}} \to \Pi\in\Gamma(\mathtt{m}_0,\mathtt{m}_1),\]
    where the convergence is weakly in duality with $C_b(\R^\infty\times\R^\infty, \tau\otimes\tau)$. Then, by $\tau_w\otimes\tau_w$-l.s.c. of $D_\infty$ we have 
    \begin{align*}
        \liminf_{j\to+\infty} W_{1,D_\infty}(\pi^{n_{k_j}}_\sharp \mathtt{m}_0,\pi^{n_{k_j}}_\sharp \mathtt{m}_1) 
        = &
        \liminf_{j\to+\infty} \int D_\infty(\mathrm{x},\mathrm{y}) d\Pi_{n_{k_j}}(\mathrm{x},\mathrm{y}) 
        \\
        \geq & 
        \int D_\infty(\mathrm{x},\mathrm{y}) d\Pi(\mathrm{x},\mathrm{y}) \geq W_{1,D_\infty}(\mathtt{m}_0,\mathtt{m}_1),
    \end{align*}
    and by arbitrariness of $(n_k)$ we conclude that $\liminf_{n\to+\infty} W_{1,D_\infty}(\pi^{n}_\sharp \mathtt{m}_0,\pi^{n}_\sharp \mathtt{m}_1) 
        \geq W_{1,D_\infty}(\mathtt{m}_0,\mathtt{m}_1) 
        $. Together with what was proved before, we have that 
        \begin{equation}
            \lim_{n\to+\infty} W_{1,D_\infty}(\pi^{n}_\sharp \mathtt{m}_0,\pi^{n}_\sharp \mathtt{m}_1) 
        = W_{1,D_\infty}(\mathtt{m}_0,\mathtt{m}_1) .
        \end{equation}
        Now, for any $x,y\in \R^n$ define $D_n(x,y) := \sup_{j\leq n} |x_j-y_j|\wedge 1$ and notice that 
        \begin{align*}
            W_{1,D_\infty}(\mathtt{m}_0,\mathtt{m}_1) 
            = & 
            \sup_{n\in \N} W_{1,D_\infty}(\pi^n_\sharp \mathtt{m}_0,\pi^n_\sharp \mathtt{m}_1)
            = 
            \sup_{n\in \N} W_{1,D_n}(\pi^n_\sharp \mathtt{m}_0,\pi^n_\sharp \mathtt{m}_1)
            \\
            = &
            \sup_{n\in \N} \ \sup_{\{\Psi\in C_c^1(\R^n): \Psi\  D_n \ 1-\operatorname{Lip}\}} \left\{ \int_{\R^n} \Psi(x) d\pi^n_\sharp \mathtt{m}_0(x) - \int_{\R^n} \Psi(y) d\pi^n_\sharp \mathtt{m}_1(y) \right\}
            \\
            = & \sup_{n, \Psi \text{ as above}} \int \Psi(\pi^n(\mathrm{x})) d(\mathtt{m}_0 - \mathtt{m}_1)(\mathrm{x}).
        \end{align*}
        The proof is then concluded by the fact that $\R^\infty \ni x \mapsto \Psi(\pi^n(x))$ is the general expression for a cylinder function from $\R^\infty\to\R$, and it is $D_\infty$ $1$-Lipschitz if and only if $\R^n \ni x \to \Psi(x)$ is $D_n$ $1$-Lipschitz. This is equivalent to ask that 
        \[\left\|\sum_{i=1}^n |\partial_i\Psi|\right\|_\infty  \leq 1 \quad \text{ and }\quad \operatorname{osc}\Psi:= \max \Psi - \min \Psi \leq 1.\]
        Moreover, up to considering a translation, we can substitute the condition on oscillation with $\|\Psi\|_\infty\leq 1/2$, which concludes the proof.
\end{proof}

In particular, the previous lemma characterizes Lipschitzianity with respect to $D_\infty$ of cylinder functions over $\R^\infty$ through conditions on the function $\Psi$ that represents it, that is
\[F = \Psi \circ \pi^n\text{ is }D_\infty \ 1\text{-Lipschitz }\iff \ \left\|\sum_{i=1}^n |\partial_i\Psi|\right\|_\infty \leq  1 \text{ and } \operatorname{osc} \Psi \leq 1.  \]

The following shows how to use the previous lemma to have a duality formula for $\hat{\mathcal{W}}_1$ using only duality with cylinder functions.

\begin{prop}\label{prop dual cyl}
    The following duality formula holds for the distance $\hat{\mathcal{W}}_1 = W_{1,\hat{W}_1}$: for all $M,N\in \PP(\PP(\R^d))$ we have 
    \begin{equation}
        \hat{\mathcal{W}}_1
        (M,N) = \sup_{F\in \operatorname{Cyl}_c^1(\PP(\R^d)), \ \operatorname{Lip}(F)\leq 1 } \int_{\PP(\R^d)} F(\mu) dM(\mu) - \int_{\PP(\R^d)}F(\nu) dN(\nu), 
    \end{equation} 
    where the Lipschitz constant has to be intended w.r.t. the $\hat{\mathcal{W}}_1$ distance.
\end{prop}

\begin{proof}
    It follows easily by Lemma \ref{lemma: dual cyl R infty}, indeed, thanks to \eqref{dual formula R^infty}, we can conclude
    \begin{align*}
        W_{1,\hat{W}_1}&(M,N) =  W_{1,D_\infty}(\iota_{\sharp }M,\iota_\sharp  N)
        =  \sup \int_{\R^\infty} F d(\iota_\sharp M - \iota_\sharp N) 
        \\
        = & \sup \bigg\{\int_{\PP(\R^d)} \Psi(L_{\phi_1}(\mu),\dots,L_{\phi_n}(\mu)) dM(\mu) 
        - \int_{\PP(\R^d)} \Psi(L_{\phi_1}(\nu),\dots,L_{\phi_n}(\nu)) dM(\nu)\bigg\},
    \end{align*}
    noticing that if $F\in\operatorname{Cyl}^1(\R^\infty)$ and it is $D_\infty$ $1$-Lipschitz then, again by the fact that $\iota$ is an isometry, we have that $\Psi(L_{\phi_1}(\mu),\dots,L_{\phi_n}(\mu)) \in \operatorname{Cyl}_c^1(\PP(\R^d))$ is $\hat{\mathcal{W}}_1$ $1$-Lipschitz.
\end{proof}

\begin{oss}\label{countable duality wass on wass}
    We actually proved a stronger result: the Wasserstein distance $W_{1,W_1}$ between $M$ and $N$ can be recovered by the cylinder functions depending only on the functions $(L_{\phi_n})_{n\in \N}$, where $(\phi_n)_{n\in \N}\subset C_c^1(\R^d)$ are $1$-Lipschitz functions from $\R^d$ to $\R$ (w.r.t. $|\cdot|\wedge 1$) such that 
    \[W_{1,|\cdot|\wedge 1}(\mu,\nu) = \sup_{n\in \N} \int_{\R^d} \phi_n(x) d(\mu-\nu)(x).\]
    Moreover, the last equality in the proof of Lemma \ref{lemma: dual cyl R infty}
    \[W_{1,D_{\infty}}(\mathtt{m}_0,\mathtt{m}_1) = \sup_{n, \Psi} \int \Psi(p^n(\mathrm{x})) d(\mathtt{m}_0 - \mathtt{m}_1)(\mathrm{x}),\]
    is telling us that for any $n\in \N$, we can consider a countable family $\mathcal{F}_n$ of $1$-Lipschitz functions (w.r.t. $D_n$) such that the $\sup$ is realized taking $F\in \mathcal{F}_n$. Defining the countable family $\mathcal{F}:= \cup \mathcal{F}_n$ and taking in considerations what has been said above, we have that 
    \begin{equation}\label{equation countable duality}
        \hat{\mathcal{W}}_1(M,N) = \sup_{F\in \mathcal{F}} \int_{\PP(\R^d)} F(L_{\psi_1}(\mu),\dots,L_{\psi_n}(\mu)) d(M-N)(\mu).
    \end{equation}
\end{oss}

\begin{lemma}\label{abs cont sol of CE}
    Let $(M_t)_{t\in[0,T]}\subset\PP(\PP(\R^d)) $ and $(B_t)_{t\in [0,T]}$ be an $L^1(M_t \otimes dt)$-derivation such that \eqref{CE meas 2} holds. Then there exists a curve $t\mapsto \overline{M}_t$ such that $\overline{M}_t = M_t$ for a.e. $t\in [0,T]$ and $(\overline{M}_t)\in AC_T\big(\PP(\PP(\R^d)),\hat{\mathcal{W}}_1\big) \subset C_T(\PP(\PP(\R^d)))$. 
\end{lemma}

\begin{proof}
    For any $F\in \operatorname{Cyl}_c^1(\PP(\R^d))$, the map $[0,T]\ni t \mapsto \int F(\mu) dM_t(\mu)$ is $W^{1,1}(0,T)$ with distributional derivative $\int B_t[F](\mu) dM_t(\mu)$. In particular, there exists $I_{F}\subset (0,T)$ of full Lebesgue measure such that for all $s,t \in I_F$
    \[\int FdM_t - \int FdM_s =  \int_s^t \int B_r[F](\mu) dM_r(\mu) dr.\]
    Consider $\mathcal{C}\subset \operatorname{Cyl}_c^1(\PP(\R^d))$ the countable set in the supremum of \eqref{equation countable duality} and take the full Lebesgue measure set $I= \cap_{F\in \mathcal{C}} I_F$. Then for any $s,t \in I$ we have 
    \begin{align*}
        \hat{\mathcal{W}}_1(M_t,M_s) 
        = &
        \sup_{F \in \mathcal{C}} \int F\big(L_{\psi_1},\dots,L_{\psi_n}\big) d(M_t-M_s)
        \\
        = &
        \sup_{F \in \mathcal{C}} \int_{s}^t \frac{d}{dr}\int F\big(L_{\psi_1}(\mu),\dots,L_{\psi_n}(\mu)\big) dM_r(\mu) dr 
        \\
        = & \sup_{F \in \mathcal{C}} \int_s^t \int B_r[F](\mu) dM_r(\mu) dr 
        \leq  \int_s^t \int c_r(\mu)dM_r dr,
    \end{align*}
    where the last inequality follows from the fact that $|F(x)-F(y)|\leq D_n(x,y)\leq \sup_{i\leq n}|x_i-y_i|$, which implies that the $1$-norm in $\R^n$ of the gradient of $F$ is bounded by $1$, i.e.
    \[\sum_{i=1}^n|\partial_i F(x)|\leq 1 \quad \forall x\in \R^n. \]
    This implies that $\forall x\in \R^d$ and $\forall \mu\in \PP(\R^d)$
    \begin{align*}
        |\nabla_W F(x,\mu)| \leq & \sum_{i=1}^n |\partial_i F\big(L_{\psi_1}(\mu),\dots,L_{\psi_n}(\mu)\big)| |\nabla \psi_i|
        \leq 
        \sum_{i=1}^n |\partial_i F\big(L_{\psi_1}(\mu),\dots,L_{\psi_n}(\mu)\big)| \leq 1,
    \end{align*}
    following also from the fact that for all $n\in \N$, $\psi_n$ is $|\cdot|\wedge 1$ $1$-Lipschitz, so in particular they are $1$-Lipschitz w.r.t. the usual Euclidean norm $|\cdot|$ in $\R^d$.
    \\
    Then, there exists a curve of random measures $t \mapsto \overline{M}_t$ that is absolutely continuous w.r.t. $W_{1,\hat{W}_1}$ and $M_t = \overline{M}_t$ for all $t\in I$.
\end{proof}

\section{Measurability in the spaces of measures}\label{app: measurability}
In this section, we state some general results of Borel measurability in the spaces $\mathcal{M}_+(Y)$, $\mathcal{M}(Y,\R^d)$ and their product, where $Y$ is a Polish space. In $\mathcal{M}_+(Y)$ and $\mathcal{M}(Y,\R^d)$ we consider the weak topology in duality with $C_b$ functions, and the product topology over product spaces.

\subsection{Equivalence between \texorpdfstring{$\sigma$}{}-algebras over \texorpdfstring{$\mathcal{M}_+(Y)$}{}}

Let $(Y,\tau)$ be a Polish space. The goal of this subsection is to prove that the Borel $\sigma$-algebra of $\mathcal{M}_+(Y)$, endowed with the narrow topology, is the same as the smallest $\sigma$-algebra $\mathcal{S}$ that makes measurable the evaluation on Borel sets, i.e. such that for any Borel set $A\subset Y$ the map $\mathcal{M}_+(Y) \ni \mu \mapsto \mu(A) \in \R$ is $\mathcal{S}$-measurable.

\begin{lemma}[Measurability of $\mu\mapsto \int gd\mu$]\label{measurability int g d mu}
    Let $Y$ be a Polish space. For each Borel map $g:Y\to [0,+\infty]$ 
     \[G: \mathcal{M}_+(Y)\to[0,+\infty], \quad G(\mu):= \int_Y g \ d\mu\]
     is a Borel function. In particular, the set 
     \[\PP_g(Y):= \left\{\mu \in \PP(Y) \ : \ \int_Y g d\mu <+\infty\right\}\]
     is a Borel subset of $\PP(Y)$.
\end{lemma}

\begin{proof}
    Define \[\mathcal{H}:= \{h:Y\to\R \ : \ h\text{ is Borel and bounded, } \mu \mapsto \int_Y h \ d\mu \text{ is Borel}\}.\]
    Obviously $\mathcal{H}$ contains $C_b(Y)$. Moreover, $\mathcal{H}$ is closed under monotone limits, indeed if $\mathcal{H}\ni h_n \to h$ monotonically, then by dominated convergence theorem, for any $\mu \in \mathcal{M}_+(Y)$ 
    \[H(\mu):= \int_Y h \ d\mu = \lim_{n\to+\infty} \int_Y h_n d\mu,\]
    so $H$ is the pointwise limit of Borel functions, thus it is Borel, which implies that $h\in \mathcal{H}$. Then we can apply \cite[Theorem 2.12.9, (iii)]{Bogachev07} to conclude that $h\in \mathcal{H}$ for any bounded and Borel function $h:Y \to \R$. Then, by monotone convergence theorem, we conclude that $G$ is a Borel function approximating $g$ pointwise with $g_n:= g\wedge n$.
\end{proof}

\begin{lemma}\label{weak conv is metrizable}
    The narrow topology over $X := \mathcal{M}_+(Y)$ is metrizable. 
\end{lemma}

\begin{proof}
    Thanks to \cite{Bogachev07}, Theorem 8.3.2, the narrow topology over $\mathcal{M}_+(Y)$ is induced by the norm 
    \begin{equation}
        \lVert \mu \rVert_{BL}:= \sup  \left\{ \int_Y \phi d\mu \ : \ \phi\in \operatorname{Lip}_b(Y), \ \operatorname{LIP}(\phi)\leq 1 \right\}.
    \end{equation}
\end{proof}

Before moving to the main goal of this section, we need a general lemma.

\begin{lemma}
    Let $Z$ be a Polish space. Then $\mathcal{B}(Z)$ is the smallest $\sigma$-algebra under which continuous functions from $Z$ to $\R$ are measurable. 
\end{lemma}

\begin{proof}
    Let $C\subset Z$ be a closed subset. Define the continuous function $\delta_C(z):= \inf\{d_Z(y,z) \ : \ y \in C\}$. Notice that $C = \delta_C^{-1}(\{0\}).$ Let $\mathcal{S}$ be the smallest $\sigma$-algebra on $Z$ such that all continuous functions from $Z$ to $\R$ are measurable. For sure $\mathcal{S}\subset \mathcal{B}(Z)$. On the other hand, for any closed set $C\subset Z$, we have that $C = \delta_C^{-1}(\{0\}) \in \mathcal{S}$, so $\mathcal{B}(Z)\subset \mathcal{S}$.
\end{proof}

We will use this lemma with $Z=\mathcal{M}_+(Y)$, where $Y$ is a Polish space. To this aim, define the class of functions 
\[\mathcal{F}:=\{F :\mathcal{M}_+(Y) \to \R \ : \ F(\mu) = \int f d\mu, \ f\in C_b(Y)\}.\]

\begin{lemma}
    Let $\mathcal{C}$ be the collection of all narrowly continuous functions from $\mathcal{M}_+(Y)$ to $\R$.
    Let $\mathcal{S}_\mathcal{F}$ (resp. $\mathcal{S}_\mathcal{C}$) be the smallest $\sigma$-algebra over $\mathcal{M}_+(Y)$ that makes measurable the functions in $\mathcal{F}$ (resp. $\mathcal{C}$).
    Then $\mathcal{S}_\mathcal{F} = \mathcal{S}_\mathcal{C}=\mathcal{B}(\mathcal{M}_+(Y)) $.
\end{lemma}

 \begin{proof}
    The fact that $\mathcal{B}(\mathcal{M}_+(Y)) = \mathcal{S}_\mathcal{C}$ follows from the previous lemma.
     Regarding $\mathcal
     S_\mathcal{F}$, trivially we have $\mathcal{S}_\mathcal{F} \subset \mathcal{S}_\mathcal{C}$. On the other hand, notice that the countable class of functions $\mathcal{C}_0\subset \mathcal{F}$, introduced in \cite[Remeark 5.1.1]{ambrosio2005gradient}, is sufficient to describe the narrow topology $\mathcal{T}$ of $\mathcal{M}_+(Y)$. This means that $\mathcal{T}$ coincides with the smallest topology that makes continuous the functions in $\mathcal{C}_0$. Then 
    \begin{align*}
        \mathcal{B}(\mathcal{M}_+(Y)) = \sigma(\tau) = \sigma\left( \left\{ 
        F^{-1}((a,b)) \ : \ F\in \mathcal{C}_0, \ a<b, \ a,b\in \mathbb{Q} \right\} \right) \subset \mathcal{S}_\mathcal{F}.
    \end{align*}
 \end{proof}

\begin{prop}\label{equiv of sigma alg prop}
    Let $Y$ be a Polish space and $\mathcal{S}$ be the smallest $\sigma$-algebra on $\mathcal{M}_+(Y)$ that makes measurable the functions $\mu\mapsto \mu(A)$ for any $A\in \mathcal{B}(Y)$. Then $\mathcal{S}$ coincides with $\mathcal{B}(\mathcal{M}_+(Y))$.
\end{prop}

\begin{proof}
    Thanks to Lemma \ref{measurability int g d mu}, it holds that $\mathcal{S}\subset \mathcal{B}(\mathcal{M}_+(Y))$. On the other hand, the integral of step functions is $\mathcal{S}$-measurable, and then also $\mu \mapsto \int fd\mu$ is measurable for any $f \in C_b(Y)$. This implies that $\mathcal{S}_\mathcal{F}\subset \mathcal{S}$, and thanks to the previous lemma we are done.
\end{proof}

\begin{co}\label{equiv of sigma alg cor}
    Let $X$ be a topological space and  $X \ni x \mapsto \mu_x \in \mathcal{M}_+(Y)$ be a map taking values in $\mathcal{M}_+(Y)$. Then 
    \[x\mapsto \mu_x \text{ is Borel } \iff  \ x\mapsto \mu_x(A) \text{ is Borel }\forall A\in \mathcal{B}(Y).\]
\end{co}

This corollary directly shows the measurability of the family of measures given by the \textit{disintegration theorem}, that we recall here for completeness.

\begin{teorema}\label{disintegration theorem}
    Let $Y,X$ be Polish spaces, $\mu\in \mathcal{M}_+(Y)$ and $e:Y\to X$ a Borel function. Define $\theta := e_\sharp \mu \in \mathcal{M}_+(X)$. Then there exists a family $\{\mu_x\}_{x\in X}\subset \PP(Y)$ such that 
    \begin{enumerate}
        \item[(i)] $x\mapsto \mu_x(A)$ is Borel measurable for any $A\in \mathcal{B}(Y)$;
        \item[(ii)] $\mu(dz) = \int_X \mu_x(dy) d\theta(x)$;
        \item[(iii)] $\mu_x$ is concentrated on $e^{-1}(\{x\})$ for $\theta$-a.e. $x\in X$.
    \end{enumerate}
Moreover, such a disintegration is unique, in the sense that if another family $\{\mu'_x\}_{x\in X}$ satisfies these properties, then $\mu_x = \mu'_x$ for $\theta$-a.e. $x\in X$. 
\end{teorema}

\subsection{Measurability of sets and maps}
\begin{prop}[Measurability of $\mu \mapsto f_\sharp \mu$]\label{meas of push forward}
    Let $X,Y$ be two Polish spaces and $f:X\to Y$ be a Borel measurable map. Then $f_\sharp :\mathcal{M}_+(X) \to \mathcal{M}_+(Y)$ is Borel measurable. The same holds for the restriction $f_\sharp :\PP(X) \to \PP(Y)$.
    If, additionally, $f$ is continuous, we have that $f_\sharp : \mathcal{M}(X;\R^d) \to \mathcal{M}(Y;\R^d)$ is continuous, and in particular Borel measurable.
\end{prop}

\begin{proof}
    By Corollary \ref{equiv of sigma alg cor}, the map $\mu\mapsto f_\sharp  \mu $\ is measurable if and only if for all $B\in \mathcal{B}(Y)$ the map $\mu \mapsto f_\sharp \mu(B)$ is measurable. By Lemma \ref{measurability int g d mu} we conclude, since $A = f^{-1}(B) \in \mathcal{B}(X)$ and
    \[f_\sharp \mu(B) = \mu(f^{-1}(B)) = \int_X \mathds{1}_{A} d\mu\]
    is measurable. For the second part, simply notice that $f_\sharp |_{\PP(X)}$ maps $\PP(X)$ to $\PP(Y)$. 

    Thanks to \cite[Proposition 3.2]{brezis2011functional}, we only need to show that $\mathcal{M}(X;\R^d)\ni \mu \mapsto \int_Y \phi d(f_\sharp \mu) \in \R $ is continuous for all $\phi \in C_b(Y;\R^d)$. We conclude noticing that $\int_Y \phi d(f_\sharp \mu) = \int_X \phi \circ f d\mu$, which is continuous since $\phi \circ f \in C_b(X;\R^d)$.
\end{proof}

\begin{teorema}[Measurability of $\mu \mapsto f\mu$]
    Let $Y$ be a Polish space. Let $f:Y\to [0,+\infty]$ be a bounded and Borel map. Then the map
    \[F:\mathcal{M}_+(Y) \to \mathcal{M}_+(Y), \quad F(\mu) = f\mu,\]
    is Borel, endowing $\mathcal{M}_+(Y)$ with the weak topology w.r.t. the duality with $C_b(Y)$ functions.
\end{teorema}

\begin{proof}
    Thanks to Corollary \ref{equiv of sigma alg cor}, it suffices to prove that for all $A \in \mathcal{B}(Y)$, the map 
    $\mu \mapsto \int_A fd\mu$
    is measurable. We are done by Lemma \ref{measurability int g d mu}, with $g(y) = f(y) \mathds{1}_A(y)$.
\end{proof}

\begin{co}
    Assume that $f:Y \to [0,+\infty]$ is Borel. Then the map
    \[F:\mathcal{M}_{+,f}(Y) \to \mathcal{M}_+(Y), \quad F(\mu) = f\mu,\]
    is Borel, where $\mathcal{M}_{+,f}(Y):=\{\mu \in \mathcal{M}_+(Y) \ : \ f\in L^1(\mu)\}$ is endowed with the subspace topology.
\end{co}

\begin{proof}
    We know that $\mathcal{M}_{+,f}(Y)$ is a Borel set. Define $F_k(\mu):= (f\wedge k) \mu $ for all $\mu \in \mathcal{M}_{+,f}(Y)$, $k\in \N$. Then $F_k$ pointwise converges to $F$; indeed, $(f\wedge k) \mu \rightharpoonup f\wedge \mu$ for all $\mu \in \mathcal{M}_{+,f}(Y)$.
\end{proof}

\begin{co}
    Assume $f:Y\to\overline{R}$ is Borel. Then the following map is Borel
    \[F:\mathcal{M}_{+,f}(Y) \to \mathcal{M}(Y), \quad F(\mu) = f\mu.\]
\end{co}
\begin{proof}
    Notice that for all $\mu \in \mathcal{M}_{+,f}(Y)$, $F(\mu) = F_+(\mu) - F_-(\mu)$, where $F_{\pm}(\mu) = f_{\pm}\mu$.
\end{proof}

\begin{co}\label{measurability vector measure f mu}
    Assume $f:Y\to\R^d$ is Borel. Then the following map is Borel
    \[F:\mathcal{M}_{+,f}(Y) \to \mathcal{M}(Y;\R^d), \quad F(\mu) = f\mu.\]
\end{co}

\begin{proof}
    Notice that $\mathcal{M}(Y;\R^d)$ is homeomorphic to $\mathcal{M}(Y)^{d}$.
\end{proof}

\begin{lemma}[Measurability of the condition $\nu=f\mu$]\label{measurability nu=f mu}
    Let $Y$ be a Polish space. For each Borel map $f:Y\to \R^d$ and any $p\geq 1$, the following set is Borel measurable:
    \[\mathfrak{D}_f:=\{(\mu,\nu) \in \mathcal{M}_+(Y)\times \mathcal{M}(Y;\R^d) \ : \ f\in L^p(\mu), \ \nu = f \mu\}.\]
\end{lemma} 

\begin{proof}
    The condition $f\in L^p(\mu)$ is a Borel condition thanks to the previous lemma. 
    \\
    Regarding the condition $\nu = f\mu$, thanks to the equivalence (i)$\iff$(ii) in  \cite[Theorem 8.10.39]{Bogachev07}, we have that $\mathcal{M}(Y,\R^d)$ is countably separated, i.e. there exists a countable set of continuous and bounded functions $\{h_n:Y \to \R\}$ such that for all $\nu_1,\nu_2\in \mathcal{M}(X,\R^d)$
    \[\int_X h_n \cdot d\nu_1 = \int_X h_n \cdot d\nu_2 \ \forall n\in \N \ \implies \ \nu_1 = \nu_2.\]
    Then, we can rewrite 
    \[\mathfrak{D}_f = \{(\mu,\nu) \in \mathcal{M}_+(Y)\times \mathcal{M}(Y;\R^d) \ : \ f\in L^p(\mu), \ \int_Y h_n\cdot d\nu = \int_Y h_n\cdot f \ d\mu \ \forall n\}.\]
    Notice that the function $\mathcal{M}(Y,\R^d)\ni\nu\mapsto \int h_n \cdot d\nu$ is continuous, while on the function $\mathcal{M}_+(X)\ni\mu \mapsto \int h_n\cdot f \ d\mu$ some comments must be done: written like this it is not well defined, because $f$ could be not bounded. Replace it with the function
    \[H_{n,f}(\mu) := 
    \begin{cases}
        \int (h_n\cdot f)_+ - (h_n \cdot f)_- \ d\mu\quad & \text{ if } f\in L^1(\mu)
        \\
        +\infty \quad \quad \quad & \text{ otherwise}
    \end{cases}\]
    This is a Borel function, because $\{\mu \ : \ f\in L^1(\mu)\}$ is a Borel set and thanks to the previous lemma the two functions $\int (h_n\cdot f)_\pm \ d\mu$ are Borel and we are done.
\end{proof}

\begin{prop}[Measurability of the condition $\nu\ll\mu$]
    Let $Y$ be Polish. Then 
    \[\{(\mu,\nu) \in \mathcal{M}_+(Y) \times \mathcal{M}_+(Y) \ : \nu\ll\mu\}\]
    is a Borel subset of $\mathcal{M}_+(Y)\times\mathcal{M}_+(Y)$.
\end{prop}

\begin{proof}
    The proof follows the same line of Lemma \ref{meas of AC^p}, with some more refinements. 
    \\
    Recall that the condition $\nu \ll \mu$ is equivalent to 
    \[\forall\varepsilon>0 \exists \delta>0 \ : \ \forall B\in\mathcal{B}(X), \ \mu(B)\leq \delta \implies \nu(B)\leq \varepsilon.\]
    Thanks to the outer regularity of any finite Borel measure, it is equivalent to require
    \[\forall\varepsilon>0 \exists \delta>0 \ : \ \forall A\subset X \text{ open}, \ \mu(B)\leq \delta \implies \nu(B)\leq \varepsilon.\]
    Define the functional
    \[\mathcal{F}_\delta(\nu | \mu) := \sup_{A\subset X \text{ open}, \mu(A)\leq \delta} \nu(A),\]
    and notice that, thank to what has already been said, 
    \[\nu \ll \mu \iff \inf_{n\in\N} \mathcal{F}_{\frac{1}{n}}(\nu|\mu) =0.\]
    Then, we are done if we prove that $(\mu,\nu) \mapsto \mathcal{F}_\delta(\nu|\mu)$ is Borel. We prove that, for all $\delta>0$, such function is actually l.s.c. in the couple $(\mu,\nu)$. Notice that such functional can be rewritten as
    \[\mathcal{F}_{\delta}(\nu|\mu) = \sup_{A\subset X \text{ open}} \left( \nu(A) + \chi_{(-\infty,\delta]}(\mu(A)) \right),\qquad \chi_I(x) = 
    \begin{cases}
        -\infty \quad &\text{ if } x\notin I
        \\
        0 \quad &\text{ if } x\in I
    \end{cases}\]
   
    Now notice that 
    \begin{itemize}
        \item $\nu\mapsto \nu(A)$ is l.s.c. for any $A\subset X$ open:
        \item the function $x \mapsto \chi_{(-\infty,\delta]}(x)$ is l.s.c. and non-decrasing. This, together with the lower semi-continuity of $\mu\mapsto \mu(A)$, implies that 
        $\mu\mapsto \chi_{(\infty,\delta]}(\mu(A))$ is lower semicontinuous.
    \end{itemize}
    Then, $\mathcal{F}_\delta(\nu|\mu)$ is the supremum of l.s.c. functionals, which implies that it is l.s.c. as well. In particular, it is a Borel map.
\end{proof}

\begin{co}\label{abs cont meas R^d}
    Given $X$ Polish, the set 
    \[\{(\mu,\nu) \in \mathcal{M}_+(X) \times \mathcal{M}(X;\R^d) \ : \nu\ll\mu\}\]
    is a Borel subset of $\mathcal{M}_+(X)\times\mathcal{M}(X;\R^d)$.
\end{co}

\begin{proof}
    The condition $\nu\ll\mu$ is equivalent to $|\nu|\ll\mu$, where $\nu$ is the total variation measure of $\nu$. We conclude thanks to the previous Lemma and the fact that $\mathcal{M}(X;\R^d)\ni\nu\mapsto |\nu|\in\mathcal{M}_+(X)$ is Borel (see Remark 2.4, \cite{AmIkLuPa24}).
\end{proof}

\printbibliography

{\small
		
		\vspace{15pt} (Alessandro Pinzi) Universit\`{a} Commerciale Luigi Bocconi, Dipartimento di Scienze delle Decisioni, \par
		\textsc{via Roentgen 1, 20136 Milano, Italy}
		\par
		\textit{e-mail address}: \textsf{alessandro.pinzi@phd.unibocconi.it}
		\par
		\textit{Orcid}: \textsf{https://orcid.org/0009-0007-9146-5434}
		\par

        \vspace{15pt} (Giuseppe Savar\'e) Universit\`{a} Commerciale Luigi Bocconi, Dipartimento di Scienze delle Decisioni and BIDSA, \par
		\textsc{via Roentgen 1, 20136 Milano, Italy}
		\par
		\textit{e-mail address}: \textsf{giuseppe.savare@unibocconi.it}
		\par
		\textit{Orcid}: \textsf{https://orcid.org/0000-0002-0104-4158}
		\par
		
	}

\end{document}